\documentclass{amsart}
\usepackage{bbm}
\usepackage{bbding}
\usepackage[mathcal]{euscript}
\usepackage{graphicx}
\usepackage{amsthm}
\usepackage{calc}
\usepackage{mathrsfs,dsfont}
\usepackage{CJK,fancyhdr,amscd}
\usepackage{anysize} %control the page
\usepackage{amsmath,amssymb,amsfonts}
\usepackage{epsfig,enumerate}
\usepackage{latexsym}
\usepackage{indentfirst, latexsym}
\usepackage{graphics}
\usepackage[all,poly,knot]{xy}
\usepackage[colorlinks,linkcolor=blue,anchorcolor=blue,citecolor=blue,pagebackref]{hyperref}
\usepackage{geometry}
\numberwithin{equation}{section}
\newtheorem{theorem} {Theorem} [section]
\newtheorem{proposition}[theorem]{Proposition}
\newtheorem{corollary}  [theorem]     {Corollary}
\newtheorem{lemma}  [theorem]     {Lemma}
\newtheorem{question}  [theorem]     {Question}

\newtheorem{remark}  [theorem]     {Remark}
\newtheorem{claim}  [theorem]     {Claim}

\theoremstyle{definition}
\newtheorem{definition}  [theorem]     {Definition}

\renewcommand*\backref[1]{}
\renewcommand*\backrefalt[4]{ \ifcase #1 \or (cited on page #2) \else (cited on pages #2) \fi}%(no citations)

\begin{document}

\title{On Hermitian manifolds with Bismut-Strominger parallel torsion}

\begin{abstract}
In this article, we study Hermitian manifolds whose Bismut-Strominger connection has parallel torsion tensor, which will be called {\em Bismut torsion parallel manifolds,} or {\em BTP} manifolds for short. We obtain a necessary and sufficient condition characterizing this class in terms of the curvature tensor. In particular, Bismut flat or Bismut K\"ahler-like manifolds are {\em BTP} manifolds, known by our earlier results. We also obtain structural results for non-balanced {\em BTP} manifolds, and classification theorems for non-balanced {\em BTP} threefolds and balanced {\em BTP} threefolds.
\end{abstract}

\author{Quanting Zhao}
\address{Quanting Zhao. School of Mathematics and Statistics, and Hubei Key Laboratory of Mathematical Sciences, Central China Normal University, P.O. Box 71010, Wuhan 430079, P. R. China.} \email{zhaoquanting@126.com;zhaoquanting@mail.ccnu.edu.cn}
\thanks{Zhao is partially supported by NSFC with the grant No.11801205 and 12171180. Zheng is partially supported by NSFC grant 12071050 and 12141101, a Chongqing grant cstc2021ycjh-bgzxm0139, and the 111 Project D21024.}

\author{Fangyang Zheng}
\address{Fangyang Zheng. School of Mathematical Sciences, Chongqing Normal University, Chongqing 401331, China}
\email{20190045@cqnu.edu.cn} \thanks{}

%\date{\today}

\subjclass[2010]{ 53C55 (Primary), 53C05 (Secondary)}
\keywords{Bismut torsion parallel, K\"ahler-like, Bismut connection, admissible frame, degenerate torsion,
Chern flat, Fano manifold, Vaisman manifold}

\maketitle

\tableofcontents

\section{Introduction}

Given a smooth manifold equipped with a connection, the parallelness of torsion often forms strong geometric restrictions. For instance, by the  result of Kamber and Tondeur in \cite{KT}, if a manifold admits a complete connection that is flat and having parallel torsion, then its universal cover must be a Lie group (equipped with a flat left invariant connection). The converse is certainly true. Also, by the classic result of Ambrose and Singer \cite{AS}, if a complete Riemannian manifold admits a metric connection with parallel torsion and curvature, then the universal cover is a homogeneous Riemannian manifold (and vice versa). The complex version of this is also true, proved by Sekigawa in \cite{Sekigawa}, namely, if a complete Hermitian manifold admits a {\em Hermitian} connection (that is, a connection compatible with both the metric and the almost complex structure) which has parallel torsion and curvature, then the universal cover is a homogeneous Hermitian manifold (and vice versa).

On a given Hermitian manifold, Bismut connection (also known as Strominger connection, see \cite{Bismut,Strominger}) is the unique Hermitian connection with totally skew-symmetric torsion. It serves as a bridge between Levi-Civita connection, which is the unique metric connection that is torsion free,  and Chern connection, the unique metric connection compatible with the complex structure. When the metric is K\"ahler, these three canonical connections coincide, but when the metric is not K\"ahler, they are mutually distinct, leaving us with three different kinds of geometry.

In this article, we will focus on Bismut connection, and we want to know when Bismut connection will have parallel torsion. For convenience, we will call the class of Hermitian manifolds whose Bismut connection has parallel torsion simply as {\em Bismut torsion parallel} manifolds, or {\em BTP} manifolds in short.

Let $(M^n,g)$ be a Hermitian manifold. Denote by $\nabla^b$ the Bismut connection and by $T^b$, $R^b$ its torsion and curvature tensor, respectively. Let us introduce the notation
\begin{equation*}
Q_{X\overline{Y}Z\overline{W}} = R^b_{X\overline{Y}Z\overline{W}} - R^b_{Z\overline{Y}X\overline{W}}, \ \ \ \mathrm{Ric}(Q)_{X\overline{Y}} = \sum_{i=1}^n \big( R^b_{X\overline{Y}e_i\overline{e_i}} - R^b_{e_i\overline{Y}X\overline{e_i}} \big)
\end{equation*}
where $X$, $Y$, $Z$, $W$ are type $(1,0)$ tangent vectors on $M$ and $\{ e_1, \ldots , e_n\}$ is any local unitary frame. Denote by $\omega$ the K\"ahler form of $g$. Recall that Gauduchon's torsion $1$-form is the global $(1,0)$-form $\eta$ on $M$ defined by $\partial (\omega^{n-1}) = - 2\eta \wedge \omega^{n-1}$. Let $\chi$ be the type $(1,0)$ vector field on $M$ dual to $\eta$, namely, $\langle Y, \overline{\chi}\rangle = \eta (Y)$ for any $Y$ where $g = \langle , \rangle $ is the (Riemannian) metric, extended bi-linearly over ${\mathbb C}$.

With these notations specified, we can now state the first theorem of this article, which says that the {\em BTP} condition $\nabla^b T^b=0$ is equivalent to some conditions involving only $R^b$:

\begin{theorem} \label{theorem1.1}
Let $(M^n,g)$ be a Hermitian manifold. The BTP condition $\nabla^bT^b=0$ is equivalent to the combination of the following
\begin{eqnarray}
&& R^b_{XYZ\overline{W}}=0 \label{eq:1.1} \\
&& R^b_{X\overline{Y}Z\overline{W}}=R^b_{Z\overline{W}X\overline{Y}} \label{eq:1.2} \\
&& \nabla^b \mathrm{Ric}(Q)=0 \label{eq:1.3}\\
&& \mathrm{Ric}(Q)_{\chi \overline{Y}} =0  \label{eq:1.4}
\end{eqnarray}
for any type $(1,0)$ tangent vectors $X$, $Y$, $Z$, $W$.
\end{theorem}

Recall that the {\em  Bismut K\"ahler-like} ({\em BKL}) condition, also known as {\em Strominger K\"ahler-like} ({\em SKL}) condition (cf. \cite{AOUV,FT,FTV,YZZ,ZhaoZ19Str}), for a Hermitian manifold is defined as
\begin{equation*}
 R^b_{XYZ\overline{W}}=0, \ \ \ \  R^b_{X\overline{Y}Z\overline{W}}=R^b_{Z\overline{Y}X\overline{W}}
\end{equation*}
for any type $(1,0)$ tangent vectors $X$, $Y$, $Z$, $W$. The main result of \cite{ZhaoZ19Str} says that
$$ \mbox{{\em BKL}} \ \Longleftrightarrow \ \mbox{{\em BTP}} + \mbox{pluriclosed} .$$
The main technical part of the proof in \cite{ZhaoZ19Str} is to show that {\em BKL} implies {\em BTP}. Note that by our Theorem \ref{theorem1.1} above, this implication is immediate, since the {\em BKL} condition simply means (\ref{eq:1.1}) and $Q=0$. In this case, all four conditions in Theorem \ref{theorem1.1} are obviously satisfied, hence {\em BTP}. So Theorem \ref{theorem1.1} can be regarded as a generalization of the main result of \cite{ZhaoZ19Str}.

We note that in Theorem \ref{theorem1.1}, the conditions \eqref{eq:1.3} and \eqref{eq:1.4} are automatically satisfied if $\mathrm{Ric}(Q)=0$.  Also, under the {\em BTP} assumption, one always has $\nabla^bQ=0$. But there are examples of non-K\"ahler {\em BKL} manifolds with $\nabla^b R^b\neq 0$. In other words, for {\em BTP} manifolds, the Bismut connection $\nabla^b$ is not always an Ambrose-Singer connection.

Recall that a Hermitian manifold is said to be {\em Gauduchon} if $\partial \overline{\partial} (\omega^{n-1})=0$, and is said to be {\em balanced} if its Gauduchon's torsion $1$-form $\eta$ vanishes, or equivalently, if $d(\omega^{n-1})=0$. Also, the metric is said to be {\em strongly Gauduchon} in the sense of Popovici \cite{Pop} if there exists a global $(n,n-2)$-form $\Omega$ such that $\partial \omega^{n-1} = \overline{\partial} \Omega$. Clearly,
$$ \mbox{balanced} \ \ \Longrightarrow \ \ \mbox{strongly\ Gauduchon} \ \ \Longrightarrow \ \ \mbox{Gauduchon}.$$
As a consequence to Theorem \ref{theorem1.1}, we have

\begin{corollary} \label{corollary1.2}
Let $(M^n,g)$ be a Hermitian manifold satisfying $\nabla^bT^b=0$. Then $g$ is Gauduchon, i.e., $\partial \overline{\partial }(\omega^{n-1})=0$. If in addition $M^n$ is compact and $g$ is not balanced, then $g$ is not strongly Gauduchon and the Dolbeault group $H^{0,1}_{\overline{\partial}}(M) \neq 0$. In particular, any compact non-balanced BTP manifold is not a $\partial \overline{\partial}$-manifold and $M$ does not admit any K\"ahler metric.
\end{corollary}

It has been shown in \cite[Theorem 7 and Proposition 4]{ZhaoZ19Str} that compact non-K\"ahler {\em BKL} manifolds, or equivalently non-balanced {\em BKL} manifolds, does not admit any strongly Gauduchon metric and the Dolbeault group $H^{0,1}_{\overline{\partial}}(M)$ is non-trivial. We speculate that the same should be true for compact non-balanced {\em BTP} manifolds as well:

\begin{question} \label{conjecture1.3}
Let $(M^n,g)$ be a compact non-balanced Hermitian manifold satisfying $\nabla^b T^b=0$. Then $M$ does not admit any strongly Gauduchon metric.
\end{question}

By \cite[Theorem 2]{ZhaoZ19Str}, the following equivalence holds for $n=2$:
$$ \mbox{{\em BKL}} \quad  \Longleftrightarrow \ \quad  \mbox{{\em BTP}}  \quad \Longleftrightarrow \quad  \mathrm{Vaisman} $$
where Vaisman means a locally conformally K\"ahler manifold whose Lee form (which is $-2(\eta +\overline{\eta})$) is parallel under the Levi-Civita connection. Compact Vaisman surfaces are fully classified by Belgun \cite{B00} and a structure theorem for Vaisman manifolds is proved by Ornea and Verbitsky \cite{OV}.

In dimension $n\geq 3$, locally conformally K\"ahler manifolds can no longer be (non-K\"ahler) {\em BKL}, shown in \cite{YZZ}, but still can  be {\em BTP}. For instance, it is easy to check that for any $n$ the classic Hopf manifold
$$ M^n=({\mathbb C}^n\setminus \{0\})/ \langle f \rangle , \ \ \ f(z)=\lambda z, \ \ \ |\lambda |>1, \ \ \ \ \omega = \sqrt{-1}\frac{\partial \overline{\partial} |z|^2}{|z|^2} ,$$
is {\em BTP.}\, More generally, all Vaisman manifolds are {\em BTP}. In fact, a locally conformally K\"ahler manifold is {\em BTP} if and only if it is Vaisman, a result due to Andrada and Villacampa \cite[Theorem 3.6]{AndV}. Here we give a slightly more general statement:

\begin{proposition} \label{prop1.4}
A locally conformally K\"ahler manifold $(M^n,g)$ is BTP if and only if it is Vaisman. Furthermore, if two conformal Hermitian manifolds  $(M^n,g)$ and $(M^n,e^{2u}g)$ are both BTP, with $du\neq 0$ in an open dense subset of $M$, then both manifolds are locally conformally K\"ahler (thus Vaisman).
\end{proposition}

In other words, for each conformal class of Hermitian metrics, there is at most one {\em BTP} metric (up to constant multiples) unless the metric is locally conformally K\"ahler. When the manifold is compact, this is clearly true since {\em BTP} metrics are Gauduchon.

Next we consider {\em BTP} manifolds amongst the class of complex nilmanifolds with nilpotent $J$. Complex nilmanifolds form an important class of Hermitian manifolds, and are often used to test and illustrate theory in non-K\"ahler geometry. See \cite{CFU}, \cite{GZZ}, \cite{LZ}, \cite{Sal}, \cite{Uga}, \cite{VYZ}, \cite{LS} as a sample of examples. Following the discussions  of \cite{ZhaoZ19Nil}, we have

\begin{proposition} \label{prop1.5}
Let $(M^n,g)$ be a complex nilmanifold, namely, a compact Hermitian manifold whose universal cover is a nilpotent Lie group $G$ equipped with a left-invariant complex structure $J$ and a compatible left-invariant metric $g$. Assume that $J$ is nilpotent in the sense of \cite{CFGU}. Then $g$ is BTP if and only if there exists a unitary left-invariant coframe $\varphi$ on $G$ and an integer $1\leq r\leq n$ such that
\begin{equation*}
\left\{ \begin{array}{ll} d\varphi_i=0, \ \ \ \ \ \  \ \ \ \ \ \  \ \ \ \ \ \ \ \ \   \forall \ 1\leq i\leq r;\  \\ d\varphi_{\alpha} = \sum_{i=1}^r Y_{\alpha i}\,\varphi_i \wedge \overline{\varphi}_i, \ \ \forall \ r<\alpha \leq n.   \end{array} \right.
\end{equation*}
Note that the metric $g$ will be balanced when and only when $\sum_{i=1}^r Y_{\alpha i}=0$ for each $\alpha$.
\end{proposition}

In particular, $G$  (when not abelian) is a 2-step nilpotent group and $J$ is abelian. Let us take the $n=3$ case as an example. In this case, either $r=3$, where $G$ is abelian and $g$ is K\"ahler and flat, or $r=2$ and there exists a unitary coframe $\varphi$ such that
$$ d\varphi_1=d\varphi_2=0, \ \ d\varphi_3 = a\varphi_1\wedge \overline{\varphi}_1 + b \varphi_2\wedge \overline{\varphi}_2, $$
where $a$, $b$ are constants. By a unitary change if necessary, we may assume that $a>0$. Under such a coframe, the metric $g$ is balanced if and only if $b=-a$, while $g$ is {\em BKL} if and only if $b\in \sqrt{-1}{\mathbb R}$ is pure imaginary. For any other choice of $b$ values, $g$ is a non-{\em BKL}, non-balanced {\em BTP} metric. We will denote the following example by $N^3$,
\begin{equation}     \label{N_3}
d\varphi_1=d\varphi_2=0, \ \ \ \ \ d\varphi_3 = \varphi_1 \wedge \overline{\varphi}_1 - \varphi_2 \wedge \overline{\varphi}_2,
\end{equation}
which will appear in the study of Lie-Hermitian threefolds that are balanced {\em BTP}, in \S \ref{LH}.

The examples above show that when $n\geq 3$, there are more {\em BTP} manifolds than {\em BKL} manifolds, and there are more non-balanced {\em BTP} manifolds than balanced ones.  For complex nilmanifolds with $J$ not nilpotent, when $n=3$ one can show  by a lengthy computation that they cannot be {\em BTP}, but for $n\geq 4$, we do not know if they can be {\em BTP} or not. More generally, it would be a very interesting question on how to classify all {\em BTP} Lie-Hermitian manifolds (see for example \cite{VYZ} for a discussion on Lie-Hermitian manifolds).

For $n\geq 3$, one can naturally divide the discussion of (non-K\"ahler) {\em BTP} manifolds into two cases: non-balanced ones and balanced ones. All (non-K\"ahler) {\em BKL} manifolds and all (non-K\"ahler) Vaisman manifolds belong to the first subset,
where the $\nabla^b$-parallel vector field $\chi$ (dual to $\eta$) will play an important role. In \S \ref{NBBTP} we will discuss some general properties and structural results of such spaces following the ideas of \cite{YZZ}.

\begin{proposition} \label{prop1.6}
Let $(M^n,g)$ be a non-balanced BTP manifold with $n\geq 3$. Then the following holds:
\begin{enumerate}
\item\label{hlnm} Locally there always exist admissible frames, and the holonomy group of the Bismut connection is contained in $U(n-1)$.

\item The metric $g$ satisfies the LP condition if and only if it has degenerate torsion. In particular, for $n=3$, non-balanced BTP threefolds coincide with GCE threefolds of Belgun.
\end{enumerate}
\end{proposition}

For the special case when $(M^n,g)$ is Vaisman, the conclusion \eqref{hlnm} is due to Andrada and Villacampa \cite{AndV}. Recall that a local unitary frame $e$ on a non-balanced {\em BTP} manifold is said to be {\em admissible,} if $\chi =\lambda e_n$ and $\phi^j_{i}= \lambda a_i \delta_{ij}$ for any $i$, $j$, where $\lambda >0$ and $a_i$ are globally defined constants, as in Definition \ref{adm}. Here the tensor $\phi$ is defined by $\phi_i^j = \sum_{k} \overline{\eta}_k T_{ik}^j$, where $\eta_k$ and $T_{ik}^j$ are components of $\eta$ and the Chern torsion $T$ under local unitary frame. The metric $g$ is said to have {\em degenerate torsion} if under any admissible frame $e$ it holds $T^{\ast}_{ik}=0$ for any $i<k<n$, as given in Definition \ref{degtor_def}. The {\em Lee potential} ({\em LP}) condition of Belgun \cite{B12} means that a Hermitian manifold satisfying the condition
\[\partial \eta =0 , \quad \partial \omega = c \,\eta \wedge\partial \overline{\eta},\]
where $c\neq 0$ is a constant, and a Hermitian manifold is {\em generalized Calabi-Eckmann} ({\em GCE}\,) if it is both {\em BTP} and {\em LP}.

The main result of \S \ref{NBBTP} is the following classification theorem of non-balanced {\em BTP} threefolds.
\begin{theorem}\label{3DNBBTP}
Let $(M^3,g)$ be a non-balanced BTP threefold. If we denote by $\{ \lambda_1, \lambda_2,0\}$ the three eigenvalues of the tensor $\phi$, it yields that
\begin{enumerate}
\item\label{twSSK} when $\lambda_1\overline{\lambda_2}+\overline{\lambda_1}\lambda_2 \neq 0$ and $\lambda_1\overline{\lambda_2}- \overline{\lambda_1}\lambda_2 \neq 0$, $(M^3,g)$ is locally twisted Sasakian product $N_1 \times_{\kappa} N_2$,
\item\label{BKL} when $\lambda_1\overline{\lambda_2}+\overline{\lambda_1}\lambda_2 = 0$, $(M^3,g)$ is BKL,
\item\label{gV} when $\lambda_1\overline{\lambda_2}+\overline{\lambda_1}\lambda_2 \neq 0$ and $\lambda_1\overline{\lambda_2}-\overline{\lambda_1}\lambda_2 = 0$, $(M^3,g)$ is generalized Vaisman. In particular, $(M^3,g)$ is Vaisman when $\lambda_1=\lambda_2$.
\end{enumerate}
\end{theorem}
The {\em twisted Sasakian product} here is actually the modified Hermitian structure on the Sasakian product, introduced by Belgun \cite{B12}, which we will see in Definition \ref{SSK} and \ref{SSKproduct}. A Hermitian manifold $(M^n,g)$ is said to be {\em generalized Vaisman} if it is {\em BTP} and the Lee form, defined as $-2(\eta+\overline{\eta})$ in this paper, is $d$-closed. It follows from Proposition \ref{prop1.4} that a Vaisman manifold is necessarily generalized Vaisman.

The balanced {\em BTP} manifolds, on the other hand, form a rather restrictive but non-empty subset. In the following, we will discuss the classification problem for such manifolds in dimension 3, starting from a technical lemma obtained in \cite{ZhouZ} for special frames on balanced threefolds.  Recall that the $B$ tensor of a Hermitian manifold $(M^n,g)$ is defined by $B_{i\overline{j}}=\sum_{r,s} T^j_{rs}\overline{T^i_{rs}}$ under any local unitary frame, where $T^j_{ik}$ are components of the Chern torsion under the frame.

\begin{proposition}\label{Btype}
Let $(M^3,g)$ be a non-K\"ahler balanced BTP threefold. Then, for any $p\in M$, there exists a local unitary frame near $p$ such that
\[B = \begin{bmatrix}  c & 0  & 0 \\ 0 & 0 &  0  \\ 0 & 0 & 0 \end{bmatrix}\!\!,\
\begin{bmatrix}  c & 0  & 0 \\ 0 & c &  0  \\ 0 & 0 & 0 \end{bmatrix}\  \text{or} \
\begin{bmatrix}  c & 0  & 0 \\ 0 & c &  0  \\ 0 & 0 & c \end{bmatrix}\!\!,\]
for some constant $c>0$.
\end{proposition}

Such a local unitary frame will be called \emph{a special frame} in \S \ref{BS3D}. We will carry out the classification of non-K\"ahler balanced {\em BTP} threefolds according to the rank of the tensor $B$.

\begin{definition} \label{mddtype}
A balanced {\em BTP} threefold $(M^3,g)$ is said to be of {\em middle type,} if $\mathrm{rank}\,B=2 $, or equivalently, for any $p\in M$, under a special frame near $p$
$$ B =  \begin{bmatrix} c & 0  & 0 \\ 0 & c &  0  \\ 0 & 0 & 0 \end{bmatrix}    $$
for some constant $c>0$.
\end{definition}

We would like to illustrate two non-middle type examples. First let us consider the complex Lie group $SO(3, {\mathbb C})$ consisting of all $3\times 3$ complex matrices $A$ such that $\det (A)=1$ and $\,^t\!A A=I$. Let $\varphi$ be a left-invariant coframe satisfying
\begin{equation}
d\varphi_1 =  \varphi_2 \wedge \varphi_3, \ \ \ d\varphi_2 =  \varphi_3 \wedge \varphi_1, \ \ \ d\varphi_3 =  \varphi_1 \wedge \varphi_2 . \label{eq:2}
\end{equation}
Let $g$ be the Hermitian metric with $\varphi$ being a unitary coframe. Then it is easy to check that $g$ is balanced {\em BTP} and its $B$ tensor takes the form
\[B = \frac{1}{2}\!\!\begin{bmatrix}  1 & 0  & 0 \\ 0 & 1 &  0  \\ 0 & 0 & 1 \end{bmatrix}\!\!.\]
So any compact quotient of $SO(3, {\mathbb C})$ is a non-K\"ahler, balanced {\em BTP} threefold of non-middle type. Note that this example has trivial canonical line bundle and is Chern flat. We remark that for this Hermitian threefold, the holonomy group of $\nabla^b$ is reduced to $SO(3)\subseteq U(3)$, but not contained in $U(2)\times 1$  like in the non-balanced {\em BTP} case.

There is another example of balanced {\em BTP} threefold which is not of the middle type: the {\em Wallach threefold} $(X,g)$. As a complex manifold, $X$ is the flag threefold $X={\mathbb P}(T_{{\mathbb P}^2} )$. Let $\tilde{\omega}$ be the K\"ahler-Einstein metric on $X$ with $\mbox{Ric}(\tilde{\omega})=2\tilde{\omega}$. Then the K\"ahler form of $g$ is given by $\omega = \tilde{\omega} - \sigma$, where $\sigma$ is the $(1,1)$-form on $X$ given by a holomorphic section $\xi \in H^0(X, \Omega_X\!\otimes\!L) \cong {\mathbb C}$ with $\parallel \! \xi \! \parallel^2_{\tilde{\omega}}=\frac{1}{2}$. Here $\Omega_X$ is the holomorphic cotangent bundle, and $L$ is the ample line bundle such that $L^{\otimes 2}$ is the anti-canonical line bundle of $X$.

We will see in \S \ref{WCH3D} that $(X,g)$ is non-K\"ahler, balanced {\em BTP}. We will also show that its Levi-Civita (Riemannian) connection has constant Ricci curvature $6$ and has non-negative sectional curvature. %, so the metric $g$ lies in the boundary of the family of metrics with positive sectional curvature discovered by Wallach in \cite{Wallach}.
Its $B$ tensor takes the form \[ B=\frac{1}{2}\!\! \begin{bmatrix} 1 & 0  & 0 \\ 0 & 0 &  0  \\ 0 & 0 & 0 \end{bmatrix}\!\!,\] so $(X,g)$ is not of middle type. This example is Fano, namely, its anti-canonical line bundle is ample. The Chern connection of $g$ has nonnegative bisectional curvature and positive holomorphic sectional curvature.

As observed by Wallach in \cite{Wallach}, all the homogeneous metrics with positive sectional curvature on $X$ are Hermitian. This set of metrics (after scaling) forms a peculiar plane region (see Figure 1 in \cite{BM}), and our metric $g$ lies in the boundary of this region.

We have the following classification theorem of balanced {\em BTP} threefolds, which is the second main theorem of this article. The proofs of the three cases below will be given in \S \ref{BS3D}, \S \ref{FANO} and \S \ref{str}.
\begin{theorem} \label{classification}
Let $(M^3,g)$ be a compact non-K\"ahler, balanced BTP threefold. After scaling the metric by a constant multiple, it holds that,
\begin{enumerate}
\item when $\mbox{rank}\,B=1$, $M$ is holomorphically isometric to the Wallach threefold.
\item when $\mbox{rank}\,B=2$, namely the middle type, then there is an unbranched cover $\pi : \hat{M} \rightarrow M$ of degree at most $2$ on which there is a bi-Hermitian structure $(\hat{M},g,J,I)$, where $(\hat{M},g,J)$ is balanced BTP while $(\hat{M},g,I)$ is Vaisman, with $IJ=JI$ and the Bismut connections of the two Hermitian structures $(\hat{M},g,J)$ and $(\hat{M},g,I)$ coincide. In particular, $\hat{M}$ is a smooth fiber bundle over $S^1$ with fiber being a Sasakian 5-manifold. Furthermore, a compact non-K\"ahler, balanced BTP threefold of middle type admits no pluriclosed metric.
\item when $\mbox{rank}\,B=3$, $M$ is holomorphically isometric to a compact quotient of $SO(3, {\mathbb C})$ with global unitary coframe satisfying (\ref{eq:2}).
\end{enumerate}
\end{theorem}
In \S \ref{LH}, we will study balanced {\em BTP} threefolds of middle type in the category of Lie-Hermitian manifolds and prove the following theorem
\begin{theorem}\label{LH_BBTP}
Let $(M^3,g)$ be a non-K\"ahler, balanced BTP Lie-Hermitian threefold of middle type, then after a scaling of the metric by a constant multiple, it is one of the members in $A_{a,b}$, $B_{u,v,w}^{\pm}$, $C_{u,v}$ or $D_{u,\rho}^{\pm}$.
\end{theorem}

Here the constant parameters $a,b\in {\mathbb R}$, $u,v,w,\rho \in {\mathbb C}$ with $|\rho |=1$, and for $B^{\varepsilon}_{u,v,w}$ the parameters also satisfies $w-\overline{w}= \varepsilon i (|u-v|^2 - |u+v|^2)$, where $\varepsilon =\pm 1$. See \S \ref{LH} for the explicit descriptions of these Lie-Hermitian threefolds.

Finally, we remark that there are examples of complete non-compact Hermitian surfaces with zero Chern curvature but non-parallel Chern torsion, so the curvature characterization of torsion parallelness for Bismut connection does not hold for Chern connection or other Hermitian connections in general. Of course it would be a much harder problem if one adds the compactness assumption on the manifolds, then things get more delicate and we plan to investigate it in future projects.

The paper is organized as follows. In \S \ref{PRE}, we will set up notations and also collect some known results that will be used later. In \S \ref{TC} we will discuss the properties of {\em BTP} manifolds, and in \S \ref{BTP}, we will prove Theorem \ref{theorem1.1} stated above. In \S \ref{NBBTP}, we will discuss the properties of non-balanced {\em BTP} manifolds, following ideas of \cite{YZZ} where non-K\"ahler {\em BKL} manifolds were investigated, and will prove Proposition \ref{prop1.6}. In \S \ref{BS3D}, we will discuss the types of the $B$ tensor of a balanced {\em BTP} threefold and prove Proposition \ref{Btype}, where we will also show that the $\mathrm{rank}\,B=3$ case of Theorem \ref{classification} leads to the compact quotients of $SO(3,{\mathbb C})$. In \S \ref{FANO}, we will prove that the $\mathrm{rank}\,B=1$ case of Theorem \ref{classification} leads to a Fano threefold of even index, so as a complex manifold it is either ${\mathbb P}^3$ or a {\em  del Pezzo threefold.} All cases will be ruled out by topological or algebro-geometric argument, except the flag threefold case: the Wallach threefold $(X,g)$. In \S \ref{WCH3D}, we will carry out the detailed computation which shows that the Wallach threefold is indeed balanced {\em BTP}, with nonnegative bisectional curvature and positive holomorphic sectional curvature. We will also show that its Levi-Civita connection has non-negative sectional curvature and constant Ricci curvature $6$. In \S \ref{mddtype3D}, we will prove the case of $\mbox{rank}\,B=2$ of Theorem \ref{classification} and investigate balanced {\em BTP} Lie-Hermitian threefolds of middle type, as well as discussing some generalization in higher dimensions.

\vspace{0.3cm}

\section{Preliminaries}\label{PRE}

Let $(M^n,g)$ be a Hermitian manifold of complex dimension $n$. We will set up notations following  \cite{YZ18Cur,YZ18Gau,ZhaoZ19Str, Zheng, Zheng1}. Let $g = \langle , \rangle $ be the metric, extended bi-linearly over ${\mathbb C}$. The bundle of complex tangent vector fields of type $(1,0)$, namely, complex vector fields of the form $v- \sqrt{-1}Jv$, where $v$ is a real vector field on $M$, is denoted by $T^{1,0}M$. Let $\{ e_1, \ldots , e_n\}$ be a local frame of $T^{1,0}M$ in a neighborhood in $M$. Write $e=\ ^t\!(e_1, \ldots , e_n) $
 as a column vector. Denote by $\varphi = \ ^t\!(\varphi_1, \ldots ,
  \varphi_n)$ the column vector of local $(1,0)$-forms which is the coframe dual to $e$.

Denote by $\nabla$, $\nabla^c$, $\nabla^b$ the Levi-Civita, Chern, and Bismut connection, respectively. Denote by $T^c=T$, $R^c$ the torsion and curvature of $\nabla^c$, and by $T^b$, $R^b$ the torsion and curvature of $\nabla^b$. Under the frame $e$, denote the components of $T^c$ as
\begin{equation}
T^c(e_i, \overline{e}_j)=0, \ \ \ T^c(e_i, e_j) = 2 \sum_{k=1}^n T^k_{ij} e_k.  \label{eq:2.1}
\end{equation}

For Chern connection $\nabla^c$, let us denote by $\theta$,  $\Theta$ the matrices of connection and
curvature, respectively, and by $\tau$ the column vector of the
torsion $2$-forms, all under the local frame $e$. Then the structure
equations and Bianchi identities are
\begin{eqnarray}
d \varphi & = & - \ ^t\!\theta \wedge \varphi + \tau,  \label{formula 1}\\
d  \theta & = & \theta \wedge \theta + \Theta. \\
d \tau & = & - \ ^t\!\theta \wedge \tau + \ ^t\!\Theta \wedge \varphi, \label{formula 3} \\
d  \Theta & = & \theta \wedge \Theta - \Theta \wedge \theta.
\end{eqnarray}
The entries of $\Theta$ are all $(1,1)$ forms, while the entries of the column vector $\tau $ are all $(2,0)$ forms, under any frame $e$.
Similar symbols such as $\theta^b,\Theta^b$ and $\tau^b$ are applied to Bismut connection. The components of $\tau$ are just $T_{ij}^k$ defined in (\ref{eq:2.1}):
\[ \tau_k = \sum_{i,j=1}^n T_{ij}^k \varphi_i\wedge \varphi_j \ = \sum_{1\leq i<j\leq n} 2 \ T_{ij}^k \varphi_i\wedge \varphi_j,\]
where $T_{ij}^k=-T_{ji}^k$. Under the frame $e$, express the Levi-Civita (Riemannian) connection $\nabla$ as
$$ \nabla e = \theta_1 e + \overline{\theta_2 }\overline{e} ,
\ \ \ \nabla \overline{e} = \theta_2 e + \overline{\theta_1
}\overline{e} ,$$
thus the matrices of connection and curvature for $\nabla $ become:
$$ \hat{\theta } = \left[ \begin{array}{ll} \theta_1 & \overline{\theta_2 } \\ \theta_2 & \overline{\theta_1 }  \end{array} \right] , \ \  \  \hat{\Theta } = \left[ \begin{array}{ll} \Theta_1 & \overline{\Theta}_2  \\ \Theta_2 & \overline{\Theta}_1   \end{array} \right], $$
 where
\begin{eqnarray*}
\Theta_1 & = & d\theta_1 -\theta_1 \wedge \theta_1 -\overline{\theta_2} \wedge \theta_2, \\
\Theta_2 & = & d\theta_2 - \theta_2 \wedge \theta_1 - \overline{\theta_1 } \wedge \theta_2,  \label{formula 7}\\
d\varphi & = & - \ ^t\! \theta_1 \wedge \varphi - \ ^t\! \theta_2
\wedge \overline{\varphi } .
\end{eqnarray*}
When $e$ is unitary, both $\theta_2 $ and $\Theta_2$ are skew-symmetric, while $\theta$, $\theta_1$, $\theta^b$, or $\Theta$, $\Theta_1$, $\Theta^b$ are all skew-Hermitian. Consider the $(2,1)$ tensor $\gamma =\frac{1}{2}(\nabla^b -\nabla^c)$ introduced in \cite{YZ18Cur}. Its representation under the frame $e$ is the matrix of $1$-forms, which by abuse of notation we will also denote by $\gamma$, is given by
\[\gamma = \theta_1 - \theta .\]
Denote the decomposition of $\gamma$ into $(1,0)$ and $(0,1)$ parts by $\gamma = \gamma ' + \gamma ''$. As in \cite[Lemma 2]{WYZ}, it yields that
\[ \theta^b = \theta + 2\gamma = \theta_1 +\gamma .\]
As observed in \cite{YZ18Cur}, when $e$ is unitary, $\gamma $ and $\theta_2$ take the following simple forms
\begin{equation}
(\theta_2)_{ij} = \sum_{k=1}^n \overline{T^k_{ij}} \varphi_k, \ \ \ \ \gamma_{ij} = \sum_{k=1}^n ( T_{ik}^j \varphi_k - \overline{T^i_{jk}} \overline{\varphi}_k ), \label{eq:2.6}
\end{equation}
while for general frames the above formula will have the matrix $(g_{i\overline{j}})=(\langle e_i,\overline{e}_j \rangle )$ and its inverse involved, which is less convenient. This is why we often choose unitary frames to work with. By (\ref{eq:2.6}), we get the expression of the components of the torsion $T^b$ under any unitary frame $e$:
\begin{equation}
T^b(e_i,e_j) = -2 \sum_{k=1}^n T^k_{ij}e_k, \ \ \ T^b(e_i, \overline{e}_j) = 2\sum_{k=1}^n \big( T^j_{ik}\overline{e}_k - \overline{T^i_{jk}} e_k \big) . \label{eq:2.7}
\end{equation}
Note that in the second equation above we assumed that the frame $e$ is unitary, otherwise the matrix $g$ needs to appear in the formula. From this, we deduce that
\begin{equation}
\nabla^bT^b=0 \ \Longleftrightarrow \nabla^bT^c =0. \ \label{eq:2.8}
\end{equation}
This is observed  in \cite[the proof of Theorem 3]{ZhaoZ19Str}. More generally, if we denote by $\nabla^t = (1-t)\nabla^c + t \nabla^b$ the $t$-Gauduchon connection, where $t$ is any real constant, then its torsion tensor is $T^t(x,y)=T^c(x,y)+ 2t(\gamma_xy -\gamma_yx)$, whose components under $e$ are again given by $T^k_{ij}$, so it is easy to see that for any fixed $t\in {\mathbb R}$
\begin{equation}
\nabla^tT^{t_1}=0 \ \Longleftrightarrow \nabla^tT^{t_2} =0 \ \ \ \ \forall \ \ t_1, t_2 \in \mathbb{R}. \ \label{eq:2.9}
\end{equation}
So for a fixed $t$-Gauduchon connection $\nabla^t$, if the torsion of one Gauduchon connection is parallel with respect to $\nabla^t$, then the torsion of any other Gauduchon connection will be parallel under $\nabla^t$ as well.

Next let us discuss the curvature. As usual, the curvature tensor $R^D$ of a linear connection $D$ on a Riemannian manifold $M$ is defined  by
$$ R^D(x,y,z,w) = \langle R^D_{xy}z, \ w \rangle = \langle D_xD_yz-D_yD_xz - D_{[x,y]}z, \ w \rangle, $$
where $x,y,z,w$ are tangent vectors in $M$. We will also write it as $R^D_{xyzw}$ for brevity. It is always skew-symmetric with respect to the first two positions, and is also skew-symmetric with respect to its last two positions if the connection is metric, namely, if $Dg=0$.  The first and second Bianchi identities are respectively
\begin{eqnarray*}
&& {\mathfrak S}\{ R^D_{xy}z - (D_xT^D)(y,z) - T^D(T^D(x,y),z)  \} = 0   \label{eq:B1} \\
&& {\mathfrak S}\{ (D_xR^D)_{yz} + R^D_{T^D(x,y)\, z} \} = 0  \label{eq:B2}
\end{eqnarray*}
where ${\mathfrak S}$ means the sum over all cyclic permutation of $x,y,z$. When the manifold $M$ is complex and the connection $D$ is {\em Hermitian,} namely, satisfies $Dg=0$ and $DJ=0$, where $J$ is the almost complex structure, then $R^D$ satisfies
$$ R^D(x,y,Jz,Jw) = R^D(x,y,z,w) $$
for any tangent vectors $x,y,z,w$, or equivalently,
\begin{equation}
 R^D(x,y,Z,W) = R^D(x,y,\overline{Z},\overline{W}) =0  \label{eq:2.10}
 \end{equation}
for any type $(1,0)$ tangent vectors $Z$, $W$. From now on, we will use $X$, $Y$, $Z$, $W$ to denote type $(1,0)$ vectors. For any Hermitian connection $D$, with the skew-symmetry of first two or last two positions in mind, the above equation (\ref{eq:2.10}) says that the only possibly non-zero components of $R^D$ are $R^D_{XYZ\overline{W}}$, $R^D_{X\overline{Y}Z\overline{W}}$ and their conjugations.

Let us now apply the Bianchi identities to the Chern connection $\nabla$, and use the fact that $T^c(e_i, \overline{e}_j)=0$ and $R^c_{e_ie_j}x =0$ for any $x$ where $e$ is any local frame of type $(1,0)$ vectors, we get the following
\begin{eqnarray}
&&  T^{\ell}_{ij;k} + T^{\ell}_{jk;i}+ T^{\ell}_{ki;j}\, = \,2  \sum_{r=1}^n \big( T^r_{ij}T^{\ell}_{kr} + T^r_{jk}T^{\ell}_{ir}+ T^r_{ki}T^{\ell}_{jr}\big)  \label{eq:B1a} \\
&&  R^c_{k\overline{j}i\overline{\ell}} -  R^c_{k\overline{j}i\overline{\ell}} \, = \, 2 T^{\ell}_{ik;\overline{j}} \label{eq:B1b} \\
&& R^c_{i\overline{j}k\overline{\ell};\, m} - R^c_{m\overline{j}k\overline{\ell}; \,i} \,= \, 2\sum_{r=1}^n T^{r}_{im} R^c_{r\overline{j}k\overline{\ell}}  \label{eq:B2}
\end{eqnarray}
 for any indices $1\leq i,j,k,\ell, m\leq n$. Here the indices after semicolon stand for covariant derivatives with respect to the Chern connection $\nabla^c$. Note that in formula (\ref{eq:B1b}) we assumed that the frame $e$ is unitary, otherwise the right hand side needs to be multiplied on the right by $g$, the matrix of the metric.

\vspace{0.3cm}

\section{Torsion and curvature of Bismut connection}\label{TC}

Since we will primarily be interested in Bismut connection $\nabla^b$, we would like to have formula involving the curvature and covariant differentiation of $\nabla^b$. Using the Bianchi identities for general metric connections, we could get formula for $\nabla^b$ which are similar to the ones for Chern connection mentioned in the previous section. But since we also need relations between $R^c$ and $R^b$ later, let us deduce them from the structure equations. Under a $(1,0)$-frame $e$, the components of the Chern, Bismut and Riemannian curvature tensors are given by
\begin{equation*}
R^c_{i\overline{j}k\overline{\ell}}  =  \sum_{p=1}^n \Theta_{kp}(e_i,
\overline{e}_j)g_{p\overline{\ell}}, \ \ \  \ \ R^b_{abk\overline{\ell}}  =  \sum_{p=1}^n \Theta^b_{kp}(e_a,
e_b)g_{p\overline{\ell}}, \ \ \ \ \ R_{abcd} = \sum_{f=1}^{2n} \hat{\Theta}_{cf}(e_a,e_b)g_{fd},
\end{equation*}
where $i,j,k,\ell,p$ range from $1$ to $n$, while $a,b,c,d,f$ range from $1$ to $2n$ with
$e_{n+i}=\overline{e}_i$, and $g_{ab}=g(e_a,e_b)$. It follows that, for any Hermitian connection $D$, $R^D_{abij} = R^D_{ab\overline{i}\overline{j}} =0$ by the discussion above.

For convenience, we will assume from now on that our local $(1,0)$-frame $e$ is  unitary for the Hermitian metric $g$. We have

\begin{lemma} \label{lemma3.1}
Let $(M^n,g)$ be a Hermitian manifold. It follows that
\begin{eqnarray}
&& R^{b}_{ijk\bar{\ell}} \ = \
2\big( T^{\ell}_{kj,i}-T^{\ell}_{ki,j} \big) +4\sum_r \big( T^r_{ij}T^{\ell}_{rk} + T^r_{jk}T^{\ell}_{ri} + T^r_{ki}T^{\ell}_{rj} \big) ,       \label{curv20}\\
 && R^b_{i\bar{j}k\bar{\ell}}-R^c_{i\bar{j}k\bar{\ell}} \ = \ 2\big( T^{\ell}_{ik,\bar{j}} + \overline{T^{k}_{j\ell,\bar{i}}} \big)  - 4 \sum_r  \big( T^{\ell}_{kr}\overline{T^{i}_{jr}} + T^{j}_{ir}\overline{T^k_{\ell r}}
+ T^r_{ik}\overline{T^r_{j\ell}} - T^{\ell}_{ir}\overline{T^{k}_{jr}} \big),     \label{curv11}\\
&& T^k_{ij,\ell} + T^k_{j\ell,i} + T^k_{\ell i,j} \ = \ 4\sum_r  \big( T^r_{\ell i}T^k_{r j} + T^r_{j\ell}T^k_{ri} + T^r_{ij}T^k_{r \ell} \big) ,      \label{pT} \\
&& T^j_{ik,\bar{\ell}} + \overline{T^i_{j\ell,\bar{k}}} - \overline{T^k_{j\ell,\bar{i}}} \ \, = \ \,   -\,2 \sum_r \big( T^r_{ik}\overline{T^r_{j\ell}} + T^j_{ir} \overline{T^k_{\ell r}} - T^j_{kr}\overline{T^i_{\ell r}}
-T^\ell_{ir} \overline{T^k_{jr}} + T^\ell_{kr}\overline{T^i_{jr}} \big)    +   \label{dbT} \\
&& \hspace{4cm} + \ \frac{1}{2}\,\big( R^{s}_{i\bar{\ell}k\bar{j}}-R^s_{k\bar{\ell}i\bar{j}} \big), \notag \\
&& R^{b}_{ijk\bar{\ell},p} + R^{b}_{pik\bar{\ell},j} +R^{b}_{jpk\bar{\ell},i} \ = \
 -\,2\sum_r \big( R^b_{irk\bar{\ell}}T_{jp}^r + R^b_{jrk\bar{\ell}}T_{pi}^r + R^b_{prk\bar{\ell}}T_{ij}^r \big) , \label{dR_30}  \\
&& R^{b}_{ipk \bar{\ell},\bar{q}} - R^{b}_{i\bar{q}k \bar{\ell}, p} + R^{b}_{p \bar{q} k \bar{\ell}, i} \  = \ 2\sum_r \{  R^b_{r\bar{q}k\bar{\ell}}T^r_{ip}  - R^b_{p\bar{r}k\bar{\ell}}T^q_{ir}   + R^b_{i\bar{r}k\bar{\ell}}T^q_{pr} +
\label{dR_21} \\
&& \hspace{4.8cm} +  \ R^b_{prk\bar{\ell}} \overline{T^i_{qr}}  - R^b_{irk\bar{\ell}}\overline{T^p_{qr}}\} ,\notag \end{eqnarray}
for any $i,j,k,\ell,p,q$, where indices after comma mean covariant derivatives with respect to $\nabla^{b}$.
\end{lemma}

\begin{proof}
Fix a point $p\in M$. We will choose our local unitary frame $e$ near $p$ such that $\theta^b(p)=0$, thus $\theta^c=-2\gamma$ at $p$.
From the structure equation $\, \Theta^b = d \theta^b - \theta^b \wedge \theta^b$,  it yields that
\[\Theta^b = \Theta + 2 (d \gamma + 2 \gamma \wedge \gamma) \]
at $p$. By comparison for the types of the differential forms we get
\begin{eqnarray}
\Theta^b_{2,0}&=& 2(\partial \gamma' + 2 \gamma' \wedge \gamma'), \notag\\
\Theta^b_{1,1}&=& \Theta + 2(\overline{\partial} \gamma' - \partial\overline{^t\!\gamma'}
- 2 \gamma' \wedge \overline{^t\!\gamma'} -2 \overline{^t\!\gamma'} \wedge \gamma'),\notag
\end{eqnarray}
hence \eqref{curv20} and \eqref{curv11} are established. From the first Bianchi identity $\, d \tau = - ^t\!\theta \wedge \tau + ^t\!\!\Theta \wedge \varphi$,
it yields at $p$ that
\begin{eqnarray}
\partial \tau&=& 2 ^t\!\gamma' \wedge \tau, \notag\\
\overline{\partial} \tau&=& -2\overline{\gamma'} \wedge \tau + ^t\!\!\Theta \wedge \varphi,\notag \\
&=& -2\overline{\gamma'} \wedge \tau + ^t\!\!\Theta^b_{1,1} \wedge \varphi
- 2^t\!(\overline{\partial} \gamma' - \partial\overline{^t\!\gamma'}
- 2 \gamma' \wedge \overline{^t\!\gamma'} -2 \overline{^t\!\gamma'} \wedge \gamma') \wedge \varphi, \notag
\end{eqnarray}
which imply \eqref{pT} and \eqref{dbT}. By the second Bianchi identity, $d \Theta^b = \theta^b \wedge \Theta^b - \Theta^b \wedge \theta^b$,  we have $d \Theta^b=0$ at $p$ since  $\theta^b=0$ at $p$. This leads to
\[\partial \Theta^b_{2,0}=0,\quad \overline{\partial} \Theta^b_{2,0}+ \partial \Theta^{b}_{1,1}=0 \]
which yields \eqref{dR_30} and \eqref{dR_21}.
\end{proof}

\begin{lemma} \label{lemma3.2}
Let $(M,g)$ be a Hermitian manifold. It follows that
\begin{eqnarray}
T^{\ell}_{ij,k} &= & -\frac{1}{2}(R^b_{jki \bar{\ell}} + R^b_{kij \bar{\ell}}), \label{pT_refined}\\
\sum_r(T^r_{ij}T^{\ell}_{rk} + T^r_{jk}T^{\ell}_{ri} + T^r_{ki}T^{\ell}_{rj}) &=&
-\frac{1}{4}(R^b_{ijk \bar{\ell}} + R^b_{jki \bar{\ell}} + R^b_{kij \bar{\ell}}), \label{permutation}\\
T^j_{ik,\bar{\ell}} & = & -\frac{2}{3} \sum_r(T^r_{ik}\overline{T^r_{j\ell}} + T^j_{ir} \overline{T^k_{\ell r}}
- T^j_{kr}\overline{T^i_{\ell r}} -T^\ell_{ir} \overline{T^k_{jr}} + T^\ell_{kr}\overline{T^i_{jr}}) \label{dbT_refined}\\
&&+ \frac{1}{3} (R^b_{i \bar{\ell} k \bar{j}} - R^b_{k \bar{\ell} i \bar{j}})
+ \frac{1}{6} (R^b_{i \bar{j} k \bar{\ell}} - R^b_{k \bar{j} i \bar{\ell}}), \notag\\
\eta_{i,j} & = & -\frac{1}{2} \sum_r (R^b_{ijr\bar{r}} + R^b_{jri\bar{r}}), \label{peta}\\
\eta_{i,\bar{j}} & = & -\frac{2}{3}(\sum_{r}\eta_r \overline{T^i_{jr}} + \overline{\eta}_r T^j_{ir} - \sum_{t,r}T^j_{tr}\overline{T^i_{tr}}) \label{dbeta} \\
& & + \frac{1}{3}\sum_{r}(R^b_{r\bar{j}i\bar{r}} - R^b_{i\bar{j}r\bar{r}})
+ \frac{1}{6}\sum_r(R^b_{r\bar{r}i\bar{j}} - R^b_{i\bar{r}r\bar{j}}) \notag.
\end{eqnarray}
for any $i,j,k,\ell$, where indices after comma mean the covariant derivative with respect to $\nabla^{b}$.
\end{lemma}

\begin{proof}
Equations \eqref{curv20} and \eqref{pT} imply  \eqref{pT_refined} and \eqref{permutation}, and equation \eqref{dbT} implies \eqref{dbT_refined}. Also, since $\eta_k=\sum_i T^i_{ik}$, \eqref{peta} and \eqref{dbeta} are implied by \eqref{pT_refined} and \eqref{dbT_refined}, respectively.
\end{proof}

\begin{definition}[See also (27) and (34) in \cite{ZhaoZ19Str}] \label{def3.3}
Let us introduce the following notations
\[P_{ik}^{j \ell}:=  \sum_r T^r_{ik}\overline{T^r_{j\ell}} + T^j_{ir} \overline{T^k_{\ell r}}
- T^j_{kr}\overline{T^i_{\ell r}} -T^\ell_{ir} \overline{T^k_{jr}} + T^\ell_{kr}\overline{T^i_{jr}}. \ \]
\[ A_{k\overline{\ell }} := \sum_{r,s} T^r_{sk} \overline{ T^r_{s\ell } } , \quad   \quad \ \  B_{k\overline{\ell }} := \sum_{r,s} T^{\ell }_{rs} \overline{ T^k_{rs } } ,  \quad  \quad  \quad  \quad \quad  \  \]
\[C_{ik} := \sum_{r,s} T^r_{si} T^s_{rk},  \quad  \quad  \quad  \phi^{\ell }_k := \sum_r \overline{\eta}_r T^{\ell }_{kr}, \quad  \quad  \quad  \quad  \quad  \quad    \]
\[Q_{i\bar{j}k\bar{\ell}} := R^b_{i\bar{j}k\bar{\ell}} - R^b_{k \bar{j} i \bar{\ell}},
\quad \mathrm{Ric}(Q)_{i\bar{j}} := \sum_r R^b_{i\bar{j}r\bar{r}} - R^b_{r \bar{j} i \bar{r}}\]
under any unitary frame. It is clear that $C$ is symmetric, $A$, $B$ are Hermitian symmetric, and $P_{ik}^{j \ell}$ satisfies
\[ P^{j\ell}_{\,ik} = - P^{j\ell }_{\,ki} = - P^{\ell j}_{\,ik} = \overline{ P^{ik}_{j\ell } }.\]
\end{definition}

\begin{proposition}\label{plcld}
Let $(M^n,g)$ be a Hermitian manifold, denote by $\omega$ the K\"ahler form of $g$. Then
\[\sqrt{-1}\partial \overline{\partial} \omega =\left\{ \frac{1}{2}(T^{\ell}_{ik,\overline{j}} - T^j_{ik,\overline{\ell}})-P^{j\ell}_{ik} \right\} \varphi_i \wedge \varphi_k \wedge \overline{\varphi}_j \wedge \overline{\varphi}_{\ell}.\]
Hence we have the following equivalence
\[\begin{aligned}
\partial \overline{\partial} \omega=0 \quad & \Longleftrightarrow \quad T^{j}_{ik,\overline{\ell}} - T^{\ell}_{ik,\overline{j}}=-2P^{j\ell}_{ik}\\
& \Longleftrightarrow \quad (R^b_{i \overline{\ell} k \overline{j}}-R^b_{k \overline{\ell} i \overline{j}})
- (R^b_{i \overline{j} k \ell} - R^b_{k \overline{j} i \overline{\ell}}) = -4 P^{j \ell}_{ik}.
\end{aligned}\]
\end{proposition}

\begin{proof}
By \cite[Lemma 1]{ZhaoZ19Str}, we have
\[\begin{aligned}
\sqrt{-1} \partial \overline{\partial} \omega &=  ^t\!\tau  \overline{\tau} + \, ^t\!\varphi \Theta \overline{\varphi}\\
&=\sum_{i,j,k,\ell,p}\left(T^p_{ik}\overline{T^p_{j\ell}}
- \frac{1}{4}(R^c_{i\bar{j}k\bar{\ell}} - R^c_{k\bar{j}i\bar{\ell}}- R^c_{i\bar{\ell}k\bar{j}}+R^c_{k\bar{\ell}i\bar{j}})\right)
\varphi_i \wedge \varphi_k \wedge \overline{\varphi}_j \wedge \overline{\varphi}_{\ell}.
\end{aligned}\]
Note that, by \cite[Lemma 7]{YZ18Cur}, for any $i,j,k,\ell$,
\begin{equation}\label{eq:DR}
2T^j_{ik;\overline{\ell}}  =  R^c_{k\overline{\ell}i\overline{j}} - R^c_{i\overline{\ell}k\overline{j}},
\end{equation}
where the index after the semicolon stands for covariant derivative with respect to the Chern connection $\nabla^c$, and it is easy to obtain,
since $\nabla^{c}-\nabla^b = -2\gamma$,
\begin{eqnarray}
T^j_{ik ; \overline{\ell}} & = & \overline{e}_{\ell} T^j_{ik} - \sum_q \left(T^j_{qk}\,g( \nabla^{c}_{\overline{e}_{\ell} } e_i, \overline{e}_q )
+ T^j_{iq}\,g( \nabla^{c}_{\overline{e}_{\ell} } e_k, \overline{e}_q )
+ T^q_{ik}\, g( \nabla^{c}_{\overline{e}_{\ell} }  \overline{e}_j, e_q) \right) \nonumber \\
& = & T^j_{ik , \overline{\ell}} + 2\sum_q \left( T^j_{qk} \gamma_{iq}(\overline{e}_{\ell}) + T^j_{iq} \gamma_{kq}(\overline{e}_{\ell}) - T^q_{ik} \gamma_{qj}(\overline{e}_{\ell}) \right) \nonumber \\
& = & T^j_{ik , \overline{\ell}} -2 \sum_q \left( T^j_{qk} \overline{T^i_{q\ell }}  + T^j_{iq} \overline{T^k_{q\ell }} - T^q_{ik} \overline{T^q_{j\ell }} \right)\label{eq:deltaTbar},
\end{eqnarray}
where the index after the comma means covariant derivative with respect to $\nabla^b$. Then the equalities \eqref{eq:DR} and \eqref{eq:deltaTbar} imply \[\sqrt{-1}\partial \overline{\partial} \omega =\left\{ \frac{1}{2}(T^{\ell}_{ik,\overline{j}} - T^j_{ik,\overline{\ell}})-P^{j\ell}_{ik} \right\} \varphi_i \wedge \varphi_k \wedge \overline{\varphi}_j \wedge \overline{\varphi}_{\ell}.\]
Therefore the equivalence in the proposition follows from the equality \eqref{dbT_refined}.
\end{proof}

\begin{proposition}\label{Tderivative}
Let $(M^n,g)$ be a Hermitian manifold, with $\omega$ its K\"ahler form.
Then it holds that
\begin{eqnarray}
\nabla^b_{1,0} T^c=0 & \Longleftrightarrow &
R^b_{ijk\bar{\ell}}=0\quad \Longrightarrow \quad \sum_r(T^r_{\ell i}T^k_{r j} + T^r_{j\ell}T^k_{ri} + T^r_{ij}T^k_{r \ell})=0,\label{parallel_10}\\
\nabla^b_{0,1} T^c=0 & \Longleftrightarrow &
R^b_{i \bar{j} k \bar{\ell}} - R^b_{k \bar{j} i \bar{\ell}} = -4P_{ik}^{j \ell} \label{parallel_01},
\end{eqnarray}
where $\nabla^b_{1,0} T^c$ and $\nabla^b_{0,1} T^c$ are respectively the $(1,0)$- and $(0,1)$-components of $\nabla^bT^c$.
\end{proposition}

\begin{proof}
It is clear that the condition $\nabla^b_{1,0}T^c=0$ (that is, $T_{ij,k}^\ell=0$) implies $\sum_r(T^r_{\ell i}T^k_{r j} + T^r_{j\ell}T^k_{ri} + T^r_{ij}T^k_{r \ell})=0$ by \eqref{pT}. Thus $R^b_{ijk\bar{\ell}}=0$ follows from \eqref{curv20}. Conversely, by \eqref{pT_refined} and \eqref{permutation}, $R^b_{ijk\bar{\ell}}=0$ implies that $T_{ij,k}^\ell=0$ and $\sum_r(T^r_{\ell i}T^k_{r j} + T^r_{j\ell}T^k_{ri} + T^r_{ij}T^k_{r \ell})=0$. Therefore, the conclusion \eqref{parallel_10} is established.

In the mean time, by \eqref{dbT}, we know that the condition $\nabla^b_{0,1}T^c=0$ (that is, $T_{ik,\bar{\ell}}^j=0$) implies that
\[ R^b_{i \bar{\ell} k \bar{j}} - R^b_{k \bar{\ell} i \bar{j}} =4P_{ik}^{j\ell}, \]
which yields
\[ R^b_{i \bar{j} k \bar{\ell}} - R^b_{k \bar{j} i \bar{\ell}} =4P_{ik}^{\ell j}=-4P_{ik}^{j \ell}.\]
Conversely, if $R^b_{i \bar{j} k \bar{\ell}} - R^b_{k \bar{j} i \bar{\ell}} =-4P_{ik}^{j \ell}$, then by the fact that $P_{ik}^{j\ell}=-P_{ik}^{\ell j}$ we get
\[R^b_{i \bar{j} k \bar{\ell}} - R^b_{k \bar{j} i \bar{\ell}} = -(R^b_{i \bar{\ell} k \bar{j}} - R^b_{k \bar{\ell} i \bar{j}}),\]
so by \eqref{dbT_refined} we have
\[T_{ik,\bar{\ell}}^j = -\frac{2}{3}P_{ik}^{j \ell} - \frac{1}{6} (R^b_{i \bar{j} k \bar{\ell}} - R^b_{k \bar{j} i \bar{\ell}}) = -\frac{2}{3}P - \frac{1}{6}(-4P) = 0.\]
This establishes the equivalence of \eqref{parallel_01}.
\end{proof}

\begin{proposition} \label{prop3.6}
Let $(M^n,g)$ be a Hermitian manifold with $\nabla^bT^c=0$. Then  it holds that
\begin{enumerate}
\item\label{R20=0} $R^b_{ijk\bar{\ell}}=0$,
\item\label{R11_12=34} $R^b_{i \bar{j} k \bar{\ell}} = R^b_{k \bar{\ell} i \bar{j}}$,
\item\label{Q_pll} $\nabla^b Q=0$,
\item\label{R=0} $R^b_{xy \chi w}=0$, where $x,y,w$ are any tangent vector and $\chi$ is the associated vector field of Gauduchon's torsion $1$-form $\eta$,
\end{enumerate}
for any $i,j,k,\ell$. Also, under the assumption $\nabla^b T^c=0$,  we have
\begin{eqnarray}
&& \partial \overline{\partial} \omega =0 \quad \Longleftrightarrow \quad   P_{ik}^{j \ell}=0  \quad \Longleftrightarrow \quad
R^b_{i \bar{j} k \bar{\ell}} = R^b_{k \bar{j} i \bar{\ell}}, \label{eq_plcld}\\
&& \partial \eta=0,\quad \overline{\partial} \eta = -2\sum_{i,j,k}\eta_k \overline{T^i_{jk}} \varphi_i \wedge \overline{\varphi}_j,
\end{eqnarray}
In particular, the  metric $g$ is Gauduchon, namely, $\partial \overline{\partial} \omega^{n-1}=0$.
\end{proposition}

\begin{proof}
Let $(M^n,g)$ be a Hermitian manifold satisfying the condition $\nabla^b T^c=0$, which means $\nabla^b_{1,0} T^c=0$ and $\nabla^b_{0,1} T^c=0$. By \eqref{parallel_10} and \eqref{parallel_01}, we get \eqref{R20=0} and \eqref{Q_pll} immediately, as $P_{ik}^{j\ell}$ is $\nabla^b$-parallel. Since Gauduchon's torsion 1-form $\eta=\sum_k \eta_k \varphi_k = \sum_{i,k}T^i_{ik}\varphi_k$ is the trace of the Chern torsion $T^c$, its associated vector field $\chi=\sum_k \overline{\eta_k} e_k$ is clearly $\nabla^b$-parallel and thus
\[R^b_{xy\chi w} = g\left((\nabla^b_x\nabla^b_y-\nabla^b_y\nabla^b_x-\nabla^b_{[x,y]})\chi,w\right)=0,\]
which yields \eqref{R=0}. As to \eqref{R11_12=34}, from $R^b_{i \bar{j} k \bar{\ell}} - R^b_{k \bar{j} i \bar{\ell}} = -4P_{ik}^{j \ell}$, it yields that
\[\begin{aligned}
R^b_{i \bar{j} k \bar{\ell}} - R^b_{k \bar{\ell} i \bar{j}} &=  R^b_{i \bar{j} k \bar{\ell}} - R^b_{k \bar{j} i \bar{\ell}} + R^b_{k \bar{j} i \bar{\ell}} - R^b_{k \bar{\ell} i \bar{j}} \\
&= -4P_{ik}^{j \ell} + (\overline{R^s_{j \bar{k} \ell \bar{i}} - R^b_{\ell \bar{k} j \bar{i}}})\\
&= -4P_{ik}^{j \ell} - \overline{4P_{j \ell}^{k i}}\\
&= -4P_{ik}^{j \ell} -4P_{ki}^{j \ell} \\
&=0.
\end{aligned}\]
Under the assumption $\nabla^bT^c=0$, the equivalence \eqref{eq_plcld} follows from Proposition \ref{plcld} and \eqref{parallel_01}. For any given point $p\in M$, let $e$ be a local unitary frame near $p$ such that  $\theta^b|_p=0$. The structure equation gives us $d\varphi = - \,^t\!\theta \wedge \varphi + \tau = 2 \,^t\!\gamma \wedge \varphi + \tau$, hence $\partial \varphi = -\tau$ and $\overline{\partial} \varphi = 2\overline{\gamma'}\wedge \varphi$  at $p$. Therefore
\[\begin{aligned}
\partial \eta &= \partial (\sum_i \eta_i \varphi_i)
= -\sum_{i,j}\eta_{i,j}\varphi_i \wedge \varphi_j - \sum_{i,j,k} \eta_k T^k_{ij} \varphi_i \wedge \varphi_j,\\
\overline{\partial} \eta &=\overline{\partial} (\sum_i \eta_i \varphi_i)
= - \sum_{i,j} \eta_{i,\bar{j}} \varphi_i \wedge \overline{\varphi}_j - 2 \sum_{i,j,k} \eta_k \overline{T^i_{jk}} \varphi_i \wedge \overline{\varphi}_j.
\end{aligned}\]
Here the indices after comma mean the covariant derivative with respect to $\nabla^{b}$. It is clear that $\nabla^b \eta=0$ and thus $\eta_{i,j}=0,\eta_{i,\bar{j}}=0$. Also, \eqref{parallel_10} implies that
\[  \sum_r(T^r_{\ell i}T^k_{r j} + T^r_{j\ell}T^k_{ri} + T^r_{ij}T^k_{r \ell})=0. \]
Let $k=\ell$ in the equation above and sum up, which yields that $\sum_r\eta_r T^r_{ij}=0$. Hence,
\[\partial \eta=0,\quad \overline{\partial} \eta = -2\sum_{i,j,k}\eta_k \overline{T^i_{jk}} \varphi_i \wedge \overline{\varphi}_j.\]
From the definition of $\eta$, it follows that $\partial \omega^{n-1} = -2 \eta \wedge \omega^{n-1}$ and thus
\[ - \sqrt{-1} \partial \overline{\partial} \omega^{n-1} = - 2 \sqrt{-1} (\overline{\partial} \eta
+ 2 \eta \wedge \overline{\eta}) \wedge \omega^{n-1} = 2 (\sum_{i} \eta_{i,\bar{i}}) \frac{\omega^n}{n} = 0.\]
This completes the proof of the proposition. \end{proof}

\vspace{0.3cm}

\section{Manifolds with Bismut parallel torsion}\label{BTP}

In this section, we will prove Theorem \ref{theorem1.1}, which characterizes Hermitian manifolds with Bismut parallel torsion in terms of the behavior of the Bismut curvature $R^b$.

First we will give the following two technical lemmata, the latter  has been already shown in the proof of \cite[Theorem 2]{ZhaoZ19Str}, but we include the proof here for readers' convenience.

\begin{lemma}\label{swap}
Let $(M^n,g)$ be a Hermitian manifold. Then the following are equivalent
\begin{enumerate}
\item $R^b_{i \bar{j} k \bar{\ell}} = R^b_{k \bar{\ell} i \bar{j}}$,
\item $Q_{i \bar{j} k \bar{\ell}}=-Q_{i\bar{\ell} k \bar{j}}= \overline{Q_{j\bar{i}\ell \bar{k}}}$,
\item $T^j_{ik ,\overline{\ell }} = - T^{\ell}_{ik ,\overline{j } } = \overline{  T^i_{j\ell  ,\overline{k }} }$,
\end{enumerate}
for any $i,j,k,\ell$.
\end{lemma}

\begin{proof}
First let us assume that $R^b_{i \bar{j} k \bar{\ell}} = R^b_{k \bar{\ell} i \bar{j}}$ holds. It follows that
\[\begin{aligned}
Q_{i\bar{j}k\bar{\ell}} &= R^b_{i\bar{j}k\bar{\ell}} - R^b_{k\bar{j}i\bar{\ell}}
= R^b_{k\bar{\ell}i\bar{j}} - R^b_{i\bar{\ell}k\bar{j}} = -Q_{i\bar{\ell}k\bar{j}},\\
\overline{Q_{i\bar{j}k\bar{\ell}}} &= R^b_{j\bar{i}\ell\bar{k}} - R^b_{j\bar{k}\ell\bar{i}}
= R^b_{j\bar{i}\ell\bar{k}} - R^b_{\ell\bar{i}j\bar{k}} = Q_{j\bar{i}\ell\bar{k}}.
\end{aligned}\]
By equation \eqref{dbT_refined}, we get
\[T^j_{ik,\bar{\ell}} = -\frac{2}{3}P_{ik}^{j \ell} - \frac{1}{6}Q_{i\bar{j}k\bar{\ell}},\]
which implies that
\[T^j_{ik ,\overline{\ell }} = - T^{\ell}_{ik ,\overline{j } } = \overline{  T^i_{j\ell  ,\overline{k }} }.\]
Conversely, assuming that  the equality $T^j_{ik ,\overline{\ell }} = - T^{\ell}_{ik ,\overline{j } } = \overline{  T^i_{j\ell  ,\overline{k }} }$ holds. Then by \eqref{dbT} we get
\[Q_{i\bar{\ell}k\bar{j}}=R^{b}_{i\bar{\ell}k\bar{j}}-R^b_{k\bar{\ell}i\bar{j}} =6 T_{ik,\bar{\ell}}^j + 4P^{j\ell}_{ik},\]
which implies that $Q_{i\bar{j}k\bar{\ell}} = -Q_{i\bar{\ell}k\bar{j}}$ and $\overline{Q_{i\bar{j}k\bar{\ell}}}= Q_{j\bar{i}\ell\bar{k}}$. Also,
\[\begin{aligned}
R^b_{i \bar{j} k \bar{\ell}} - R^b_{k \bar{\ell} i \bar{j}} &= R^b_{i \bar{j} k \bar{\ell}} - R^b_{k\bar{j}i\bar{\ell}}
+R^b_{k\bar{j}i\bar{\ell}} - R^b_{k \bar{\ell} i \bar{j}} \\
&= Q_{i\bar{j}k\bar{\ell}} + \overline{R^b_{j\bar{k}\ell\bar{i}} - R^b_{\ell \bar{k} j \bar{i}}}\\
&= Q_{i\bar{j}k\bar{\ell}} + \overline{Q_{j\bar{k}\ell\bar{i}}} \\
&= 0.
\end{aligned}\]
This completes the proof of the lemma. \end{proof}

\begin{lemma}\label{commt}
Let $(M^n,g)$ be a Hermitian manifold. If $R^b_{ijk\bar{\ell}}=0$ for any $i,j,k,\ell$, then the following commutation formulae hold
\begin{enumerate}
\item $\eta_{i, \overline{j}\,\overline{k} } - \eta_{i, \overline{k}\,\overline{j} }  =  2 \sum_r \overline{T^{r}_{kj}} \, \eta_{i, \overline{r}}$,
\item $T^i_{j\ell , \overline{k} \, \overline{p}} - T^i_{j\ell , \overline{p} \, \overline{k} } =  2 \sum_r \overline{T^r_{pk} } T^i_{j\ell , \overline{r}}$,
\end{enumerate}
for any $i,j,k,\ell,p$. Again indices after comma stand for covariant derivatives with respect to $\nabla^b$.
\end{lemma}

\begin{proof}
First let us prove the second equality. From the definition of covariant derivatives,
\[T_{j \ell,\bar{k}}^i= \overline{e}_k T_{j \ell}^i + \sum_r T^i_{\ell r}\,\langle \nabla^b_{\overline{e}_k}e_j,e_r \rangle  - T^i_{j r}\,\langle \nabla^b_{\overline{e}_k} e_{\ell},e_r \rangle  + T^r_{j \ell}\,\langle \nabla^b_{\overline{e}_k}e_r,e_i\rangle .\]
Fix any given point in $M$, choose a local unitary frame $e$ so that the matrix $\theta^b=0$ at the point, then we have
\[T_{j \ell,\bar{k}\,\bar{p}}^i = \overline{e}_p\overline{e}_k T_{j \ell}^i +
\sum_r T^i_{\ell r}\,\langle  \nabla^b_{\overline{e}_p}\nabla^b_{\overline{e}_k}e_j,e_r \rangle  - T^i_{j r}\,\langle \nabla^b_{\overline{e}_p}\nabla^b_{\overline{e}_k} e_{\ell},e_r\rangle   + T^r_{j \ell}\,\langle \nabla^b_{\overline{e}_p}\nabla^b_{\overline{e}_k}e_r,e_i \rangle , \]
which implies that
\[ T^i_{j\ell , \overline{k} \, \overline{p}} - T^i_{j\ell , \overline{p} \, \overline{k} } =
[\overline{e}_p,\overline{e}_k] T_{j \ell}^i + \sum_r T^i_{\ell r}\,R^b_{\bar{p}\bar{k}j\bar{r}} - T^i_{j r}\,R^b_{\bar{p} \bar{k} \ell \bar{r}} + T^r_{j \ell}\,R^b_{\bar{p}\bar{k}r\bar{i}}.\]
At that point, we have
\[ [\overline{e}_p,\overline{e}_k ]= \nabla_{\overline{e}_p}\overline{e}_k-\nabla_{\overline{e}_k}\overline{e}_p = -2\overline{T^r_{kp}} \overline{e}_r,\]
where $\nabla$ is the Riemannian connection. So the second equality in the lemma is established since $R^b_{ijk\bar{\ell}}=0$. By letting $i=j$ in the second equality and sum over, we get the first equality.
\end{proof}

Now we are ready to prove the main result of this article, Theorem \ref{theorem1.1}.

\begin{proof}[{\bf Proof of Theorem \ref{theorem1.1}:}]
By Proposition \ref{prop3.6}, it is clear that $\nabla^b T^c=0$ implies the four conditions in Theorem \ref{theorem1.1}, so we just need to prove the converse. We will follow the strategy of the proof of \cite[Theorem 2]{ZhaoZ19Str}. Note that the first condition is equivalent to $\nabla^b_{1,0}T^c=0$, by \eqref{parallel_10}, which implies \[\sum_r(T^r_{\ell i}T^k_{r j} + T^r_{j\ell}T^k_{ri} + T^r_{ij}T^k_{r \ell})=0,\]
and the fourth condition means
\begin{equation}\label{vR}
\sum_{i,r} \overline{\eta_i} ( R^b_{i\bar{j}r\bar{r}} - R^b_{r \bar{j} i \bar{r}}) =0,
\end{equation}
for any $i,j,\ell$. In the mean time,  $\mathrm{Ric}(Q)$ is Hermitian symmetric due to $R^b_{i\bar{j}k\bar{\ell}}=R^b_{k\bar{\ell}i\bar{j}}$, as
\[\mathrm{Ric}(Q)_{i \bar{j}}= \sum_r R^b_{i \bar{j} r \bar{r}} - R^b_{r \bar{j} i \bar{r}}
= \overline{\sum_r R^b_{j \bar{i} r \bar{r}} - R^b_{j \bar{r} r \bar{i}}}
= \overline{\sum_r R^b_{j \bar{i} r \bar{r}} - R^b_{r \bar{i} j \bar{r}}}
=\overline{\mathrm{Ric}(Q)_{j \bar{i}}}.\]

Firstly, we will show that $|\eta|^2$ is a constant. We have $\eta_{i,j}=0$ since $R^b_{ijk\overline{\ell}}=0$. By \eqref{dbeta} and the symmetry assumption $R^b_{i\bar{j}k\bar{\ell}}=R^b_{k \bar{\ell} i \bar{j}}$, we get
\[|\eta|^2_{,\bar{j}}= \sum_i \eta_{i,\bar{j}}\overline{\eta_i}
=-\frac{2}{3}(\sum_{i,r}\eta_r \overline{T^i_{jr}} + \overline{\eta}_r T^j_{ir} - \sum_{i,t,r}T^j_{tr}\overline{T^i_{tr}})\overline{\eta_i} - \frac{1}{6}\sum_{i,r}(R^b_{i\bar{j}r\bar{r}} - R^b_{r\bar{j}i\bar{r}})\overline{\eta_i}=0\]
for any $j$. Here we used \eqref{vR} and the fact $\sum_r \eta_r T_{ij}^r =0$, which follows from the identity $\sum_r(T^r_{\ell i}T^k_{r j} + T^r_{j\ell}T^k_{ri} + T^r_{ij}T^k_{r \ell})=0$. Similarly $|\eta|^2_{,j}=0$ also holds, thus $|\eta|^2$ is a constant.

Secondly, we will show that $\nabla^b \eta=0$, which is equivalent to $\nabla^b_{0,1}\eta=0$. As $|\eta|^2$ is a constant, it yields, after the covariant derivative in $\overline{\ell}$ is taken, that
\[\sum_k \eta_{k,\bar{\ell}} \,\overline{\eta_k}=0.\]
Take the covariant derivative in $\ell$ and sum up $\ell$, it follows that
\begin{equation}\label{eta_cst}
\sum_{k,\ell} \big( |\eta_{k,\bar{\ell}}|^2 + \eta_{k,\bar{\ell}\,\ell}\,\overline{\eta_{k}}  \big) =0.
\end{equation}
The equality \eqref{dbeta} and $R^b_{i\bar{j}k\bar{\ell}}=R^b_{k\bar{\ell}i\bar{j}}$ imply that
\begin{equation}\label{dbeta_TPLL}
\eta_{i,\bar{j}}  =  -\frac{2}{3}( \overline{\phi_j^i} + \phi_i^j - B_{i\bar{j}})
- \frac{1}{6} \mathrm{Ric}(Q)_{i\bar{j}},
\end{equation}
thus $\eta_{i,\bar{j}}$ is Hermitian symmetric in $i,j$ since $\mathrm{Ric}(Q)_{i\bar{j}}$ is so.
It yields
\[\begin{aligned}
\sum_{k,\ell} |\eta_{k,\bar{\ell}}|^2 &= \frac{4}{9}|\phi+\phi^*-B|^2 + \frac{1}{36} |\mathrm{Ric}(Q)|^2
+\frac{2}{9} \mathrm{Re}\left(\phi \mathrm{Ric}(Q) + \overline{\phi \mathrm{Ric}(Q)} - B \mathrm{Ric}(Q)\right), \\
&= \frac{4}{9}\left(|\phi + \phi^*|^2+|B|^2-4 \mathrm{Re}(\phi B)\right)
+ \frac{1}{36} |\mathrm{Ric}(Q)|^2 + \frac{4}{9} \mathrm{Re} \left( \phi \mathrm{Ric}(Q) \right)- \frac{2}{9} \mathrm{Re} \left( B \mathrm{Ric}(Q) \right),
\end{aligned}\]
where
\[|\phi+\phi^*-B|^2 =\sum_{k,\ell}(\phi_k^\ell+\overline{\phi_\ell^k}-B_{k\bar{\ell}})\overline{(\phi_k^\ell+\overline{\phi_\ell^k}-B_{k\bar{\ell}})},
\quad |\mathrm{Ric}(Q)|^2=\sum_{k,\ell} \mathrm{Ric}(Q)_{k \bar{\ell}} \overline{ \mathrm{Ric}(Q)_{k \bar{\ell}}},\]
\[\phi\mathrm{Ric}(Q)=\sum_{k,\ell} \phi^k_\ell \mathrm{Ric}(Q)_{k \bar{\ell}}, \quad
B\mathrm{Ric}(Q) = \sum_{k,\ell} B_{k\bar{\ell}}\mathrm{Ric}(Q)_{\ell \bar{k}}, \quad \phi B = \sum_{k,\ell} \phi^k_{\ell} \mathrm{Ric}(Q)_{k\bar{\ell}}.\]
Note that
\[|\phi + \phi^*|^2= \sum_{k,\ell}(\phi^\ell_k + \overline{ \phi^k_\ell} ) ( \overline{ \phi^\ell_k + \overline{\phi^k_\ell}} )
= 2|\phi|^2 + 2\mathrm{Re}( \phi \cdot \phi ),\]
where
\[ \phi \cdot \phi = \sum_{k, \ell} \phi_k^\ell \phi_\ell^k.\]
Then it follows from Lemma \ref{commt}, the $\nabla^b$-parallelness of $\mathrm{Ric}(Q)$, \eqref{dbeta_TPLL} and Lemma \ref{swap} that
\[\begin{aligned}
\sum_\ell \eta_{k,\bar{\ell}\, \ell} &= \sum_\ell \left( \overline{\eta_{\ell,\bar{k}}}\right)_{,\ell}=\sum_{\ell} \overline{\eta_{\ell,\bar{k}\,\bar{\ell}}}  = \sum_\ell  \big( \overline{\eta_{\ell, \bar{\ell}\,\bar{k}}
+ 2 \sum_r \overline{T^r_{\ell k}}\eta_{\ell,\bar{r}}} \big)  \\
&= \left(\sum_\ell \eta_{\ell,\bar{\ell}}\right)_{,k} + 2 \sum_{\ell,r} T^r_{\ell k}\eta_{r,\bar{\ell}} \\
&= \left( \frac{2}{3}(|T|^2-2|\eta|^2)-\frac{1}{6} \sum_r \mathrm{Ric}(Q)_{r \bar{r}}\right)_{\!,k}
+ 2 \sum_{\ell,r} T^r_{\ell k}\eta_{r,\bar{\ell}} \\
&= \frac{2}{3} \sum_{i,j,r}T^i_{jr} \overline{T^i_{jr,\bar{k}}} - \frac{4}{3} \sum_{\ell ,r}T^r_{\ell k} \left(\phi_r^{\ell} + \overline{\phi_{\ell}^r} - B_{r \bar{\ell}} + \frac{1}{4} \mathrm{Ric}(Q)_{r \bar{\ell}} \right) \\
&=  \frac{2}{3} \sum_{i,j,r}T^i_{jr} T^j_{ik,\bar{r}} - \frac{4}{3} \sum_{\ell ,r}T^r_{\ell k} \left(\phi_r^{\ell} + \overline{\phi_{\ell}^r} - B_{r \bar{\ell}} + \frac{1}{4} \mathrm{Ric}(Q)_{r \bar{\ell}} \right),
\end{aligned}\]
Hence, it yields that, from \eqref{dbT_refined}, $R^b_{i\bar{j}k\bar{\ell}}=R^b_{k\bar{\ell}i\bar{j}}$, and $\sum_r \eta_r T_{ij}^r =0$,
\[ \begin{aligned}
\sum_{k,\ell}\eta_{k,\bar{\ell}\,\ell} \overline{\eta_k} &=
\frac{2}{3} \sum_{i,j,k,r} \overline{\eta_k} T^i_{jr} T^j_{ik,\bar{r}} - \frac{4}{3} \sum_{k,\ell ,r} \overline{\eta_k} T^r_{\ell k} \left(\phi_r^{\ell} + \overline{\phi_{\ell}^r} - B_{r \bar{\ell}} + \frac{1}{4} \mathrm{Ric}(Q)_{r \bar{\ell}} \right) \\
&= -\frac{4}{9} \sum_{i,j,k,r} \overline{\eta_k} T^i_{jr} \left(\sum_q T^q_{ik}\overline{T^q_{j r}} + T^j_{iq} \overline{T^k_{r q}}
- T^j_{kq}\overline{T^i_{r q}} - T^r_{iq} \overline{T^k_{jq}} + T^r_{kq}\overline{T^i_{jq}}
+ \frac{1}{4} ( R^b_{i\bar{j}k\bar{r}} - R^b_{k \bar{j} i \bar{r}} )\right) \\
&\quad - \frac{4}{3}(\phi \cdot \phi + |\phi|^2 - \phi B + \frac{1}{4}\phi \mathrm{Ric}(Q)) \\
&= -\frac{4}{9}(\phi B - 2\phi A) - \frac{4}{3}(\phi \cdot \phi + |\phi|^2 - \phi B + \frac{1}{4}\phi \mathrm{Ric}(Q))
- \frac{1}{9} \sum_{i,j,k,r} \overline{\eta_k} T^i_{jr} Q_{i\bar{j}k\bar{r}} \\
&= \frac{8}{9}(\phi B + \phi A)- \frac{4}{3}(\phi \cdot \phi + |\phi|^2) - \frac{1}{3} \phi \mathrm{Ric}(Q)
- \frac{1}{9} \sum_{i,j,k,r} \overline{\eta_k} T^i_{jr} Q_{i\bar{j}k\bar{r}} ,
\end{aligned}\]
where
\[\phi A = \sum_{k,\ell}\phi^k_{\ell} A_{k \bar{\ell}}.\]
Hence, the equality \eqref{eta_cst} implies
\[\sum_{k,\ell} |\eta_{k,\bar{\ell}}|^2 + \mathrm{Re}( \eta_{k,\bar{\ell}\,\ell}\overline{\eta_{k}})=0,\] which yields
\begin{equation}\label{eq_1}
\begin{aligned}
|B|^2 - 2\mathrm{Re}(\phi B) + 2 \mathrm{Re}(\phi A) - \frac{1}{2} \mathrm{Re}(B\mathrm{Ric}(Q)) &=
\frac{1}{2}|\phi+\phi^*|^2 - \frac{1}{4} \mathrm{Re}(\phi \mathrm{Ric}(Q))\\
& - \frac{1}{16}|\mathrm{Ric}(Q)|^2 + \frac{1}{4} \mathrm{Re} \left(\sum_{i,j,k,r} \overline{\eta_k} T^i_{jr} Q_{i\bar{j}k\bar{r}}\right).
\end{aligned}
\end{equation}
As $\sum_r \eta_r T^r_{ij}=0$ has already been shown, it follows that, after the covariant derivative in $\bar{\ell}$ is taken,
\[\sum_r \eta_{r,\bar{\ell}}\,T^r_{ij} + \eta_r\,T^r_{ij,\bar{\ell}} =0,\]
which yields
\begin{gather}
\sum_r\left(\phi_r^{\ell} + \overline{\phi_{\ell}^r} - B_{r \bar{\ell}} + \frac{1}{4}\mathrm{Ric}(Q)_{r\bar{\ell}} \right)T^r_{ij} \notag\\
+ \sum_r \eta_r \left(\sum_q (T^q_{ij}\overline{T^q_{r \ell}} + T^r_{iq} \overline{T^j_{\ell q}}
- T^r_{jq}\overline{T^i_{\ell q}} - T^\ell_{iq} \overline{T^j_{rq}} + T^\ell_{jq}\overline{T^i_{rq}})
+ \frac{1}{4} ( R^b_{i\bar{r}j\bar{\ell}} - R^b_{j \bar{r} i \bar{\ell}} )\right)=0 , \notag
\end{gather}
or equivalently,
\begin{equation}\label{tr_etaT}
\sum_r\left((\phi_r^{\ell}-B_{r\bar{\ell}})T^r_{ij} + \frac{1}{4} \mathrm{Ric}(Q)_{r\bar{\ell}} T^r_{ij}
+  \overline{\phi^j_r}T^{\ell}_{ir} - \overline{\phi^i_r}T^{\ell}_{jr} + \frac{1}{4} \eta_r Q_{i\bar{r}j\bar{\ell}}\right)=0.
\end{equation}
Multiply $\overline{\eta_j}$ on both sides of \eqref{tr_etaT} and sum up $j$, it yields
\[ \sum_{r} \left(\phi_r^{\ell}\phi_i^{r} - \phi^r_i B_{r\bar{\ell}}
+ \frac{1}{4} \phi^r_i \mathrm{Ric}(Q)_{r\bar{\ell}} +  \phi^{\ell}_{r}\overline{\phi^i_r}
+ \frac{1}{4} \sum_j \eta_r \overline{\eta_j} Q_{i\bar{r}j\bar{\ell}}\right) =0,\]
where $\sum_r \eta_r T^r_{ij}=0$ is used here. Let $i=\ell$ and sum up $i$, it implies from \eqref{vR} that
\[ \phi B = \phi \cdot \phi + \frac{1}{4} \phi \mathrm{Ric}(Q) + |\phi|^2,\]
and thus
\begin{equation}\label{eq_2}
\mathrm{Re}(\phi B) = \mathrm{Re}(\phi \cdot \phi) + \frac{1}{4} \mathrm{Re}(\phi \mathrm{Ric}(Q)) + |\phi|^2
=\frac{1}{2}|\phi+\phi^*|^2 + \frac{1}{4} \mathrm{Re}(\phi \mathrm{Ric}(Q)).
\end{equation}
Similarly, multiply $\overline{T^{\ell}_{ij}}$ on both sides of \eqref{tr_etaT} and sum up $i,j,\ell$, it yields
\[ \phi B - |B|^2 + \frac{1}{4} B \mathrm{Ric}(Q) + 2 \overline{\phi A}
+ \frac{1}{4}\sum_{i,j,r,\ell} \eta_r \overline{T^{\ell}_{ij}} Q_{i\bar{r}j\bar{\ell}}=0, \]
and thus, by Lemma \ref{swap},
\begin{equation}\label{eq_3}
|B|^2 - \mathrm{Re}(\phi B) - 2 \mathrm{Re}(\overline{\phi A}) - \frac{1}{4} \mathrm{Re}( B \mathrm{Ric}(Q))
=- \frac{1}{4} \mathrm{Re} \left( \sum_{i,j,r,\ell} \overline{\eta_r} T^{\ell}_{ij}Q_{\ell \bar{i} r \bar{j}} \right).
\end{equation}
Note that the equality \eqref{vR} implies that
\[\sum_{i}\overline{\eta_i}\mathrm{Ric}(Q)_{i\bar{j}}=0,\]
for any $j$. Take the covariant derivative in $k$ of the equality above, it yields
\[\sum_i \eta_{i,\bar{k}} \mathrm{Ric}(Q)_{j \bar{i}}=0,\]
where $\nabla^b \mathrm{Ric}(Q)=0$ is used here, and thus
\[ \sum_i \left( \phi_i^k + \overline{\phi_k^i} - B_{i\bar{k}} + \frac{1}{4} \mathrm{Ric}(Q)_{i\bar{k}}\right) \mathrm{Ric}(Q)_{j \bar{i}}=0. \]
Let $j=k$ and sum up $j$, we get
\[ \phi \mathrm{Ric}(Q) + \overline{\phi \mathrm{Ric}(Q)} - B \mathrm{Ric}(Q) + \frac{1}{4}|\mathrm{Ric}(Q)|^2 =0, \]
and thus
\begin{equation}\label{eq_4}
\mathrm{Re}(B\mathrm{Ric}(Q)) = 2 \mathrm{Re}(\phi \mathrm{Ric}(Q))  + \frac{1}{4}|\mathrm{Ric}(Q)|^2 .
\end{equation}
In the meantime, we will also use the equality
\[ \sum_r(T^r_{\ell i}T^k_{r j} + T^r_{j\ell}T^k_{ri} + T^r_{ij}T^k_{r \ell})=0.\]
Multiply $\overline{T^k_{\ell j}}$ on both sides and sum up $j,k,\ell$, we get
\[ \sum_{r,\ell} T^r_{\ell i}A_{r \bar{\ell}} - \sum_{k,r} T^k_{r i}B_{k \bar{r}} +  \sum_{j,r}T^r_{ji}A_{r\bar{j}}=0.\]
Hence
\begin{equation}\label{TBTA}
2 \sum_{r,\ell}T^r_{\ell i}A_{r \bar{\ell}} = \sum_{r,\ell} T^r_{\ell i}B_{r \bar{\ell}}.
\end{equation}
Multiply $\overline{\eta_i}$ on both sides above and sum up $i$, we get
\[2 \phi A = \phi B,\]
and thus
\begin{equation}\label{eq_5}
2 \mathrm{Re} (\phi A) = \mathrm{Re} (\phi B).
\end{equation}
Therefore, the equalities \eqref{eq_1}, \eqref{eq_2}, \eqref{eq_3}, \eqref{eq_4} and \eqref{eq_5} imply that
\[|B|^2 = |\phi + \phi^*|^2 + \mathrm{Re}(\phi \mathrm{Ric}(Q)) + \frac{1}{16}|\mathrm{Ric}(Q)|^2, \quad
\mathrm{Re}(\phi B)= \frac{1}{2}|\phi+\phi^*|^2 + \frac{1}{4} \mathrm{Re}(\phi  \mathrm{Ric}(Q)),\]
\[ \mathrm{Re}(\phi A)= \frac{1}{4}|\phi+\phi^*|^2 + \frac{1}{8} \mathrm{Re}(\phi  \mathrm{Ric}(Q)),
\quad \mathrm{Re}(B\mathrm{Ric}(Q)) = 2 \mathrm{Re}(\phi \mathrm{Ric}(Q))  + \frac{1}{4}|\mathrm{Ric}(Q)|^2, \]
\[\mathrm{Re} \left(\sum_{i,j,k,r} \overline{\eta_k} T^i_{jr} Q_{i\bar{j}k\bar{r}}\right)=0.\]
which yields
\[\sum_{k,\ell}|\eta_{k,\bar{\ell}}|^2 = \frac{4}{9}\left(|\phi + \phi^*|^2+|B|^2-4 \mathrm{Re}(\phi B)\right)
+ \frac{1}{36} |\mathrm{Ric}(Q)|^2 + \frac{4}{9} \mathrm{Re} \left( \phi \mathrm{Ric}(Q) \right)- \frac{2}{9} \mathrm{Re} \left( B \mathrm{Ric}(Q) \right)=0,\]
hence $\nabla^b\eta=0$ is proved.

Thirdly, we will show that $|T|^2$, $\phi_i^j$ and $B_{i \bar{j}}$ are $\nabla^b$-parallel. To this end, the equality \eqref{dbeta_TPLL} implies that
\[0 = \sum_i \eta_{i,\bar{i}}= \frac{2}{3}(|T|^2-2|\eta|^2) - \frac{1}{6} \sum_i \mathrm{Ric}(Q)_{i\bar{i}},\]
and it follows that $|T|^2$ is $\nabla^b$-parallel, since $|\eta|^2$ and $\mathrm{Ric}(Q)$ are. It is clear that $\phi_{i,k}^j=0$ now due to $\nabla^b \eta=0$, while it is easy to see
\[ \phi_{i,\bar{k}}^j = \sum_r \overline{\eta_r} T^j_{ir,\bar{k}} =- \sum_r \overline{\eta_r} T^j_{ri,\bar{k}}
= -\sum_r \overline{\eta_r T^r_{jk,\bar{i}}}=0,\]
where Lemma \ref{swap} and $\sum_r \eta_r T^r_{ij}=0$ are used. Hence $\phi_i^j$ is $\nabla^b$-parallel. Note that the equality \eqref{dbeta_TPLL} implies
\[0 = \eta_{i,\bar{j}} = -\frac{2}{3} \left( \phi_i^j + \overline{\phi_j^i} - B_{i\bar{j}} + \frac{1}{4} \mathrm{Ric}(Q)_{i \bar{j}} \right),\]
which yields that $B_{i \bar{j}}$ is $\nabla^b$-parallel, as $\phi_i^j$ and $\mathrm{Ric}(Q)$ are.

Finally, we are ready to show that $\nabla^b T^c=0$.
It follows from \eqref{TBTA} that
\begin{equation}\label{vTA}
\sum_{i,\ell,r} \left(T^r_{\ell i} A_{r \bar{\ell}}\right)_{,\bar{i}}
= \frac{1}{2}\sum_{i,\ell,r} \left(T^r_{\ell i}B_{r \bar{\ell}}\right)_{,\bar{i}}
= \frac{1}{2}\sum_{i,\ell,r} T^r_{\ell i,\bar{i}}B_{r \bar{\ell}}
= -\frac{1}{2} \sum_{i,\ell,r} T^i_{\ell i,\bar{r}}B_{r \bar{\ell}}
= \frac{1}{2} \sum_{\ell,r} \eta_{\ell,\bar{r}}B_{r \bar{\ell}}=0,
\end{equation}
where Lemma \ref{swap} is used.
Then, since $|T|^2$ is $\nabla^b$-parallel, it follows, after the covariant derivative in $\bar{\ell}$ is taken,
\[0=|T|^2_{,\bar{\ell}}=\sum_{i,j,k}\overline{T^j_{ik}}T^j_{ik,\bar{\ell}} =0.\]
Take the covariant derivative in $\bar{\ell}$  and sum up $\ell$, we get
\[\sum_{i,j,k,\ell} |T^j_{ik,\bar{\ell}}|^2 + \overline{T^j_{ik}}T^j_{ik,\bar{\ell}\,\ell} =0.\]
Note that, from Lemma \ref{swap}, Lemma \ref{commt} and $\nabla^b \eta=0$,
\[\begin{aligned}
\sum_{i,j,k,\ell} \overline{T^j_{ik}}T^j_{ik,\bar{\ell}\,\ell}& =
\sum_{i,j,k,\ell} \overline{T^j_{ik}} \left(\overline{T_{j \ell,\bar{k}}^i}\right)_{,\ell}
= \sum_{i,j,k,\ell} \overline{T^j_{ik} T_{j \ell,\bar{k}\,\bar{\ell}}^i}
= \sum_{i,j,k,\ell} \overline{T^j_{ik}( T_{j \ell,\bar{\ell}\,\bar{k}}^i + 2 \sum_r \overline{T^r_{\ell k}} T^i_{j \ell, \bar{r}} )} \\
&= - \sum_{i,j,k,\ell} \overline{T^j_{ik} T^{\ell}_{j \ell, \bar{i}\,\bar{k}}}
+ 2 \sum_{i,j,k,\ell,r} T^r_{\ell k} \overline{T^j_{ik}} T^j_{ir,\bar{\ell}}
= \sum_{i,j,k} \overline{T^j_{ik}\eta_{j,\bar{i}\,\bar{k}}}
+ 2 \sum_{k,\ell,r} T^r_{\ell k}A_{r\bar{k},\bar{\ell}} \\
&= -2 \sum_{k,\ell,r} \left( T^r_{k \ell}A_{r\bar{k}} \right)_{,\bar{\ell}}
+2 \sum_{k,\ell,r} T^r_{k \ell, \bar{\ell}} A_{r\bar{k}}
= -2 \sum_{k,\ell,r} \left( T^r_{k \ell}A_{r\bar{k}} \right)_{,\bar{\ell}}
-2 \sum_{k,\ell,r} T_{k \ell,\bar{r}}^\ell A_{rk} \\
&= -2 \sum_{k,\ell,r} \left( T^r_{k \ell}A_{r\bar{k}} \right)_{,\bar{\ell}}
+2 \sum_{k,r} \eta_{k,\bar{r}}A_{r\bar{k}}
=-2 \sum_{k,\ell,r} \left( T^r_{k \ell}A_{r\bar{k}} \right)_{,\bar{\ell}} \\
&=0,
\end{aligned}\]
where the last equality is due to \eqref{vTA}. Therefore, $T^j_{ik,\bar{\ell}}=0$ and the proof is completed.
\end{proof}

Next we prove Corollary \ref{corollary1.2}. The proof is the same as the {\em BKL} case in \cite{ZhaoZ19Str}, and we include it here for the sake of completeness.

\begin{proof}[{\bf Proof of Corollary \ref{corollary1.2}: }]
By Proposition \ref{prop3.6}, we know that $\partial \eta = 0$. If $g$ is strongly Gauduchon, that is, there exists a $(n,n-2)$-form $\Omega$ on $M^n$ such that $\partial \omega^{n-1}=\overline{\partial }\Omega$, then by the fact that $\partial \eta =0$ and $\partial \omega^{n-1}=-2\eta \, \omega^{n-1}$, we get
\[ \int_M 2\eta \, \overline{\eta} \,\omega^{n-1} = \int_M \overline{\eta} \,\partial  \omega^{n-1} = \int_M \overline{\eta}\, \overline{\partial }\Omega = \int_M \overline{\partial } \overline{\eta} \,\Omega = 0.\]
Hence $|\eta|^2 = 0$, contradicting to our assumption that $g$ is not balanced. Similarly, if $\overline{\eta} = \overline{\partial }f$ for some smooth function $f$ on $M^n$, then
$$ \int_M \partial \overline{\eta} \wedge \omega^{n-1} = \int_M \partial \overline{\partial }f \wedge \omega^{n-1} =0$$
as $g$ is Gauduchon. However, Proposition \ref{prop3.6} implies $\partial \overline{\eta} = 2 \sum_{i,j} \phi^{j}_{i} \varphi_i \wedge \overline{\varphi}_j$. This leads to the vanishing of $\eta$, hence $g$ is balanced, which is a contradiction. This completes the proof of Corollary \ref{corollary1.2}. \end{proof}

From the proof of Theorem \ref{theorem1.1} above, we know that on a {\em BTP} manifold it holds that
\begin{equation}
\phi + \phi^{\ast} = B -\frac{1}{4}\mbox{Ric}(Q). \label{eq:4.10}
\end{equation}
The right hand side of (\ref{eq:4.10}) is not necessarily non-negative, so the proof of \cite[Theorem 6]{ZhaoZ19Str} does not go through, that is why we cannot prove Conjecture \ref{conjecture1.3} here for non-balanced {\em BTP} manifolds. If $\mbox{Ric}(Q)=0$, or more generally if $4B \geq \mbox{Ric}(Q)$, then Conjecture \ref{conjecture1.3} holds.

Now let us prove Proposition \ref{prop1.4}, where the first part is due to Andrada and Villacampa \cite[Theorem 3.6]{AndV}. For readers convenience we gave a proof here as well.

\begin{proof}[{\bf Proof of Proposition \ref{prop1.4}:}]
Let $(M^n,g)$ be a locally conformally K\"ahler manifold. Then under any local unitary frame $e$, the torsion components are given by
\begin{equation}
 T_{ik}^j = \frac{1}{n-1} \big( \eta_k\delta_{ij} - \eta_i \delta_{kj} \big)  \label{eq:4.11}
 \end{equation}
and the torsion $1$-form $\eta$ is equal to $\eta = (n-1)\partial u$ for some smooth real-valued local function $u$. For any given point $p\in M$, choose the local unitary frame $e$ so that the matrix $\theta^b$ of Bismut connection vanishes at $p$.  Denote by $\varphi$ the local unitary coframe dual to $e$. By \cite{YZ18Cur}, the Levi-Civita connection $\nabla$ has connection matrix
$$ \nabla e = \theta^1 e + \overline{\theta^2} \,\overline{e}, \ \ \ \theta^1=\theta^b -\gamma , \ \ \ \gamma_{ij} = \sum_k \big( T^j_{ik}\varphi_k - \overline{T_{jk}^i } \overline{\varphi_k} \big) , \ \ \ \theta^2_{ij} = \sum_k \overline{T^k_{ij}} \varphi_k.$$
Since $\theta^b|_p=0$, we have at point $p$ that
\begin{eqnarray*}
\nabla \varphi_i & = & - \theta^1_{ji} \otimes \varphi_j - \theta^2_{ji} \otimes \overline{\varphi}_j \ \, = \ \, \gamma_{ji} \otimes \varphi_j - \theta^2_{ji} \otimes \overline{\varphi}_j \\
& = & \sum_{j,k} \big( T^i_{jk} \varphi_k\otimes \varphi_j - \overline{T^j_{ik} } \overline{\varphi}_k \otimes \varphi_j -  \overline{T^k_{ji} } \varphi_k \otimes \overline{\varphi}_j  \big) \\
& = & \frac{1}{n-1}\{ \eta \otimes \varphi_i - \varphi_i \otimes \eta + \varphi_i \otimes \overline{\eta } - \overline{\eta } \otimes \varphi_i + \overline{\eta}_i \sum_k \big( \overline{\varphi}_k \otimes \varphi_k - \varphi_k \otimes \overline{\varphi}_k \big) \} \\
& = & \frac{2}{n-1} \{ \eta \wedge \varphi_i + \varphi_i \wedge \overline{\eta} - \overline{\eta}_i \sum_k  \varphi_k\wedge \overline{\varphi}_k \}
\end{eqnarray*}
Therefore
\begin{eqnarray*}
\nabla \eta & = & \nabla \sum_i \eta_i \varphi_i \  \, = \ \, \sum_{i,j} \big( \eta_{i,j} \varphi_j \otimes \varphi_i + \varphi_{i,\overline{j}} \overline{\varphi}_j \otimes \varphi_i \big) + \sum_i \eta_i \nabla \varphi_i \\
 & = & \sum_{i,j} \big( \eta_{i,j} \,\varphi_j \otimes \varphi_i + \eta_{i,\overline{j}} \,\overline{\varphi}_j \otimes \varphi_i \big) + \frac{2}{n-1} \{ \eta \wedge  \overline{\eta} - \sum_i |\eta_i|^2  \sum_k  \varphi_k\wedge \overline{\varphi}_k \} ,
\end{eqnarray*}
where indices after comma denote covariant derivatives with respect to $\nabla^b$. If $(M^n,g)$ is {\em BTP}, then $\nabla^b\eta=0$, hence $\nabla (\eta + \overline{\eta})=0$ and $g$ is Vaisman. Conversely, when $g$ is Vaisman, which means the manifold is locally conformally K\"ahler with  $\nabla (\eta + \overline{\eta})=0$, then from the above calculation we get
$$ \sum_{i,j} \{  \eta_{i,j} \,\varphi_j \otimes \varphi_i + \eta_{i,\overline{j}} \,\overline{\varphi}_j \otimes \varphi_i +  \overline{\eta_{i,\overline{j}}} \,\varphi_j \otimes \overline{\varphi}_i  + \overline{\eta_{i,j} } \ \overline{\varphi}_j \otimes \overline{\varphi}_i  \}  = 0.$$
Hence $\nabla^b \eta =0$, and $\nabla^b T=0$ by (\ref{eq:4.11}). So we have proved that for a locally conformally K\"ahler manifold, the {\em BTP} condition is equivalent to the Vaisman condition.

Next we consider the uniqueness problem for {\em BTP} metrics within a conformal class. Note that when the manifold is compact, such metrics are clearly unique (up to constant multiples) within each conformal class, if it exists at all, as {\em BTP} metrics are Gauduchon. So the statement in Proposition \ref{prop1.4} is really about non-compact cases.

Suppose $(M^n,g)$ and $(M^n,e^{2u}g)$ are two {\em BTP} manifolds, where $u$ is a smooth real-valued function on $M$ with $du\neq 0$ in an open dense subset $U\subseteq M$. Let $e$ be a local unitary frame for $g$ with dual coframe $\varphi$. Then $\tilde{e}=e^{-u}e$ and $\tilde{\varphi} = e^u\varphi$ are local unitary frame and dual coframe for $\tilde{g}=e^{2u}g$. We have (cf. \cite[the proof of Theorem 3]{YZZ})
\begin{eqnarray*}
e^u\tilde{T}^j_{ik} & = & T^j_{ik} + u_i \delta_{kj} - u_k \delta_{ij} \\
e^u \tilde{\eta}_k & = & \eta_k - (n-1)u_k \\
P_{ik} \ \ & = & 2u_i \varphi_k - \partial u \, \delta_{ik} - 2\overline{u_k} \, \overline{\varphi}_i + \overline{\partial}u \, \delta_{ik}
\end{eqnarray*}
where $P=\tilde{\theta}^b-\theta^b$ and $u_k=e_k(u)$. Since $\nabla^bT=0$ and $\tilde{\nabla}^b \tilde{T} =0$, by the analysis of the $(1,0)$-part, a straight-forward computation leads to
\begin{equation}
u_{i,\ell}\delta_{kj} - u_{k,\ell} \delta_{ij} = 2u_kT^j_{i\ell} - 2u_i T^j_{k\ell} +2u_iu_{\ell} \delta_{kj} - 2u_k u_{\ell} \delta_{ij} - 2\delta_{j\ell} \sum_r u_r T^r_{ik}   \label{eq:4.12}
\end{equation}
for any indices $1\leq i,j,k,\ell \leq n$ under any local unitary frame $e$. As $du\neq 0$ holds on the open dense subset $U$, we may choose $e$ such that $\varphi^n= \frac{\partial u}{|\partial u|_g}$, which implies that $u_1=\cdots =u_{n-1}=0$ and $u_n=|\partial u|_g=\lambda >0$ are established where the local unitary frame $e$ intersects with $U$.

Assume $i<k$. If we take $j\notin \{ i,k\}$ in (\ref{eq:4.12}), we get
\begin{equation}
\delta_{kn} T^j_{i\ell} = \delta_{j\ell} T^n_{ik},   \qquad  i<k, \ j\neq i, k. \label{eq:4.13}
\end{equation}
Take $k=n$ and $j\neq \ell$, we get $T^j_{i\ell}=0$. Take $k<n$ and $j=\ell$, we get $T^n_{ik}=0$. So under this particular unitary frame $e$, we have $T^j_{ik}=0$ whenever all three indices are distinct. Now let us take $n=k$ and $j=\ell$ in (\ref{eq:4.13}), we get $T^j_{ij}=T^n_{in}$, for all $i,j<n$ with $i\neq j$. Therefore $T^{\ell}_{\ell i}$ are equal for $\ell \neq i$, and $T^{\ell}_{\ell i} = \frac{1}{n-1} \eta_i $ whenever $i<n$ and $\ell \neq i$, as $\sum_\ell T^{\ell}_{\ell i}=\eta_i$.

Take $k=\ell =n$ and $i=j<n$ in (\ref{eq:4.12}), which yields that
$$ -u_{n,n}= 2\lambda T^i_{in} - 2\lambda ^2,\qquad i<n.$$
This shows that all $T^i_{in}$ for $i<n$ are equal, hence equal to $\frac{1}{n-1}\eta_n$. Combining the above results, we conclude that under this particular frame $e$, it holds that
\begin{equation}
T^j_{ik} = \frac{1}{n-1} \big( \eta_k \delta_{ij} - \eta_i \delta_{kj}  \big),\qquad 1\leq i,j,k\leq n.   \label{eq:4.14}
\end{equation}
which implies that the above equation holds in the open dense subset $U\subseteq M$, hence in the entire $M$. From this equation, it follows that
$$ \phi_i^j = \sum_r \overline{\eta}_r T^j_{ir} = \frac{1}{n-1}\left( \delta_{ij} \sum_r |\eta_r|^2 - \eta_i \overline{\eta}_j \right).$$
In particular, $\phi_i^j = \overline{\phi^i_j}$ for any $i,j$. By Proposition \ref{prop3.6}, we know that the torsion $1$-form $\eta$ for any {\em BTP} manifold satisfies $\partial \eta =0$ and $\overline{\partial }\eta = -2\sum_{i,j} \overline{\phi^i_j} \varphi_i\wedge \overline{\varphi}_j$. Therefore, it yields that
$$ \overline{\partial }\eta + \partial \overline{\eta} = 2\sum_{i,j} \big( \phi_i^j - \overline{\phi^i_j}) \varphi_i\wedge \overline{\varphi}_j = 0,$$
which amount to $d(\eta + \overline{\eta}) =0$.
It implies that locally there exists smooth, real-valued function $v$ such that $\eta = (n-1)\partial v$. By (\ref{eq:4.14}), the torsion of the metric $e^{2v}g$ will vanish, hence $g$ is locally conformally K\"ahler. This completes the proof of Proposition \ref{prop1.4}
\end{proof}

Next let us prove Proposition \ref{prop1.5}.

\begin{proof}[{\bf Proof of Proposition \ref{prop1.5}:}]
Let $(M^n,g)$ be a complex nilmanifold with nilpotent $J$, namely, its universal cover is a nilpotent Lie group $G$ equipped with a left-invariant complex structure $J$ and a compatible left-invariant metric $g$. Assume that $J$ is nilpotent in the sense of \cite{CFGU} and $\nabla^bT^c=0$. By Lemma 1, Lemma 2 and the discussion right after the proof of Lemma 2 in \cite{ZhaoZ19Nil}, we know that $G$ is of step most 2 and $J$ is abelian (namely, $C=0$), and there exists an integer $1\leq r\leq n$ and a unitary coframe $\varphi$ such that
$$ d\varphi_i =0, \ \ \ \forall \ 1\leq i\leq r; \ \ \ \ \ \ d\varphi_{\alpha} = \sum_{i=1}^r Y_{\alpha i} \varphi_i \wedge \overline{\varphi}_i , \ \ \ \forall \ r<\alpha \leq n. $$
Conversely, for such a complex nilmanifold, it is easy to check that $\nabla^bT^c=0$. This completes the proof of Proposition \ref{prop1.5}. \end{proof}

\vspace{0.3cm}

\section{Non-balanced {\em BTP} manifolds}\label{NBBTP}

When $n\geq 3$, as illustrated by Proposition \ref{prop1.5}, the set of all non-balanced {\em BTP} manifolds in complex dimension $n$ contains the set of all non-K\"ahler {\em BKL} manifolds in complex dimension $n$ as a proper subset.

Given any non-balanced {\em BTP} manifold $(M^n,g)$, the type $(1,0)$ vector field $\chi$  dual to $\eta$ is also $\nabla^b$-parallel, hence is nowhere zero. The following notion of admissible frames introduced by the authors in \cite[Definition 4]{YZZ} to study non-K\"ahler {\em BKL} manifolds can be generalized to non-balanced {\em BTP} manifolds:

\begin{definition}\label{adm}
Let $(M^n,g)$ be a non-balanced {\em BTP} manifold and $\chi$ the vector field dual to Gauduchon's torsion $1$-form $\eta$. A local unitary frame $e$ on $M^n$ is said to be \textbf{admissible}, if $\chi =\lambda e_n$ for some constant $\lambda >0$ and under $e$ the matrix $ \phi = (\phi_i^j)$ is diagonal.
\end{definition}

\begin{corollary}[See also Proposition 3.6 of \cite{B12} and Lemma 3 of \cite{YZZ}]\label{Tparallel_add}
Let $(M^n,g)$ be a non-balanced BTP manifold. Then around any given point there exit an admissible frame $e$ such that $\chi=\lambda e_n$ and $\phi_i^j=\lambda a_i \delta_{ij}$, where $\lambda>0$ and $a_i$ are global constants with $a_n=0$. Therefore, $T^{1,0}M$ admits a $\nabla^b$-parallel, $g$-orthogonal decomposition \[T^{1,0}M=\bigoplus_{i=0}^k H_i,\]
where $H_0$ is generated by $\chi$ and other distributions $\{H_i\}_{i=1}^k$ are the eigenspaces of $\phi$ corresponding to distinct eigenvalues of $\{ \lambda a_1, \ldots , \lambda a_{n-1}\}$.
\end{corollary}

\begin{proof}
Clearly, $|\eta|^2$ is a positive constant, which will be denoted by $\lambda^2$. Thus we may choose the unitary coframe $\varphi$ such that $\eta = \lambda \varphi_n$. This means that $\eta_1=\cdots=\eta_{n-1}=0$ and $\eta_n=\lambda$, thus $\chi=\lambda e_n$. From Proposition \ref{prop3.6}, $R^b_{xy\chi w}=0$ yields $R^b_{i\bar{j}n\bar{\ell}}=0$ for any $i,j,\ell$, and by $R^b_{i\bar{j}k\bar{\ell}}=R^b_{k\bar{\ell}i\bar{j}}$ we get \[R^b_{n\bar{j}k\bar{\ell}}=R^b_{i\bar{n}k\bar{\ell}}=R^b_{i\bar{j}k\bar{n}}=0\]
for any $i,j,k,\ell$. Moreover, the conclusion \eqref{parallel_10} of Proposition \ref{Tderivative} implies
\begin{equation}
\sum_r \eta_r T_{ij}^r =0, \label{eq:5.1}
\end{equation}
which yields  $T^n_{ij}=0$ for any $i$, $j$. Hence $\phi_i^j=\lambda T^j_{in}$, where $\phi_i^n = \phi_n^i =0$ for any $i$.

Since  $R^b_{i\bar{j}k\bar{\ell}}=R^b_{k\bar{\ell}i\bar{j}}$, the equality \eqref{dbT_refined} amounts to
\[0= \sum_r(T^r_{ik}\overline{T^r_{j\ell}} + T^j_{ir} \overline{T^k_{\ell r}}
- T^j_{kr}\overline{T^i_{\ell r}} -T^\ell_{ir} \overline{T^k_{jr}} + T^\ell_{kr}\overline{T^i_{jr}})
 + \frac{1}{4} (R^b_{i \bar{j} k \bar{\ell}} - R^b_{k \bar{j} i \bar{\ell}}). \]
Take $k=\ell =n$ above and multiply $\lambda^2$ on both sides, we get
$$ \sum_r \phi_i^r \overline{ \phi_j^r }- \phi_r^j \overline{ \phi_r^i  }   =0 , \ \ \ \forall \ i,j.$$
That is, the matrix $\phi = (\phi_i^j)$ satisfies $\phi \phi^{\ast} = \phi^{\ast} \phi$ hence is normal. Here $\phi^{\ast}$ stands for the conjugate transpose of $\phi$. So around each given point, by a unitary change of $\{ \varphi_1 , \ldots , \varphi_{n-1}\}$ with $\varphi_n$ fixed, we could always make $\phi$ diagonal. This gives us an admissible local unitary frame.

As $\nabla^b T^c=0$, the tensor $\phi_i^j$ is also $\nabla^b$-parallel, which implies the eigenvalues $\{\lambda a_i\}_{i=1}^n$ of $\phi$ are global constants with $a_n=0$. It is also obvious that the eigenspaces of $\phi$ with respect to all the distinct eigenvalues lead to a $\nabla^b$-parallel and $g$-orthogonal decomposition of $T^{1,0}M$.
\end{proof}

\begin{remark}
It is easy to verify that, for non-balanced BTP manifolds, the distribution $H_0 \oplus \overline{H}_0$ is a foliation since $[\chi,\overline{\chi}]=0$, which is known as canonical foliation in the study of Vaisman manifolds.
\end{remark}

The following notation of degenerate torsion, introduced in \cite[Defintion 5]{YZZ}, could also be applied to non-balanced {\em BTP} manifolds.

\begin{definition}\label{degtor_def}
A non-balanced {\em BTP} manifold $(M^n,g)$ is said to have {\em degenerate torsion},  if under any admissible frame $e$, $T^{\ast }_{ik}=0$ for any $i,k <n$.
\end{definition}

The {\em Lee potential} ({\em LP}) condition was introduced by Belgun \cite{B12} to study {\em generalized Calabi-Eckmann} ({\em GCE}) structures (see also \cite{B00, B03}).

\begin{definition}[\cite{B12}]\label{LP_GCE}
A Hermitian manifold $(M^n,g)$ is {\em LP} if the Gauduchon torsion $1$-form $\eta$ satisfies
\[\partial \eta =0 , \quad \partial \omega = c \,\eta \wedge\partial \overline{\eta},\]
where $c\neq 0$ is a constant. A Hermitian manifold is {\em GCE} if it is {\em LP} and {\em BTP}.
As shown in \cite[Proposition 3.2]{B12}, both the standard Hermitian structure and the modified ones on the Sasakian product are non-K\"ahler {\em GCE} manifolds.
\end{definition}

\begin{proof}[{\bf Proof of Proposition \ref{prop1.6}:}] By Corollary \ref{Tparallel_add}, we know that non-balanced {\em BTP} manifolds always admit local admissible frames. Under such a frame $e$, the vector $e_n=\frac{1}{\lambda}\chi$ is parallel under $\nabla^b$, so the connection matrix $\theta^b$ of $\nabla^b$ satisfies $\theta^b_{n\ast}=\theta^b_{\ast n}=0$. Therefore, the curvature matrix $\Theta^b$ under $e$ will also be block-diagonal:
$$ \Theta^b = \left[ \begin{array}{cc} \ast & 0 \\ 0 & 0 \end{array} \right] $$
where $\ast$ is an $(n\!-\!1)\times (n\!-\!1)$ matrix. As is well-known, the Lie algebra of the holonomy group $\mbox{Hol}(\nabla^b)$ of $\nabla^b$ is generated by conjugation class of $R^b_{xy}=\Theta^b(x,y)$ under $\nabla^b$-parallel transports, all of them operates in the $\nabla^b$-parallel block $e_n^{\perp}$, so $\mbox{Hol}(\nabla^b) \subseteq U(n-1)$. This proves the first part of Proposition \ref{prop1.6}.

For the second part, under an admissible frame $e$, it follows from Proposition \ref{prop3.6} that
\[\begin{aligned}
\eta \wedge \partial \overline{\eta}\  &= \ 2\lambda^2 \varphi_n \wedge \sum_{i,j} T_{i n}^j \varphi_i \wedge \overline{\varphi_j} \\
\ &= \ -2\lambda^2 \sum_{i < n} T_{i n}^i \varphi_i \wedge\varphi_n \wedge \overline{\varphi_i},\\
-\sqrt{-1}\partial \omega &= ^t\!\!\tau \wedge \overline{\varphi} = \sum_{i,j,k} T_{ik}^j \varphi_i \wedge \varphi_k \wedge \overline{\varphi_j} \\
\ &= \ 2\sum_{\begin{subarray}{c}i<k\\j<n\end{subarray}}T^j_{ik}\varphi_i \wedge \varphi_k \wedge \overline{\varphi_j}.
\end{aligned}\]
It is clear that the condition $\partial \omega = c \,\eta \wedge \partial \overline{\eta}$ implies $T^{j }_{ik}=0$ for any $i,k <n$.
Conversely, the torsion degeneracy condition $T^{j }_{ik}=0$ for any $i,k <n$ yields
\[-\sqrt{-1}\partial \omega =2\sum_{\begin{subarray}{c}i<n\\j<n\end{subarray}} T^j_{in}\varphi_i \wedge \varphi_n \wedge \overline{\varphi_j}
=2\sum_{i<n} T^i_{in}\varphi_i \wedge \varphi_n \wedge \overline{\varphi_i} \]
by the definition of admissible frames, thus $\partial \omega = c \,\eta \wedge \partial \overline{\eta}$, which shows that the {\em LP} condition is equivalent to the torsion degeneracy condition. To prove the last statement, namely when $n=3$ the set of non-balanced {\em BTP} threefolds is equal to the set of {\em GCE} threefolds, it suffices to show that any non-balanced {\em BTP} threefold $(M^3,g)$ always has degenerate torsion. Let $e$ be a local admissible frame on $M^3$. By definition, it follows that $\eta_1=\eta_2=0$, $\eta_3=\lambda >0$, and $T^1_{23}=T^2_{13}=0$ since $\phi$ is diagonal. The equality (\ref{eq:5.1}) implies $T^3_{ij}=0$ for any $i,j$. Then it yields that
$$ T^1_{12} = \eta_2- T^3_{32} = 0, \ \ \ T^2_{12} = - T^2_{21} = -\eta_1 + T^3_{31} = 0.$$
That is, $T^{\ast}_{12}=0$, so the manifold has degenerate torsion by definition. This completes the proof of Proposition \ref{prop1.6}.
\end{proof}

\begin{remark}
In light of the proposition above, a BKL manifold with degenerate torsion is exactly a Hermitian manifold which is both BKL and GCE. A complete non-K\"ahler BKL manifold with degenerate torsion will split off a K\"ahler factor, on the universal cover, of complex codimension $2$ or $3$, as in \cite[Theorem 9]{YZZ}.
\end{remark}

In order to get the classification theorem of non-balanced {\em BTP} threefolds, we recall here the following definitions, including \textit{Sasakian} manifolds, the standard Hermitian structure and the modified ones on the Sasakian product.

\begin{definition}\label{SSK}
Let $(N^{2m+1},g,\xi)$ be a \textit{Sasakian} manifold, that is, as defined in \cite{YZZ} for instance, an odd dimensional Riemannian manifold $(N,g)$ equipped with a Killing vector field $\xi$ of unit length, such that:
\begin{enumerate}
\item The tensor field $\frac{1}{c}\nabla \xi$ (where $c>0$ is a constant), which sends a tangent vector $X$ to the tangent vector $\frac{1}{c}\nabla_X\xi$, gives an integrable orthogonal complex structure $I$ on the distribution $H$, where $H$ is the perpendicular complement of $\xi $ in the tangent bundle $TN$.
\item Denote by $\phi $ the $1$-form dual to $\xi$, namely, $\phi (X) = g(X, \xi)$ for any $X$, then $\phi \wedge (d\phi )^m$ is nowhere zero. That is, $\phi$ gives a contact structure on $N$.
\end{enumerate}
\end{definition}

\begin{definition}\label{SSKproduct}
Let $(N_1^{2n_1+1}, g_1, \xi_1)$ and $(N_2^{2n_2+1}, g_2, \xi_2)$ be two Sasakian manifolds. As introduced by Belgun \cite{B12}, on the product Riemannian manifold $M=N_1\times N_2$, of even dimension $2n=2n_1+2n_2+2$, for $\kappa \in \{\kappa \in \mathbb{C} \big| \mathrm{Im} (\kappa) \neq 0\}$, the natural almost complex structure $J_{\kappa}$ is defined by
\begin{gather*}
J_{\kappa}\xi_1 = \mathrm{Re}(\kappa) \xi_1 + \mathrm{Im}(\kappa)\xi_2, \\
J_{\kappa}X_i = J_i X_i, \quad \forall \ X_i \in H_i , \ i=1,2.
\end{gather*}
Define the associated metric $g_{\kappa}$ as
\begin{gather*}
g_{\kappa}(\xi_1,\xi_1)= g_{\kappa}(J_{\kappa}\xi_1,J_{\kappa}\xi_1)=1, \quad g_{\kappa}(\xi_1,J_{\kappa}\xi_1)=0,\\
g_{\kappa}(\xi_i,X)=0,\quad \forall \ X \in H_1 \oplus H_2, \ i=1,2.\\
g_{\kappa}(X_i,Y_i)=g_i(X_i,Y_i), \quad g_{\kappa}(X_1,X_2)=0,\quad \forall \ X_i, Y_i \in H_i,\ i=1,2.
\end{gather*}
It is easy to verify that $g_{\kappa}$ is a $J_{\kappa}$-Hermitian metric, $J_{\kappa}$ is integrable, and when $\kappa=\sqrt{-1}$, the Hermitian manifold defined above is the standard Hermitian structure on the Sasakian product. We will call the Hermitian manifold $(N_1\times N_2, g_{\kappa}, J_{\kappa})$ {\em twisted Sasakian product,} and denote it by $N_1 \times_{\kappa} N_2$.
\end{definition}
\begin{lemma}\label{twisted}
Let $(N_1,g_1,\xi_1),(N_2,g_2,\xi_2)$ be two $3$-dimensional Sasakian manifolds and $\kappa = x + y \sqrt{-1}$ for $x,y \in \mathbb{R},y \neq 0$. Then there exists a natural unitary frame $(e_1,e_2,e_3)$, with dual frame $(\varphi_1,\varphi_2,\varphi_3)$, on the twisted Sasakian product
$N_1 \times_{\kappa} N_2$, such that the Riemannian connection $\nabla$ under $e$ takes the form
\begin{equation}\label{eq_twSSK}
\begin{cases}
\nabla e_1 = c_1\sqrt{-1} \overline{\varphi}_1\, \xi_1 + \sigma_1\, e_1 \\
\nabla e_2 = c_2\sqrt{-1} \overline{\varphi}_2\, \xi_2 + \sigma_2\, e_2 \\
\nabla e_3 = \frac{c_1\sqrt{-1}}{\sqrt{2}}(\varphi_1\, e_1 - \overline{\varphi}_1 \, \overline{e}_1)
+ \frac{c_2 (1-x\sqrt{-1})}{y\sqrt{2}} (\varphi_2\, e_2 - \overline{\varphi}_2 \, \overline{e}_2),
\end{cases}
\end{equation}
where $e_3 = \frac{1}{\sqrt{2}}(\xi_1 - \sqrt{-1}(x\xi_1 + y \xi_2))$, $c_i$ is the defining constant of the Sasakian manifold $N_i$, and $\sigma_i$ is a $1$-form satisfying $\overline{\sigma}_i = - \sigma_i$, for $i=1,2$.
\end{lemma}
\begin{proof}
From Definition \ref{SSK}, for $i=1,2$, we may assume that $(\xi_i,e_i,\overline{e}_i)$ be a frame of the Sasakian manifold $(N_i, g_i, \xi_i)$ such that $e_i$ is a unitary frame of the distribution $\xi_i^{\bot}$ with respect to the orthogonal complex structure $\frac{1}{c_i} \nabla \xi_i$, which yields the following structure equation
\[ \nabla \! \begin{bmatrix} \xi_i \\ e_i \\ \overline{e}_i \end{bmatrix}=
\begin{bmatrix} 0 & c_i \sqrt{-1} \varphi_i & - c_i \sqrt{-1} \overline{\varphi}_i \\
c_i \sqrt{-1} \overline{\varphi}_i & \sigma'_i & 0  \\
-c_i \sqrt{-1} \varphi_i  & 0 & \overline{\sigma}'_i  \\ \end{bmatrix}\!\!
\begin{bmatrix} \xi_i \\ e_i \\ \overline{e}_i \\ \end{bmatrix}\!\!,\]
where $\sigma'_i$ is a $1$-form satisfying $\overline{\sigma}'_i = - \sigma'_i$ and the dual frame of $(\xi_i,e_i,\overline{e}_i)$ is denoted by $(\phi_i,\varphi_i,\overline{\varphi}_i)$. The dual version of the equation above is
\[ \nabla \! \begin{bmatrix} \phi_i \\ \varphi_i \\ \overline{\varphi}_i \end{bmatrix}=
\begin{bmatrix} 0 & - c_i \sqrt{-1} \overline{\varphi}_i & c_i \sqrt{-1} \varphi_i \\
-c_i \sqrt{-1} \varphi_i & \overline{\sigma}'_i & 0 \\
c_i \sqrt{-1} \overline{\varphi}_i & 0 & \sigma'_i  \\ \end{bmatrix}\!\!
\begin{bmatrix} \phi_i \\ \varphi_i \\ \overline{\varphi}_i \\ \end{bmatrix}\!\!.\]
As $\nabla$ is torsion free, it follows that
\[ d \!  \begin{bmatrix} \phi_i \\ \varphi_i \\ \overline{\varphi}_i \end{bmatrix}=
\begin{bmatrix} 0 & - c_i \sqrt{-1} \overline{\varphi}_i & c_i \sqrt{-1} \varphi_i \\
-c_i \sqrt{-1} \varphi_i & \overline{\sigma}'_i & 0 \\
c_i \sqrt{-1} \overline{\varphi}_i & 0 & \sigma'_i  \\ \end{bmatrix}\!\!
\begin{bmatrix} \phi_i \\ \varphi_i \\ \overline{\varphi}_i \\ \end{bmatrix}\!\!.\]

From Definition \ref{SSKproduct}, we may define
\begin{gather*}
e_3 = \frac{1}{\sqrt{2}}(\xi_1 - \sqrt{-1}J_{\kappa}\xi_1 ) = \frac{1}{\sqrt{2}}\big((1-x\sqrt{-1})\xi_1 - y\sqrt{-1} \xi_2 \big),\\
\varphi_3 = \frac{1}{\sqrt{2}}(\phi_1 - \frac{x}{y} \phi_2 + \frac{\sqrt{-1}}{y} \phi_2 ) =  \frac{1}{\sqrt{2}}(\phi_1 - \frac{x-\sqrt{-1}}{y} \phi_2 ).
\end{gather*}
It is easy to verify that $(e_1,e_2,e_3)$ is a unitary frame on the twisted Sasakian product $N_1 \times_{\kappa} N_2$ and  $(\varphi_1,\varphi_2,\varphi_3)$ is its dual coframe. Then it yields that
\[ \begin{cases}
d \varphi_1 = \frac{c_1}{\sqrt{2}}\big( (x + \sqrt{-1}) \varphi_3 - (x-\sqrt{-1}) \overline{\varphi}_3 \big)\varphi_1
+ \overline{\sigma}'_1 \varphi_1, \\
d \varphi_2 = \frac{c_2y}{\sqrt{2}}(\varphi_3 - \overline{\varphi}_3) \varphi_2 + \overline{\sigma}'_2 \varphi_2, \\
d \varphi_3 = \sqrt{2} c_1 \sqrt{-1} \varphi_1 \overline{\varphi}_1 - \frac{\sqrt{2}c_2 (1+x\sqrt{-1})}{y} \varphi_2 \overline{\varphi}_2.
\end{cases} \]
In matrix form we have $d \varphi = \overline{\theta}_1 \varphi + \theta_2 \overline{\varphi}$, where $\theta_1$ is skew Hermitian and $\theta_2$ is skew symmetric, given by
\begin{gather*}
\theta_1= \begin{bmatrix} \sigma_1 & 0 & \frac{c_1\sqrt{-1}}{\sqrt{2}} \overline{\varphi}_1  \\
                            0 & \sigma_2 & \frac{-c_2(1+x\sqrt{-1})}{y\sqrt{2}} \overline{\varphi}_2 \\
                            \frac{c_1\sqrt{-1}}{\sqrt{2}} \varphi_1 & \frac{c_2(1-x\sqrt{-1})}{y\sqrt{2}} \varphi_2 & 0 \\
                            \end{bmatrix}\!\!,\quad
\theta_2 = \begin{bmatrix} 0 & 0 & \frac{-c_1\sqrt{-1}}{\sqrt{2}} \varphi_1 \\
                           0 & 0 & \frac{c_2(1+x\sqrt{-1})}{y\sqrt{2}} \varphi_2 \\
                           \frac{c_1 \sqrt{-1}}{\sqrt{2}} \varphi_1 & \frac{-c_2(1+x\sqrt{-1})}{\sqrt{2}} \varphi_2 & 0 \\ \end{bmatrix}\!\!,\\
\sigma_1 = \sigma_1'-\frac{c_1 x}{\sqrt{2}}(\varphi_3-\overline{\varphi}_3),\qquad
\sigma_2 = \sigma_2' - \frac{c_2}{\sqrt{2}}\big((y-\frac{1-x\sqrt{-1}}{y})\varphi_3 - (y - \frac{1+x\sqrt{-1}}{y}) \overline{\varphi}_3 \big).
\end{gather*}
Note that $\nabla e = \theta_1 e + \overline{\theta}_2 \overline{e}$, $\xi_1 = \frac{e_3 + \overline{e}_3}{\sqrt{2}}$ and $\sqrt{-1}\xi_2 =  \frac{-(1+x\sqrt{-1})}{y\sqrt{2}} e_3 +  \frac{(1-x\sqrt{-1})}{y\sqrt{2}} \overline{e}_3$. Then the equation \eqref{eq_twSSK} in the lemma is established.
\end{proof}

\begin{proof}[{\bf Proof of Theorem \ref{3DNBBTP}:}]
It follows, from Proposition \ref{prop1.6} and its proof, that admissible frames, denoted by $e$, can be established on the non-balanced {\em BTP} threefold, and the connection matrix $\theta^b$ of Bismut connection $\nabla^b$ under an admissible frame $e$  satisfies $\theta^b_{3\ast} = \theta^b_{\ast 3}=0$. Then we may write $\lambda_1 = \lambda a$ and $\lambda_2 = \lambda b$. As the admissible frame is applied and the torsion is degenerate when $n=3$ by Proposition \ref{prop1.6}, it yields that the only possibly nonzero components of the Chern torsion are
\[T_{13}^1=a,\quad T^2_{23}=b, \quad \mbox{with} \ \ a+b=\lambda .\]
In  case \eqref{twSSK},  clearly $a,b\neq 0$ and $a \neq b$. This implies that the connection matrix $\theta^b$ is diagonal under $e$, denoted by
\[\theta^b = \begin{bmatrix} \psi_1 & 0 & 0 \\
                               0    & \psi_2 & 0 \\
                               0  & 0 & 0 \\ \end{bmatrix}\!\!,\]
where $\overline{\psi}_1 = - \psi_1$ and $\overline{\psi}_2 = - \psi_2$.
Since the connection matrix of the Riemannian connection $\nabla$ is given by $\nabla e = \theta_1 e + \overline{\theta}_2 \overline{e}$, where $\theta_1 = \theta^b - \gamma$ and $\gamma,\theta_2$ are determined by the Chern torsion $T$, we have
\begin{equation}\label{generic}
\begin{cases}
\nabla e_1 = \overline{\varphi}_1 \, (a \overline{e}_3 - \overline{a} e_3 ) + \sigma_1 \, e_1 \\
\nabla e_2 = \overline{\varphi}_2\,(b \overline{e}_3 - \overline{b} \overline{e}_3) + \sigma_2 \, e_2 \\
\nabla e_3 = a (\varphi_1\, e_1 - \overline{\varphi}_1 \, \overline{e}_1)
+ b (\varphi_2\, e_2 - \overline{\varphi}_2 \, \overline{e}_2),
\end{cases}
\end{equation}
where $\varphi$ is the dual coframe of $e$, and $\sigma_1 = \psi_1 - a \varphi_3 + \overline{a} \overline{\varphi}_3$, $\sigma_2 = \psi_2 - b \varphi_3 + \overline{b} \overline{\varphi}_3$. Then we may define two global real vector fields of unit length by
\[ \xi_1 = \frac{-\sqrt{-1}}{\sqrt{2}|a|}(a \overline{e}_3 - \overline{a} e_3 ),\quad
\xi_2 = \frac{- \sqrt{-1}}{\sqrt{2}|b|}(b \overline{e}_3 - \overline{b} e_3 ),\]
where it is easy to verify from the condition of the case \eqref{twSSK} that $\mbox{span}_{\mathbb{C}} \{e_3,\overline{e}_3\} = \mbox{span}_{\mathbb{C}} \{\xi_1,\xi_2\}$. It follows that
\[e_3 = \frac{\sqrt{2}|a|b \sqrt{-1}}{a\overline{b}-\overline{a}b} \xi_1
- \frac{\sqrt{2} |b| a \sqrt{-1}}{a\overline{b} - \overline{a} b} \xi_2,\quad
J \xi_1 = \frac{-(a\overline{b}+\overline{a}b)\sqrt{-1}}{a\overline{b}-\overline{a}b} \xi_1
+ \frac{2|a||b|\sqrt{-1}}{a\overline{b}-\overline{a}b} \xi_2. \]
Let us define
\begin{gather*}
x=\frac{-(a\overline{b}+\overline{a}b)\sqrt{-1}}{a\overline{b}-\overline{a}b},\quad
y=\frac{2|a||b|\sqrt{-1}}{a\overline{b}-\overline{a}b}, \quad \kappa = x + y\sqrt{-1}, \\
e'_3 = \frac{1}{\sqrt{2}}(\xi_1 - \sqrt{-1}J\xi_1)= \frac{1}{\sqrt{2}}\big((1-x\sqrt{-1})\xi_1 - y\sqrt{-1} \xi_2 \big).
\end{gather*}
It is easy to see that $x,y$ are real constants and $e'_3=\frac{\sqrt{-1} \overline{a}}{|a|}e_3$.
Then the first two equations of \eqref{generic} enable us to contruct two 3-dimensional Sasakian manifolds $N_1$ and $N_2$ by the following two structure equations respectively,
\[ \nabla \! \begin{bmatrix} \xi_1 \\ e_1 \\ \overline{e}_1 \end{bmatrix}=
\begin{bmatrix} 0 & \sqrt{2}|a| \sqrt{-1} \varphi_1 & - \sqrt{2}|a| \sqrt{-1} \overline{\varphi}_1 \\
\sqrt{2}|a| \sqrt{-1} \overline{\varphi}_1 & \sigma'_1 & 0  \\
-\sqrt{2}|a| \sqrt{-1} \varphi_1  & 0 & \overline{\sigma}'_1  \\ \end{bmatrix}\!\!
\begin{bmatrix} \xi_1 \\ e_1 \\ \overline{e}_1 \\ \end{bmatrix}\!\!,\]
\[ \nabla \! \begin{bmatrix} \xi_2 \\ e_2 \\ \overline{e}_2 \end{bmatrix}=
\begin{bmatrix} 0 & \sqrt{2}|b| \sqrt{-1} \varphi_2 & - \sqrt{2}|b| \sqrt{-1} \overline{\varphi}_2 \\
\sqrt{2}|b| \sqrt{-1} \overline{\varphi}_2 & \sigma'_2 & 0  \\
-\sqrt{2}|b| \sqrt{-1} \varphi_2  & 0 & \overline{\sigma}'_2  \\ \end{bmatrix}\!\!
\begin{bmatrix} \xi_2 \\ e_2 \\ \overline{e}_2 \\ \end{bmatrix}\!\!,\]
where the dual coframe of $(\xi_i,e_i,\overline{e}_i)$ is denoted by $(\phi_i,\varphi_i,\overline{\varphi}_i)$ for $i=1,2$, and we may also define
\begin{gather*}
\varphi'_3 = \frac{1}{\sqrt{2}}(\phi_1 - \frac{x}{y} \phi_2 + \frac{\sqrt{-1}}{y} \phi_2 ) =  \frac{1}{\sqrt{2}}(\phi_1 - \frac{x-\sqrt{-1}}{y} \phi_2 ),\\
\sigma_1'= \sigma_1 + x|a|(\varphi'_3 - \overline{\varphi'_3}),\quad \sigma'_2 = \sigma_2 + |b| \big( (y- \frac{1-x\sqrt{-1}}{y}) \varphi'_3 - (y-\frac{1+x\sqrt{-1}}{y}) \overline{\varphi'_3} \big).
\end{gather*}
Here $\sigma_1, \sigma_2$ are coefficient forms from the equation \eqref{generic}. Then it follows from Lemma \ref{twisted} that the structure equation of $N_1 \times_{\kappa} N_2$ is exactly the one \eqref{generic}, after we note that $c_1 = \sqrt{2}|a|, c_2 = \sqrt{2}|b|, e'_3=\frac{\sqrt{-1} \overline{a}}{|a|}e_3$ here. Therefore we have shown the case \eqref{twSSK}.

In case \eqref{BKL}, either  $a=0$ and $b=\lambda$, or $a=\frac{\lambda}{2}(1+\rho)$, $b=\frac{\lambda}{2}(1-\rho)$ for some $\rho$ satisfying $|\rho|=1$ and $\mathrm{Im}(\rho)>0$. It is easy to verify that $P_{ik}^{j \ell}=0$ under any admissible frame $e$ in this case, thus $(M^3,g)$ is {\em BKL} by Proposition \ref{prop3.6}.

In case \eqref{gV}, both  $a$ and $b$ are real nonzero constants with $b=\lambda-a$. It follows from Proposition \ref{prop3.6} that $d(\eta + \overline{\eta})=0$ and thus $(M^3,g)$ is generalized Vaisman. It is easy to verify that, when $\lambda_1=\lambda_2$, i.e. when $a=b=\frac{\lambda}{2}$, it holds that $d \omega = - (\eta + \overline{\eta}) \omega$, which implies that $(M^3,g)$ is Vaisman by Proposition \ref{prop1.4}. This completes the proof of Theorem \ref{3DNBBTP}.
\end{proof}

\section{Balanced {\em BTP} threefolds}\label{BS3D}

Let $(M^3,g)$ be a compact non-K\"ahler balanced {\em BTP} manifold. We will start with a technical observation in \cite[Proposition 2]{ZhouZ} which says that any balanced threefold always admits a particular type of unitary frame under which the Chern torsion takes a simple form, namely, for any given $p\in M^3$, there exists a unitary frame $e$ near $p$ such that the Chern torsion components satisfy
\begin{equation}
T^1_{1k}=T^2_{2k}=T^3_{3k}=0, \ \ \ \forall \ 1\leq k\leq 3. \label{eq:6.1}
 \end{equation}
Under such a frame, the only possibly non-zero torsion components are
\begin{equation}
 a_1=T^1_{23}, \ \ \ a_2=T^2_{31}, \ \ \ a_3=T^3_{12}. \label{eq:6.2}
 \end{equation}
By a rotation and a permutation of $\{ e_1, e_2, e_3\}$ if needed, we may assume that $|a_1|\geq |a_2|\geq |a_3|$.

First we claim that each $|a_i|$ is a global constant on $M^3$. To see this, recall that tensor $B$ is defined by $B_{i\overline{j}}=\sum_{r,s} T^j_{rs}\overline{T^i_{rs}}$. Under our $e$, we have
$$ B = 2 \begin{bmatrix} |a_1|^2 & 0 & 0 \\ 0 & |a_2|^2 & 0 \\ 0 & 0 & |a_3|^2 \end{bmatrix}\!\!.$$
By the {\em BTP} assumption, we have $\nabla^bB=0$, so the eigenvalues of $B$ are all global constants on $M$, thus we know that each $|a_i|$ is a global constant.

Then one can rotate $e$ to make  $a_i=|a_i|$ for each $i$. To see this, since $|a_i|$ is a constant, we can write $a_i=\rho_i|a_i|$ where $\rho_i$ is a smooth local function with $|\rho_i|=1$. We will simply let $\rho_i=1$ if $a_i=0$.
Replace each $e_i$ by $\tilde{e}_i =(\overline{\rho}_j \overline{\rho}_k)^{\frac{1}{2}}e_i$, where $(ijk)$ is a cyclic permutation of $(123)$. Then under the new frame we will have $a_i=|a_i|$ for each $i$. An appropriate permutation of $\{ e_1, e_2, e_3\}$ yields
\begin{equation}
T^1_{23}\geq T^2_{31} \geq T^3_{12}\geq 0. \label{eq:6.3}
 \end{equation}
We will call a local unitary frame $e$ on $M^3$ a {\em special frame} if both (\ref{eq:6.1}) and (\ref{eq:6.3}) are satisfied.

Let us fix a special frame $e$ in a neighborhood of $p\in M^3$. We have $a_1\geq a_2\geq a_3\geq 0$, with $a_1>0$ since $M^3$ is not K\"ahler. Denote by $\theta^b$ the matrix of the Bismut connection $\nabla^b$ under $e$. Since $\nabla^bT=0$, and all $T^j_{ik}$ are constants, we have
\begin{equation}
0 = dT^j_{ik} =  \sum_r \big( T^j_{rk} \theta^b_{ir} + T^j_{ir} \theta^b_{kr} - T^r_{ik} \theta^b_{rj} \big)  \label{eq:6.4}
\end{equation}
Since the only possibly non-zero components of $T$ are $a_1$, $a_2$ and $a_3$, if we take $i$, $j$, $k$ all distinct in (\ref{eq:6.4}), we get
\begin{equation} \label{eq:6.5}
\left\{
 \begin{array}{c}
 a_1(\theta^b_{22}+\theta^b_{33} - \theta^b_{11} )= 0  \\
 a_2(\theta^b_{11}+\theta^b_{33} - \theta^b_{22} )= 0  \\
a_3( \theta^b_{11}+\theta^b_{22} - \theta^b_{33} ) = 0
\end{array}
\right.
\end{equation}
Similarly, by taking $j=i \neq k$ in (\ref{eq:6.4}), we get
\begin{equation}
a_1\theta^b_{12}+a_2\theta^b_{21} = a_1\theta^b_{13}+a_3\theta^b_{31} = a_2\theta^b_{23}+a_3\theta^b_{32} = 0. \label{eq:6.6}
\end{equation}
Denote by $\varphi$ the unitary coframe dual to the special frame $e$, and by $\theta$, $\tau$ the matrix of connection and column vector of torsion under $e$ for the Chern connection $\nabla^c$. By \cite{YZ18Gau}, we have $\theta = \theta^b-2\gamma$ where
\begin{equation} \label{eq:6.7}
\gamma = \begin{bmatrix} 0 & -\overline{\psi}_3  & \psi_2 \\ \psi_3 & 0 & -\overline{\psi}_1 \\ - \overline{\psi}_2 & \psi_1 &  0 \end{bmatrix}\!\!, \ \ \ \ \ \tau = \begin{bmatrix} 2a_1 \varphi_2\varphi_3 \\ 2a_2 \varphi_3\varphi_1 \\ 2a_3 \varphi_1\varphi_2  \end{bmatrix}\!\!,
\end{equation}
\begin{equation} \label{eq:6.8}
\psi_1 = a_2\varphi_1+a_3\overline{\varphi}_1, \ \ \ \psi_2 = a_3\varphi_2+a_1\overline{\varphi}_2, \ \ \  \psi_3 = a_1\varphi_3+a_2\overline{\varphi}_3.
\end{equation}
We shall divide the classification into the following three cases:
\begin{enumerate}
\item\label{thr_dst} $a_1>a_2>a_3$;
\item\label{thr_eql} $a_1=a_2=a_3$;
\item\label{tw_eql} $a_1=a_2>a_3$ or $a_1>a_2=a_3$,
\end{enumerate}
where the third case contains the {\em middle type} and is the main part of the discussion.

\begin{proof}[\textbf{Proof of Proposition \ref{Btype}}]
\mbox{}
\vspace{0.15cm}

\noindent {\bf Case \ref{thr_dst}:} $a_1>a_2>a_3$.

\vspace{0.15cm}

In this case, $B$ has distinct eigenvalues. Since its eigenspaces are all $\nabla^b$-parallel, we know that the matrix $\theta^b$ is diagonal.

If $a_3>0$, by (\ref{eq:6.5}) we get $\theta^b_{ii}=0$ for each $i$, hence $\theta^b=0$. This means that $M^3$ is Bismut flat. Such a manifold cannot be balanced unless it is K\"ahler, contradicting to our assumption that $M^3$ is balanced and non-K\"ahler, so we must have $a_3=0$, which implies $\psi_1=a_2\varphi_1$ and $\psi_2=a_1\overline{\varphi}_2$. From Equation (\ref{eq:6.5}), we get
\begin{equation*}
 \theta^b_{33}=0, \ \ \  \theta^b_{11}=\theta^b_{22} = \alpha ,
\end{equation*}
and
$$ \theta = \theta^b-2\gamma =  \begin{bmatrix} \alpha  & 2\overline{\psi}_3 & -2a_1\overline{\varphi}_2 \\ -2\psi_3 & \alpha & 2a_2 \overline{\varphi}_1 \\ 2a_1\varphi_2 & -2a_2\varphi_1 & 0 \end{bmatrix}\!\!. $$
Then the structure equation of Chern connection yields
\begin{equation}
 d\varphi  =  -\,^t\!\theta \wedge \varphi + \tau = \begin{bmatrix} -\alpha \varphi_1 + 2\psi_3 \varphi_2 \\ - \alpha \varphi_2 - 2\overline{\psi}_3 \varphi_1  \\ 2a_2 \varphi_2 \overline{\varphi}_1 - 2a_1 \varphi_1 \overline{\varphi}_2   \end{bmatrix} \!\!. \label{eq:6.9}
 \end{equation}
By the first two equations of (\ref{eq:6.9}), we get
\begin{eqnarray*}
 d(\varphi_2 \overline{\varphi}_1 ) &  = & d\varphi_2 \wedge \overline{\varphi}_1 - \varphi_2 \wedge d\overline{\varphi}_1  \\
 & = & (-\alpha \varphi_2 - 2\overline{\psi}_3 \varphi_1 )\,\overline{\varphi}_1 - \varphi_2 (-\overline{\alpha} \,\overline{\varphi}_1 + 2 \overline{\psi}_3 \overline{\varphi}_2 ) \\
 & = & 2 \overline{\psi}_3 \,( \varphi_2 \,\overline{\varphi}_2 - \varphi_1\,\overline{\varphi}_1),
\end{eqnarray*}
where $\alpha + \overline{\alpha} =0$ is used. Complex conjugation of the equality above yields
$$  d(\varphi_1 \overline{\varphi}_2 ) = -  \overline{ d(\varphi_2 \overline{\varphi}_1 )} = 2 \psi_3 \, ( \varphi_2 \,\overline{\varphi}_2 - \varphi_1\,\overline{\varphi}_1) .$$
The exterior differentiation of the third equation of (\ref{eq:6.9}) implies
$$ 0 = d^2\varphi_3 = 2a_2 \,d (\varphi_2 \overline{\varphi}_1) - 2a_1 \,d(\varphi_1 \overline{\varphi}_2) = 4(a_2\overline{\psi}_3 - a_1\psi_3) \wedge (\varphi_2 \,\overline{\varphi}_2 - \varphi_1\,\overline{\varphi}_1) .$$
Note that
$$ a_2\overline{\psi}_3 - a_1\psi_3 = a_2(a_1 \overline{\varphi}_3 + a_2 \varphi_3) - a_1 (a_1 \varphi_3 + a_2 \overline{\varphi}_3 ) = (a_2^2-a_1^2) \varphi_3, $$
which yields a contradiction. This shows that the case of distinct $a_1, a_2$ and $a_3$ cannot occur.
\vspace{0.15cm}

\noindent {\bf Case \ref{thr_eql}:} $a_1=a_2=a_3$.

\vspace{0.15cm}

Let us denote by $a>0$ the common value of those $a_i$ in this case. Equalities (\ref{eq:6.5}) and (\ref{eq:6.6}) now imply that $\theta^b$ is skew-symmetric. Since $\psi_i=a(\varphi_i+\overline{\varphi}_i) = \overline{\psi}_i$ for each $i$, the equality (\ref{eq:6.7}) shows that $\gamma$ is also skew-symmetric. So $\theta= \theta^b-2\gamma$ is skew-symmetric. Then it yields that
\begin{equation} \label{eq:6.10}
 \theta = \begin{bmatrix} 0 & x  & y \\ -x & 0 & z \\ - y & -z &  0 \end{bmatrix}\!\!, \ \ \ \ \tau = 2a \begin{bmatrix} \varphi_2\varphi_3 \\ \varphi_3\varphi_1 \\ \varphi_1\varphi_2 \end{bmatrix}\!\!,
 \end{equation}
where $x$, $y$, $z$ are real $1$-forms. As a result, the Chern curvature matrix $\Theta =d\theta - \theta \wedge \theta $ is also skew-symmetric. The structure equation of Chern connection gives us
\begin{equation*}
 d \varphi = - \,^t\!\theta \wedge \varphi + \tau =
\begin{bmatrix} x \varphi_2 + y \varphi_3 +2a\varphi_2\varphi_3 \\ -x \varphi_1 + z \varphi_3 +2a\varphi_3\varphi_1 \\  -y \varphi_1 -z \varphi_2 +2a\varphi_1\varphi_2  \end{bmatrix}\!\!.
\end{equation*}
It follows that
\begin{eqnarray*}
d (\varphi_2\varphi_3) & = & (x\varphi_3-y\varphi_2)\varphi_1 \\
d (\varphi_3\varphi_1) & = & (x\varphi_3+ z\varphi_1)\varphi_2 \\
d (\varphi_2\varphi_3) & = & (-y\varphi_2+z\varphi_1)\varphi_3.
\end{eqnarray*}
On one hand, if we let $\xi = x\varphi_3 -y\varphi_2 + z\varphi_1$, then the above equations simply says $d\tau = 2a \,\xi \wedge \varphi$. On the other hand, by (\ref{eq:6.10}) we get $\theta \tau = 2a\,\xi \wedge \varphi$. So by the first Bianchi identity
$$ d\tau = \,^t\!\Theta \varphi - \,^t\!\theta \tau, $$
we conclude that $ \,^t\!\Theta \varphi = 0$. This means that the Hermitian threefold $M^3$ is Chern K\"ahler-like, that is, the Chern curvature tensor $R^c$ obeys the K\"ahler symmetry $R^c_{i\overline{j}k\overline{\ell}} = R^c_{k\overline{j}i\overline{\ell}} $ for any $i,j,k,\ell$.

We claim that $R^c=0$. If $\{ i,k\} \cap \{ j,\ell\} \neq \emptyset$, for instance $1$ is contained in the intersection, then by the K\"ahler symmetry, $R^c_{i\overline{j}k\overline{\ell}}$ can be written as $R^c_{a\overline{b}1\overline{1}}$, which has to vanish since $\Theta_{11}=0$ as $\Theta$ is skew-symmetric. When $\{ i,k\} \cap \{ j,\ell\} = \emptyset$, what we need to show are the equalities $R_{i\overline{j}k\overline{j}}=0$ and $R_{i\overline{j}i\overline{j}}=0$ where $i,j,k$ are distinct, as the dimension is $3$. From the skew-symmetric $\Theta$, it yields that
$$ R^c_{i\overline{j}k\overline{j}} = - R^c_{i\overline{j}j\overline{k}} = -R^c_{i\overline{k}j\overline{j}}=0,$$
$$ R^c_{i\overline{j}i\overline{j}} = - R^c_{i\overline{j}j\overline{i}} = -R^c_{i\overline{i}j\overline{j}}=0.$$
So $M^3$ is Chern flat. By \cite{Boothby}, we know that the universal cover of $M$ is a connected, simply-connected complex Lie group $G$, and the lifting metric $\tilde{g}$ of $g$ is left-invariant and compatible with the complex structure of $G$. Let $\{ \varepsilon_1, \varepsilon_2, \varepsilon_3\}$ be a left-invariant unitary frame of $G$. Denote by $T^j_{ik}$ the components of the Chern torsion under the frame $\varepsilon$,
which are all constants. As $M$ is three-dimensional and balanced, by the proof of \cite[Proposition 2]{ZhouZ}, we know that one can make a constant unitary change of $\varepsilon$ so that $T^i_{ij}=0$ for each $1 \leq i \leq 3$. Under the case assumption, the tensor $B$
\[B =  \begin{bmatrix} c & 0 & 0 \\ 0 & c & 0 \\ 0 & 0 & c   \end{bmatrix}\!\!,\]
where $c=2a^2$. Similarly, by a suitable constant rotation $\rho_i\varepsilon_i$, where $|\rho_i|=1$ for each $i$, we may assume that $T^1_{23}=T^2_{31}=T^3_{12}=a>0$ under the left-invariant unitary frame $\varepsilon$ of $G$, hence the Lie algebra of $G$ is ${\mathfrak g}={\mathbb C}\{ X,Y,Z\}$, satisfying
$$ [X,Y]=2aZ, \ \ \ [Y,Z]=2aX, \ \ \ [Z,X]=2aY. $$
Therefore $G$ is isomorphic to $SO(3,{\mathbb C})$. Quotients of $SO(3,{\mathbb C})$ give us the only non-K\"ahler balanced {\em BTP} threefolds that are Chern flat.

Note that since $\theta^b$ is skew-symmetric under a special frame, so is the curvature matrix $\Theta^b$, therefore the holonomy group of $\nabla^b$ is contained in $SO(3)\subseteq U(3)$, but unlike the case of non-balanced {\em BTP} manifolds, it is not contained in $U(n-1)\times 1$ here.

\vspace{0.15cm}

\noindent {\bf Case \ref{tw_eql}:} $a_1=a_2>a_3$ or $a_1>a_2=a_3$ .

\vspace{0.15cm}

In this case $B$ has two distinct eigenvalues. First we will rule out the possibility of $a_3>0$. Assume that $a_3>0$,  namely, either $a_1=a_2>a_3>0$ or $a_1>a_2=a_3>0$, we will derive at a contradiction.

Since the argument for these two situations are exactly analogous, we will just focus on the case $a_1=a_2>a_3>0$. Write $a_1=a_2=a$ for simplicity.  Since the eigenspaces of $B$ are $\nabla^b$-parallel, it follows that $\theta^b_{13}=\theta^b_{23}=0$. By (\ref{eq:6.5}) and (\ref{eq:6.6}), it yields that
$$ \theta^b =  \begin{bmatrix} 0  & \alpha & 0 \\ -\alpha & 0 & 0  \\ 0 & 0 & 0 \end{bmatrix}\!\!, $$
where $\overline{\alpha }= \alpha$. Write $\alpha' =\alpha + 2\psi_3$, where $\psi_3=a(\varphi_3+\overline{\varphi}_3)$ is real, and the structure equation of Chern connection amounts to
\begin{equation} \label{eq:6.11}
 d \varphi = - \,^t\!(\theta^b -2\gamma ) \varphi + \tau =
\begin{bmatrix} \alpha' \varphi_2 + 2a_3 \varphi_3 \overline{\varphi}_2 \\ -\alpha' \varphi_1 - 2a_3 \varphi_3 \overline{\varphi}_1 \\  -2a_3\varphi_1\varphi_2 +2a (\varphi_2 \overline{\varphi}_1 - \varphi_1 \overline{\varphi}_2)  \end{bmatrix}\!\!.
\end{equation}
Then we get, from the first two lines,
\begin{eqnarray*}
d (\varphi_1\varphi_2) & = & d\varphi_1 \wedge \varphi_2 - \varphi_1 \wedge d\varphi_2 \ = \ -2a_3 \varphi_3 (\varphi_1 \overline{\varphi}_1 + \varphi_2 \overline{\varphi}_2),\\
d (\varphi_1\overline{\varphi}_2) & = & d\varphi_1\wedge \overline{\varphi}_2 - \varphi_1\wedge d\overline{\varphi}_2 \ = \ \alpha' ( \varphi_2 \overline{\varphi}_2 - \varphi_1 \overline{\varphi}_1 ).
\end{eqnarray*}
Here the fact $\overline{\alpha'}=\alpha'$ is used. In particular,
$$ d(\varphi_2 \overline{\varphi}_1 ) = - \overline{d(\varphi_1 \overline{\varphi}_2) } = \alpha ' ( \varphi_2 \overline{\varphi}_2 - \varphi_1 \overline{\varphi}_1 ) = d(\varphi_1 \overline{\varphi}_2) . $$
Take the exterior differentiation of the third equation of (\ref{eq:6.11}) and we obtain
$$ 0 = d^2\varphi_3 = -2a_3 d(\varphi_1\varphi_2) + 2a\, d(\varphi_2 \overline{\varphi}_1 - \varphi_1 \overline{\varphi}_2) = 4a_3^2  \varphi_3 (\varphi_1 \overline{\varphi}_1 + \varphi_2 \overline{\varphi}_2),$$
which is a contradiction. This shows that the case $a_3>0$ cannot occur and either $a_1>0=a_2=a_3$ or $a_1=a_2>0=a_3$ yields, which implies
\[B=\begin{bmatrix} c & 0 & 0 \\ 0 & 0 & 0 \\ 0 & 0 & 0   \end{bmatrix}\ \text{or}\
\begin{bmatrix} c & 0 & 0 \\ 0 & c & 0 \\ 0 & 0 & 0   \end{bmatrix}\!\!,\]
where $c=2a_1^2$. This completes the proof of Proposition \ref{Btype}.
\end{proof}

Therefore, we have shown the $\mbox{rank}\,B=3$ case of Theorem \ref{classification}. We will deal with the case $\mathrm{rank}\,B=1$ in \S \ref{FANO}, which leads to a Fano threefold and eventually end up with the Wallach threefold $(X,g)$. It will be computed in details in \S \ref{WCH3D}. The case $\mathrm{rank}\,B=2$ is the middle type, which will be discussed in \S \ref{mddtype3D}. This will complete the proof of Theorem \ref{classification}.

\vspace{0.3cm}

\section{The Fano case}\label{FANO}

In this section, we deal with the case $\mathrm{rank}\,B=1$. Eventually, we will show that this leads us to a unique example, the Wallach threefold, up to a scaling of the metric by a constant.

Let $(M^3,g)$ be a compact, non-K\"ahler, balanced {\em BTP} threefold with $B$ tensor
\[ B= \begin{bmatrix} 2a_1^2 & 0 & 0 \\ 0 & 0 & 0 \\ 0 & 0 & 0 \end{bmatrix}\!\!,\]
under a special frame $e$, where the only non-zero component of the Chern torsion tensor is $a_1=T^1_{23} >0$. We may assume for simplicity that $a_1=\frac{1}{2}$, after a scaling of the metric $g$ by a suitable constant.

From (\ref{eq:6.5}) and (\ref{eq:6.6}) we get $\theta^b_{12}=\theta^b_{13}=0$ and $\theta^b_{11}=\theta^b_{22} +\theta^b_{33}$. Then it follows that
\begin{equation*}
\theta^b = \begin{bmatrix} x+y & 0  & 0 \\ 0 & x & \alpha \\ 0 & -\overline{\alpha} &  y \end{bmatrix}\!\!, \ \ \ \
\gamma = \frac{1}{2} \begin{bmatrix} 0 & -\overline{\varphi}_3  & \overline{\varphi}_2 \\ \varphi_3 & 0 & 0 \\ -  \varphi_2 & 0 &  0 \end{bmatrix}\!\!, \ \ \
\tau = \begin{bmatrix} \varphi_2\varphi_3 \\ 0 \\ 0  \end{bmatrix}\!\!,
\end{equation*}
where $\overline{x}=-x$ and $\overline{y}=-y$. From $\theta = \theta^b-2\gamma$ and $d\varphi = -\,^t\!\theta \wedge \varphi +\tau $, it yields that
\begin{equation*}
\theta = \begin{bmatrix}  x+y & \overline{\varphi}_3  & -\overline{\varphi}_2 \\ -\varphi_3 & x & \alpha \\ \varphi_2 & -\overline{\alpha} &  y \end{bmatrix}\!\!, \ \ \ \ d\varphi  = \begin{bmatrix}  -(x+y)\varphi_1 - \varphi_2\varphi_3 \\ \varphi_1 \overline{\varphi}_3 - x \varphi_2 +\overline{\alpha} \varphi_3  \\ - \varphi_1 \overline{\varphi}_2 - \alpha \varphi_2 -y \varphi_3  \end{bmatrix}\!\!.
\end{equation*}
The matrix of the curvature of $\nabla^b$ is $\Theta^b=d \theta^b - \theta^b \wedge \theta^b$, whose entries are
\begin{equation} \label{eq:7.1}
\left\{ \begin{array}{lll}  \Theta^b_{22} \ = \ dx + \alpha \overline{\alpha}, \ \ \ \ \ \Theta^b_{33} \ = \ dy - \alpha \overline{\alpha}, \\
 \Theta^b_{23} \ = \ d\alpha  - x \alpha - \alpha y, \ \ \  \Theta^b_{12} \ = \ \Theta^b_{13} \ = \ 0, \\
 \Theta^b_{11} \ = \ dx+dy \ = \  \Theta^b_{22}+ \Theta^b_{33}.
 \end{array} \right.
\end{equation}
For convenience, we will use $\varphi_{ij}$ and $\varphi_{i\bar{j}}$ as the abbreviation of $\varphi_i \wedge \varphi_j$ and $\varphi_i \wedge \overline{\varphi}_j$, respectively. From the exterior differentiation $d\varphi$ of $\varphi$, the first Bianchi identity of $\nabla^b$ amounts to
\begin{equation} \label{eq:7.2}
\left\{ \begin{array}{lll} 0 =  d^2\varphi_1 \ = \ \{ \varphi_{2\overline{2}} + \varphi_{3\overline{3}} - \Theta^b_{11} \} \wedge \varphi_1, \\
0 = d^2\varphi_2 \ = \ \{  \varphi_{1\overline{1}} -  \varphi_{3\overline{3}} - \Theta^b_{22} \} \wedge \varphi_2 - \Theta^b_{32}\wedge \varphi_3, \\
0 = d^2\varphi_3 \ = \ \{  \varphi_{1\overline{1}} -  \varphi_{2\overline{2}} - \Theta^b_{33} \} \wedge \varphi_3 - \Theta^b_{23}\wedge \varphi_2
\end{array} \right.
\end{equation}
It follows from Theorem \ref{theorem1.1} that $\nabla^bT=0$ implies
$$ R^b_{ijk\overline{\ell}} =0 \ \ \ \mbox{and} \ \ \  R^b_{i\overline{j}k\overline{\ell}} = R^b_{k\overline{\ell}i\overline{j}} $$
for any $i,j,k,\ell$. So by $\Theta^b_{12}=\Theta^b_{13}=0$ and $\Theta^b_{11}= \Theta^b_{22} + \Theta^b_{33}$, we get
$$ R^b_{1\overline{b}i \overline{j}} = R^b_{i \overline{j}1\overline{b}}=0, \ \ \ \ \ R^b_{1 \overline{1}i\overline{j}} = R^b_{i\overline{j}1 \overline{1}} = R^b_{i\overline{j}2 \overline{2}} + R^b_{i\overline{j}3 \overline{3}}
= R^b_{2 \overline{2}i\overline{j}} + R^b_{3 \overline{3}i\overline{j}} $$
for any $i$, $j$ and any $b\in \{ 2,3\}$. Write
\begin{eqnarray*}
\Theta^b_{22} & = & A \varphi_{2\overline{2}} + B \varphi_{3\overline{3}} + E \varphi_{2\overline{3}} + \overline{E} \varphi_{3\overline{2}} + (A+B) \varphi_{1\overline{1}} \\
\Theta^b_{33} & = & B \varphi_{2\overline{2}} + C \varphi_{3\overline{3}} + F \varphi_{2\overline{3}} + \overline{F} \varphi_{3\overline{2}} + (B+C) \varphi_{1\overline{1}} \\
\Theta^b_{23} & = & E \varphi_{2\overline{2}} + F \varphi_{3\overline{3}} + D \varphi_{2\overline{3}} + G \varphi_{3\overline{2}} + (E+F) \varphi_{1\overline{1}} \\
\Theta^b_{11} & = & (A+B) \varphi_{2\overline{2}} + (B+C) \varphi_{3\overline{3}} + (E+F) \varphi_{2\overline{3}} + (\overline{E}+\overline{F}) \varphi_{3\overline{2}} + (A+2B+C) \varphi_{1\overline{1}}
\end{eqnarray*}
where $A,B,C,G$ are local real smooth functions. Then the first Bianchi identity (\ref{eq:7.2}) indicates
$$ E+F=0, \ \ \ \ A+B=B+C=G-B=1. $$
In particular,
\begin{equation}  \label{eq:7.3}
\Theta^b_{11} = 2\varphi_{1\overline{1}} + \varphi_{2\overline{2}} + \varphi_{3\overline{3}}
\end{equation}

\begin{remark} The pattern of the Bismut curvature implies that the holonomy group of $\nabla^b$ in this case would have its Lie algebra contained in the subalgebra
$$ {\mathfrak h} = \left\{ \begin{bmatrix} \mbox{tr}(X) & 0 \\ 0 & X \end{bmatrix} \mid X \in {\mathfrak u}(2) \right\} \subseteq {\mathfrak u}(3). $$
\end{remark}

Denote by $\Theta$ the curvature matrix of the Chern connection $\nabla^c$ under $e$, and the balanced condition $\eta =0$ implies that
$$ \mathrm{tr} \,\Theta = \mathrm{tr}\, \Theta^b = 2\Theta^b_{11} .$$
Let $\omega = \sqrt{-1} \big(  \varphi_1\overline{\varphi}_1 + \varphi_2\overline{\varphi}_2 +\varphi_3\overline{\varphi}_3 \big)$ be the K\"ahler form of the Hermitian metric $g$, and let
\begin{equation} \label{eq:7.4}
\tilde{\omega} = \sqrt{-1} \big( 2 \varphi_1\overline{\varphi}_1 + \varphi_2\overline{\varphi}_2 +\varphi_3\overline{\varphi}_3 \big) .
\end{equation}
be the K\"ahler form of another Hermitian metric $\tilde{g}$.
Clearly $\tilde{\omega}$ is independent of the choice of special frames hence is globally defined on $M$.  The Chern Ricci form of $g$ is $\mathrm{Ric}(\omega ) = \sqrt{-1}\mathrm{tr} \,\Theta = 2 \tilde{\omega}$. As $\mathrm{Ric}(\omega )$ is always closed, we know $\tilde{\omega}$ is K\"ahler. Note that $\tilde{\omega }^3 = 2 \omega^3$ and thus
\begin{equation} \label{eq:7.5}
 \mbox{Ric} (\tilde{\omega}) = \mbox{Ric}(\omega ) = 2 \tilde{\omega}.
 \end{equation}
Therefor $\tilde{\omega}$ is a K\"ahler-Einstein metric with positive Ricci curvature and $M^3$ is a Fano threefold.

Denote by $E$ and $F$ the $C^{\infty}$ complex vector bundle on $M$ with fibers $E_x = {\mathbb C} \{ e_2(x), e_3(x)\}$ and $F_x={\mathbb C}  e_1(x)$, respectively. They are both globally defined since $E$ is the eigenspace of $B$ corresponding to the eigenvalue $0$, and $F$ is the orthogonal complement of $E$ in $T^{1,0}M$.

We claim that $E$ is a holomorphic subbundle of $T^{1,0}M$. To see this, for any $i\in \{ 2, 3\}$ and any $j$, we have
$$ \langle \nabla^c_{\overline{j}}e_i , \overline{e}_1 \rangle = \langle \nabla^b_{\overline{j}}e_i - 2\gamma_{\overline{j}}e_i , \overline{e}_1  \rangle = \theta^b_{i1}(\overline{e}_{j}) - 2 \gamma_{i1}(\overline{e}_{j}) = 0.$$
This means that $\nabla^c_{\overline{X}}E \subseteq E$ for any type $(1,0)$ vector field $X$, so $E$ is holomorphic. Note that the distribution $E$ is not a foliation, while $F$ on the other hand is a foliation, but is not holomorphic.

Equip $E$ with the restriction metric from $(M^3,g)$. By the formula of Chern connection matrix $\theta$ on $M$, the matrices of connection and curvature of the Hermitian bundle $E$ under the local frame $\{ e_2, e_3\}$ are respectively
$$ \theta^E = \begin{bmatrix} x & \alpha \\ -\overline{\alpha} & y \end{bmatrix}\!\!, \ \ \ \  \ \Theta^E = d\theta^E -\theta^E \wedge \theta^E = \begin{bmatrix} \Theta^b_{22} & \Theta^b_{23} \\ \Theta^b_{32} & \Theta^b_{33} \end{bmatrix}\!\!.$$
In particular, $\sqrt{-1}\mathrm{tr}\Theta^E = \sqrt{-1}(\Theta^b_{22} + \Theta^b_{33}) = \frac{\sqrt{-1}}{2}\mathrm{tr}\,\Theta^b = \frac{\sqrt{-1}}{2}\,\mathrm{tr}\Theta = \tilde{\omega}$. This means that
\begin{equation} \label{eq:7.6}
 c_1(E) = \frac{1}{2} c_1(M) = [\tilde{\omega}].
\end{equation}
Denote by $L$ the holomorphic line bundle on $M$ which is the quotient of $E$ in $T^{1,0}M$, namely, the exact sequence follows
\begin{equation} \label{eq:7.7}
0 \rightarrow E \rightarrow T^{1,0}M \rightarrow L \rightarrow 0.
\end{equation}
Let $h$ be the first Chern class $c_1(L)$ of $L$. The above short exact sequence implies
$$ c_1(E)+h =c_1(M), \ \ \ c_2(E) + hc_1(E) = c_2(M), \ \ \ c_2(E)h =c_3(M).$$
Then the equality (\ref{eq:7.6}) yields
\[c_1(E)=h, \quad c_1(M)=2h, \quad c_2(E)=c_2(M)-h^2, \quad c_2(M)h - h^3=c_3(M).\]
In particular, $L$ is an ample line bundle on $M$, the anti-canonical line bundle $K_M^{-1}=2L$ as holomorphic line bundles are uniquely determined by their Chern classes on Fano manifolds, and the Chern numbers of $M^3$ satisfy
\begin{equation}  \label{eq:7.8}
c_1(M)c_2(M) = 2h (c_2(E)+h^2) = 2c_3(M) + \frac{1}{4}c_1^3(M).
\end{equation}

Recall that the {\em index} of a Fano manifold $X^n$ is the largest positive integer $r$ so that $K^{-1}_X = rA$ for an ample line bundle $A$. It is necessarily less than or equal to $n+1$, where $r=n+1$ if and only if $X={\mathbb P}^n$ and $r=n$ if and only if $X={\mathbb Q}^n$, the smooth quadric hypersurface in ${\mathbb P}^{n+1}$. Fano manifolds satisfying $r=n-1$ are called {\em del Pezzo manifolds,} which are classified by Fujita \cite{Fujita} as one of the following seven types, according to their {\em degree} $d$, which is the self intersection number $A^n$:
\begin{enumerate}
\item $d=1$: $X^n_6\subset {\mathbb P}(1^{n-1}, 2, 3)$, a degree $6$ hypersurface in the weighted projective space.
\item $d=2$: $X^n_4\subset {\mathbb P}(1^{n}, 2)$, a degree $4$ hypersurface in the weighted projective space.
\item $d=3$: $X^n_3\subset {\mathbb P}^{n+1}$, a cubic hypersurface.
\item $d=4$: $X^n_{2,2}\subset {\mathbb P}^{n+2}$, a complete intersection of two quadrics.
\item $d=5$: $Y^n$, a linear section of ${\mathbb Gr}(2,5) \subset {\mathbb P}^9$.
\item $d=6$: ${\mathbb P}^1\!\times \! {\mathbb P}^1\!\times \!{\mathbb P}^1$, or ${\mathbb P}^2\!\times \! {\mathbb P}^2$, or the flag threefold ${\mathbb P}(T_{ \!{\mathbb P}^{2}} )$.
\item $d=7$: ${\mathbb P}^3\# \overline{{\mathbb P}^3} $, the blow-up of ${\mathbb P}^3$ at a point.
\end{enumerate}

For $n=3$, del Pezzo threefolds were classified by Iskovskikh \cite{Iskov} earlier, and in Table 12.2 of \cite{IP} we can find the third betti number $b_3$, hence the Euler number $c_3=4-b_3$ of del Pezzo threefolds of degree $1\leq d\leq 5$:
\begin{equation} \label{eq:7.9}
c_3(X^3_6)=-38, \ \ \ c_3(X^3_4)=-16, \ \ \ c_3(X^3_3) = -6, \ \ \ c_3(X^3_{2,2})= 0, \ \ \ c_3(Y^3)= 4.
\end{equation}

Let us return to our manifold $M^3$, where it holds that $K_M^{-1}=2L$ for the ample $L$, so the index of $M^3$ is either $4$ or $2$, which means $M^3$ is biholomorphic to either ${\mathbb P}^3$ or a del Pezzo threefold. It is well-known that $c_1c_2=24$ holds for any Fano threefold, so the equality (\ref{eq:7.8}) implies
\begin{equation} \label{eq:7.10}
c_3(M) = 12 - \frac{1}{8}c_1^3 .
\end{equation}
If $M^3$ is a del Pezzo threefold of degree $d$, then $c_1^3=8d$, hence the equality (\ref{eq:7.10}) yields $c_3=12-d$. This rules out the possibility of $1\leq d\leq 5$ by (\ref{eq:7.9}). The case ${\mathbb P}^3\# \overline{{\mathbb P}^3}$ of degree $d=7$ has the Euler number $c_3=6$, which is not equal to $12-7$. Similarly, for the case ${\mathbb P}^1 \times {\mathbb P}^1 \times {\mathbb P}^1$ of degree $d=6$, its Euler number $c_3=8$ is not equal to $12-6$, which indicates that neither can be our $M^3$. Therefore only two possibilities are left, namely, $M^3$ is either the flag threefold  ${\mathbb P}(T_{{\mathbb P}^2} )$ or ${\mathbb P}^3$. Note that in both of these two cases the short exact sequence (\ref{eq:7.7}) exists, for which we have to dig further.

Consider the short exact sequence (\ref{eq:7.7}) of holomorphic vector bundles on our threefold $M^3$, where $E$ has the fiber ${\mathbb C}\{ e_2,e_3\}$ under a special frame $e$. Denote by $\Omega$ the holomorphic cotangent bundle of $M^3$ and write $\Omega(L)=\Omega \otimes L$. Let $\xi \in H^0(M, \Omega (L))$ be the nowhere zero holomorphic section which gives the map $T^{1,0}M \rightarrow L$ in (\ref{eq:7.7}). The K\"ahler-Einstein metric $\tilde{\omega}$ on $M^3$ naturally induces Hermitian metrics on $\Omega$ and $L=-\frac{1}{2}K_M$, hence on $\Omega (L)$. We have the following claim
\begin{claim}\label{eq:7.12}
The norm $\parallel \!\xi \!\parallel^2$ is a constant under the K\"ahler-Einstein metric $\tilde{\omega}$.
\end{claim}

\begin{proof}
To see this, let $e$ be a local special frame of $(M^3,g)$, with dual coframe $\varphi$. Let $s$ be a local holomorphic frame of $L$. Since $\xi$ is a $L$-valued $1$-form, we have $\xi = \psi \otimes s$ where $\psi$ is a nowhere-zero local holomorphic $1$-form. The kernel of the map given by $\xi$ is $E$, so $\psi = f\varphi_1$ for some local smooth function $f$, which is nowhere zero. By the structure equation, it follows that $d\varphi_1=-(x+y)\varphi_1 - \varphi_2\varphi_3$, so the holomorphicity of $\psi$ gives us
$$ 0 = \overline{\partial} \psi =  \overline{\partial}f \wedge \varphi_1 - f (x+y)^{0,1}\wedge \varphi_1. $$
Hence $(x+y)^{0,1}=\overline{\partial}\log f$. By the fact that $\overline{x}=-x$ and $\overline{y}=-y$, we get
$$ x+y = - \partial \log \overline{f} + \overline{\partial}\log f, \ \ \ \  \frac{1}{\sqrt{-1}}\tilde{\omega} = \Theta^b_{11} = d(x+y) = \partial \overline{\partial}\log |f|^2. $$
On the other hand,  $L=-\frac{1}{2}K_M$ is equipped with the induced metric from $\tilde{\omega}$, then it yields
$$ \frac{1}{\sqrt{-1}} \tilde{\omega} = \Theta^L = -\partial \overline{\partial}\log \parallel \!s \!\parallel ^2. $$
Combine the above two equations and we get
$$ \partial \overline{\partial}\log (|f|^2 \parallel \!s\! \parallel^2) = 0. $$
Since $\{ \sqrt{2} \varphi_1, \varphi_2, \varphi_3\}$ is a local unitary coframe for $\tilde{\omega}$, the norm square of $\varphi_1$ under $\tilde{\omega}$ is $\frac{1}{2}$, hence
$$ \parallel \!\xi \!\parallel^2=\frac{1}{2}\,|f|^2\parallel \!s \!\parallel^2.$$
It is a global positive function on $M^3$, and its log is pluriclosed by the above equation, hence it must be a constant. This establishes the claim above.
\end{proof}

Then we will use Claim \ref{eq:7.12} to rule out the possibility of $M^3$ being ${\mathbb P}^3$. Assume that $M^3$ is ${\mathbb P}^3$. In this case $L={\mathcal O}(2)$. Any nowhere-zero holomorphic section $\xi \in V:=H^0({\mathbb P}^3,\Omega (2))$ gives the so-called {\em null-correlation bundle} (see for example \cite{Schneider}), which is $E(-1)$ in our notation. Let $\tilde{\omega}$  be the (scaled) Fubini-Study metric of ${\mathbb P}^3$ with Ricci curvature $2$. It has constant holomorphic sectional curvature $1$. Let $[Z_0\!:\!Z_1\!:\!Z_2\!:\!Z_3]$ be the standard unitary homogeneous coordinate of ${\mathbb P}^3$. In the coordinate neighborhood $U_0=\{ Z_0\neq 0\}$, let $z_i=\frac{Z_i}{Z_0}$, $1\leq i\leq 3$ and it follows that
$$ \sqrt{-1}\, \Theta^L=\tilde{\omega} = \frac{1}{2}\mbox{Ric}(\tilde{\omega}) = 2\sqrt{-1}\,\partial \overline{\partial} \log (1+|z|^2), \ \ \ \ \ \ \parallel \!Z_0^2\! \parallel^2 = \frac{1}{(1+|z|^2)^2} $$
where $Z_0^2$ is a local frame of $L$ in $U_0$ and $|z|^2= |z_1|^2 + |z_2|^2 + |z_3|^2$. Under the coordinate $z$,
$$ \tilde{g}_{i\overline{j}} = \frac{2}{1+|z|^2}\delta_{ij} - \frac{2}{(1+|z|^2)^2}\overline{z}_iz_j, \ \ \ \ \ \ \tilde{g}^{\overline{i}j} = \frac{1}{2}(1+|z|^2) \big( \delta_{ij} + \overline{z}_iz_j\big)  .$$

As is well-known, $V \cong {\mathbb C}^6 $ has a basis $\{ \lambda_{ij} \}_{0\leq i<j\leq 3}$, where
$ \lambda_{ij} = Z_i dZ_j - Z_j dZ_i$. Suppose that $\xi \in V$ is a nowhere zero section.
It follows that $\xi = \sum a_{ij}\lambda_{ij}$ for some constants $a_{ij}$. In $U_0$,
$$  \lambda_{0i}=Z_0^2dz_i, \ \ \ \ \lambda_{ij}=Z_0^2(z_idz_j-z_jdz_i) $$
for any $i,j>0$, hence $ \xi = Z_0^2 (\ell_{1}dz_1 + \ell_2dz_2 + \ell_3dz_3)$, where
\begin{eqnarray*}
\ell_1 & = & a_{01} - a_{12}z_2-a_{13}z_3 \\
\ell_2 & = & a_{02} + a_{12}z_1-a_{23}z_3 \\
\ell_3 & = & a_{03} + a_{13}z_1 + a_{23}z_2
\end{eqnarray*}
It yields that
\begin{eqnarray*}
\parallel \xi \parallel^2 & = & \parallel \!Z_0^2\! \parallel^2 \sum_{i,j=1}^3 \ell_i \overline{\ell_j} \, \tilde{g}^{\overline{i}j} \\
& = & \parallel \!Z_0^2\! \parallel^2  \frac{1}{2}(1+|z|^2) \big( \sum_i |\ell_i|^2 + |\sum_k \ell_k\overline{z}_k |^2 \big) \\
& = & \frac{1}{2(1+|z|^2)}  \big(  |\ell_1|^2 + |\ell_2|^2 + |\ell_3|^2 + |\ell_1 \overline{z}_1 + \ell_2 \overline{z}_2 + \ell_3 \overline{z}_3 |^2 \big)
\end{eqnarray*}
\begin{claim}
The expression $\|\xi\|^2$ above can not be a constant. Hence it rules out the possibility of $M^3={\mathbb P}^3$.
\end{claim}

\begin{proof}
If $\|\xi\|^2$ above were a constant, there exists a constant $C>0$ such that
\[  |\ell_1|^2 + |\ell_2|^2 + |\ell_3|^2 + |\ell_1 \overline{z}_1 + \ell_2 \overline{z}_2 + \ell_3 \overline{z}_3 |^2 = C(1 + |z_1|^2 + |z_2|^2 + |z_3|^2). \]
Then it is clear that the degree $4$ part of the left hand side of the equality above has to vanish, which is exactly
\[ |\ell_1^{(1)} \overline{z}_1 + \ell_2^{(1)} \overline{z}_2 + \ell_3^{(1)} \overline{z}_3 |^2, \]
if we decompose $\ell_i$ into $\ell_i^{(0)} + \ell_i^{(1)}$ for $1 \leq i \leq 3$, where $\ell_i^{(0)},\ \ell_i^{(1)}$ are the parts of the degree $0$ and the degree $1$ of $\ell_i$ respectively. This indicates
\[ |\ell_1^{(1)} \overline{z}_1 + \ell_2^{(1)} \overline{z}_2 + \ell_3^{(1)} \overline{z}_3 |^2 =0,\]
which implies
\[ a_{12}=a_{13}=a_{23}=0. \]
It follows that
\[ |a_{01}|^2 + |a_{02}|^2 + |a_{03}|^2 + |a_{01} \overline{z}_1 + a_{02} \overline{z}_2 + a_{03} \overline{z}_3 |^2 = C(1 + |z_1|^2 + |z_2|^2 + |z_3|^2), \]
which implies
\[\begin{bmatrix} \overline{a_{01}}a_{01} & \overline{a_{01}}a_{02} & \overline{a_{01}}a_{03} \\
\overline{a_{02}}a_{01} & \overline{a_{02}}a_{02} & \overline{a_{02}}a_{03} \\
\overline{a_{03}}a_{01} & \overline{a_{03}}a_{02} & \overline{a_{03}}a_{03}\end{bmatrix}
= \begin{bmatrix} C & & \\ & C & \\ & & C\end{bmatrix}\!\!.\]
It is a contradiction by comparison of the rank of matrices above.
\end{proof}

Finally let us consider the flag threefold $X^3 = {\mathbb P}(T_{{\mathbb P}^2} )$. It is the hypersurface in $N^4={\mathbb P}^2\times {\mathbb P}^2$ defined by $Z_0W_0+Z_1W_1+Z_2W_2=0$, where $Z$, $W$ are the standard unitary homogeneous coordinate of the two factors of $N$. For $i=1,2$, denote by $\pi_i: X^3 \rightarrow {\mathbb P}^2$ the restriction on $X$ of the projection map from $N$ onto its $i$-th factor. The Picard group $\mbox{Pic}(X)\cong {\mathbb Z}^2$ is generated by $L_1$ and $L_2$, where $L_i=\pi_i^{\ast}{\mathcal O}_{{\mathbb P}^2}(1)$, and the anti-canonical line bundle of $X$ is $-K_X = 2L=2(L_1+L_2)$. The K\"ahler-Einstein metric $\tilde{\omega}$ on $X$ is the restriction of the product of Fubini-Study metric, and we have
\begin{equation} \label{eq:7.13}
 \sqrt{-1}\, \Theta^L = \tilde{\omega} = \frac{1}{2}\mbox{Ric}(\tilde{\omega}) = \omega_0|_X, \ \ \ \omega_0 = \sqrt{-1} \, \partial \overline{\partial} \log |Z|^2 + \sqrt{-1} \, \partial \overline{\partial} \log |W|^2  .
 \end{equation}

Then the following claim yields
\begin{claim} \label{eq:7.14}
\[H^0(X, \Omega \otimes L) = {\mathbb C} \xi, \ \ \ \ \ \ \ \ \xi = \sum_{i=0}^2 W_idZ_i = - \sum_{i=0}^2 Z_idW_i.\]
\end{claim}

\begin{proof}
It is clear that $\xi$ is a global holomorphic section of $\Omega \otimes L$, and is nowhere zero, hence give a surjective bundle map $T_X \rightarrow L$, which  will lead to the Wallach space example as we shall see in the next section. Here we want to show that the vector space $H^0(X, \Omega \otimes L)$ is one-dimensional, hence any section is a constant multiple of $\xi$. To see this, let us denote by $T_{X|{\mathbb P}^2}$ the relative tangent bundle of the map $\pi_1: X \rightarrow {\mathbb P}^2$, given by
$$ 0 \rightarrow T_{X|{\mathbb P}^2} \rightarrow T_X \rightarrow \pi_1^{\ast}T_{{\mathbb P}^2} \rightarrow  0 .$$
Then we have $T_{X|{\mathbb P}^2}=2L-3L_1=2L_2-L_1$. Taking the dual of the above short exact sequence and tensoring it with $L$, we get
\begin{equation} \label{eq:7.15}
0 \rightarrow  \pi_1^{\ast}\Omega_{{\mathbb P}^2} \!\otimes \!L \rightarrow \Omega_X\! \otimes\! L \rightarrow L' \rightarrow  0 , \ \ \ \ \ \  \ \ L' = -T_{X|{\mathbb P}^2} + L = 2L_1-L_2.
\end{equation}
On the other hand, $X={\mathbb P}(T_{{\mathbb P}^2})$, so the relative Euler sequence is
\begin{equation} \label{eq:7.16}
0 \rightarrow  {\mathcal O}_X \rightarrow \pi_1^{\ast}\Omega_{{\mathbb P}^2} \!\otimes \!L \rightarrow  T_{X|{\mathbb P}^2} \rightarrow  0.
\end{equation}
Since $L^3_1=0$, $L_1^2L_2=1$, we have $L_1^2 L'=-1$, so $H^0(X,L')=0$ as $L^2_1$ is represented by the fibers of $\pi_1$ which will have non-negative intersection with any effective divisor in $X$. Similarly, $H^0(X, T_{X|{\mathbb P}^2})=0$. So by (\ref{eq:7.15}) and (\ref{eq:7.16}), we get
$$  H^0(X, \Omega_X\! \otimes \!L ) \cong H^0(X, \pi_1^{\ast}\Omega_{{\mathbb P}^2} \!\otimes \!L) \cong H^0(X,  {\mathcal O}_X) \cong {\mathbb C}. $$
This establishes Claim \ref{eq:7.14}.
\end{proof}

Consider the global $(1,1)$-form on $X$ defined by
\begin{equation} \label{eq:7.17}
\sigma = \sqrt{-1}\,\frac{\,^t\!WdZ \wedge \overline{^t\!WdZ} }{|Z|^2 \,|W|^2}
\end{equation}
where $Z$, $W$ are unitary homogeneous coordinate on the two factors of $N={\mathbb P}^2\times {\mathbb P}^2$, viewed as  column vectors. It is not hard to see that the norm $\parallel \! \sigma \!\parallel =\frac{1}{2}$ with respect to the K\"ahler-Einstein metric $\tilde{\omega}$ of $X$. Consider the Hermitian metric $g$ on $X$ with K\"ahler form $\omega = \tilde{\omega }- \sigma$, which is clearly a homogeneous Hermitian metric on $X$. We will call the Hermitian manifold $(X,g)$ the {\em Wallach threefold} from now on, to honor the influential work \cite{Wallach} in geometry.

We will verify in the next section that $(X,g)$ is indeed balanced and {\em BTP}.  We will also show that its Chern connection has non-negative bisectional curvature and positive holomorphic sectional curvature, and all three Ricci tensors of the Chern connection are positive. We will compute the sectional curvature of the Levi-Civita (Riemannian) connection, and show that it is non-negative, and the Levi-Civita connection has constant Ricci curvature $6$, thus  $g$ lies in the boundary of the set of  metrics with positive sectional curvature discovered by Wallach in \cite{Wallach}.

Note that homogeneous metrics on $X$ with positive sectional curvature, which are all Hermitian as observed by Wallach in \cite{Wallach}, form a moduli which depends on three real parameters. After scaling, these metrics form a peculiar planer region (see for example Figure 1 in \cite{BM}). It is not clear which metric in the set is the `best' amongst its peers. Our metric $g$ corresponds to the one where all three parameters are equal (that is, the metric given by the Killing form). It is Einstein and has non-negative sectional curvature but not strictly positive sectional curvature, but it is the unique (up to scaling) balanced and {\em BTP} metric on $X$. Its Chern connection also has non-negative bisectional curvature, while the (unique) K\"ahler-Einstein metric $\tilde{g}$ of $X$ does not have nonnegative bisectional curvature.

\vspace{0.3cm}

\section{The Wallach threefold}\label{WCH3D}

Let $\omega_0$ be the product of Fubini-Study metric on $N^4={\mathbb P}^2\times {\mathbb P}^2$, given by (\ref{eq:7.13}), where $Z$ and $W$ are unitary homogeneous coordinates, and the flag threefold $X$ given as the smooth ample divisor $\{ \,^t\!ZW=0\}$ in $N$. Here and below we will consider $Z$ and $W$ as column vectors. The restriction $\tilde{\omega }= \omega_0|_X$ is the K\"ahler-Einstein metric on $X$ with $\mbox{Ric}(\tilde{\omega})=2\tilde{\omega}$, and our Hermitian metric $g$, which will be called the {\em Wallach metric} from now on, is defined by $\omega = \tilde{\omega} - \sigma $ where the global $(1,1)$-form  $\sigma$ on $X$ is defined by (\ref{eq:7.17}). We will verify that $g$ is balanced and {\em BTP}, and compute its Chern and Riemannian curvature.

Fix any point $p\in X$. Notice that for any $A\in SU(3)$, the map $([Z], [W]) \mapsto ([AZ],[\overline{A}W])$ is an isometry on $(X,g)$. So without loss of generality, we may assume that $p =([1\!:\!0\!:\!0], [0\!:\!0\!:\!1])$. Let $U_{02}=\{ Z_0\neq 0\} \times \{ W_2\neq 0\}$ be a coordinate neighborhood in $N$, with holomorphic coordinate $(z_1, z_2, w_0, w_1)$ where $z_i=\frac{Z_i}{Z_0}$, $i=1,2$, and $w_j = \frac{W_j}{W_2}$, $j=0,1$. Within $U_{02}$, the hypersurface $X$ is defined by
\begin{equation} \label{eq:8.1}
w_0=-z_2-z_1w_1,
\end{equation}
so $(z_1, z_2, w_1)$ becomes a local holomorphic coordinate in $U=X\cap U_{02}$. Let us write $|z|^2=|z_1|^2+|z_2|^2$ and $|w|^2 = |w_0|^2 + |w_1|^2$ as usual, then in $U_{02}$ we have
\begin{equation} \label{eq:8.2}
\frac{1}{\sqrt{-1}}\omega_0 = \sum_{i,j=1}^2 \frac{(1+|z|^2)\delta_{ij} - \overline{z}_iz_j } {(1+|z|^2)^2} dz_i\wedge d\overline{z}_j + \sum_{i,j=0}^1 \frac{(1+|w|^2)\delta_{ij} - \overline{w}_iw_j } {(1+|w|^2)^2} dw_i\wedge d\overline{w}_j,
\end{equation}
and in $U$, $\tilde{\omega}$ is just the restriction of $\omega_0$ on $X$ using the equation (\ref{eq:8.1}). For convenience, let us write $w_1=z_3$, and define
\begin{equation} \label{eq:8.3}
 \alpha = 1+ |z_1|^2 + |z_2|^2, \ \ \  \beta = 1+|z_3|^2 + |f|^2, \ \ \ f=z_2+z_1z_3.
\end{equation}
In $U$, $(z_1,z_2,z_3)$ gives local holomorphic coordinate for $X$, and $p$ corresponds to the origin $(0,0,0)$. By (\ref{eq:8.2}) the metric $\tilde{g}$ has components
\begin{equation} \label{eq:8.4}
 \tilde{g}_{i\overline{j}} = \frac{\alpha_{i\overline{j}}}{\alpha} - \frac{\alpha_i \alpha_{\overline{j}} }{\alpha^2} + \frac{\beta_{i\overline{j}}}{\beta} - \frac{\beta_i \beta_{\overline{j}} }{\beta^2}, \ \
 \ \ \ \ 1\leq i,j\leq 3,
\end{equation}
where subscripts stand for partial derivatives in $z_i$ or $\overline{z}_j$.  Taking partial derivative in $z_k$, we obtain
\begin{equation} \label{eq:8.5}
 \tilde{g}_{i\overline{j},k} = - \frac{1}{\alpha^2} \big( \alpha_k \alpha_{i\overline{j}}  + \alpha_i \alpha_{k\overline{j}} \big) + \frac{2\alpha_i \alpha_k \alpha_{\overline{j}} }{\alpha^3} + \frac{\beta_{ik\overline{j}}}{\beta} - \frac{1}{\beta^2} \big( \beta_k \beta_{i\overline{j}} + \beta_i \beta_{k\overline{j}} + \beta_{\overline{j}} \beta_{ik} \big)  + \frac{2\beta_i\beta_k \beta_{\overline{j}} }{\beta^3}. \ \ \
\end{equation}
Here we used the fact that $\alpha_{ik}=0$ and $\alpha_{ik\overline{j}} =0$. At the origin, $\alpha (0)=\beta (0)=1$, $\alpha_i(0)=\beta_i(0)=0$, $\beta_{ik}(0)=0$, and $\alpha_{i\overline{j}}(0)= \delta_{i1}\delta_{j1}+ \delta_{i2}\delta_{j2}$, $\beta_{i\overline{j}}(0)=  \delta_{i2}\delta_{j2}+ \delta_{i3}\delta_{j3}$, so we have
\begin{eqnarray}
&&  \tilde{g}_{i\overline{j},k}(0) \ = \beta_{ik\overline{j}}(0) =f_{ik}\overline{f_j}(0)= (\delta_{i1}\delta_{k3}+\delta_{i3}\delta_{k1})\delta_{j2} , \label{eq:8.6} \\
 && \tilde{g}_{i\overline{j},kp}(0) = 0, \label{eq:8.7} \\
 && \tilde{g}_{i\overline{j},k\overline{\ell}}(0) = - \big(\alpha_{i\overline{j}} \alpha_{k\overline{\ell}} + \alpha_{i\overline{\ell}} \alpha_{k\overline{j}} \big)  + f_{ik}\overline{ f_{j\ell} } - \big(\beta_{i\overline{j}} \beta_{k\overline{\ell}} + \beta_{i\overline{\ell}} \beta_{k\overline{j}} \big). \label{eq:8.8}
\end{eqnarray}
From this, we see that
\begin{equation} \label{eq:8.9}
\left\{ \begin{array}{llll}  \tilde{g}_{i\overline{j},k\overline{\ell}}(0)  & = & 0 , \ \ \ \ \mbox{if} \ \{ i,k\} \neq \{ j,\ell\}  \\ \tilde{g}_{i\overline{i},i\overline{i}}(0) &  =  &  -2 (\alpha_{i\overline{i}} + \beta_{i\overline{i}} )  \ = \ - 2\,\big( 1+ \delta_{i2}\big), \\
 \tilde{g}_{i\overline{i},k\overline{k}}(0)  & = &  \tilde{g}_{i\overline{k},k\overline{i}}(0)  \ = \ -(\alpha_{i\overline{i}}\alpha_{k\overline{k}} + \beta_{i\overline{i}}\beta_{k\overline{k}}) + f_{ik} \ = \  \left\{ \begin{array}{ll} -1, \ \mbox{if} \ \{ i,k\} = \{ 1,2\} \ \mbox{or} \ \{ 2,3\} , \\ \ \,1, \ \ \mbox{if} \ \{ i,k\} = \{ 1,3\} .\end{array} \right.
 \end{array}  \right.
\end{equation}
Similarly, since in $U$ the $(1,1)$-form $\sigma$ is given by
$$ \sigma = \sqrt{-1}\sum_{i,j=1}^3 \sigma_{i\overline{j}} dz_i \wedge d\overline{z}_j = \frac{\sqrt{-1} }{\alpha \beta } (z_3dz_1+dz_2) \wedge (\overline{z}_3 d\overline{z}_1 + d\overline{z}_2),
$$
therefore we have
 \begin{equation*}
 \sigma_{i\overline{j}} = \frac{1}{\alpha \beta} ( \delta_{i1}\delta_{j1} |z_3|^2 + \delta_{i1}\delta_{j2} z_3 + \delta_{i2}\delta_{j1}\overline{z}_3 + \delta_{i2}\delta_{j2}).
\end{equation*}
From this we compute
\begin{equation} \label{eq:8.10}
\left\{ \begin{array}{llll} \sigma_{i\overline{j}}(0) & = & \delta_{i2}\delta_{j2}, \\
\sigma_{i\overline{j},k}(0) & = & \delta_{i1}\delta_{j2}\delta_{k3}, \\
 \sigma_{i\overline{j},kp}(0) & = & 0,\\
 \sigma_{i\overline{j},k\overline{\ell}}(0)  & = & \delta_{ij}\delta_{i1} \delta_{k\ell}\delta_{k3} - \delta_{k\ell} (1+\delta_{k2})\delta_{ij} \delta_{i2}.
\end{array}  \right.
\end{equation}
By (\ref{eq:8.4}), we have $\tilde{g}_{i\overline{j}}(0)=\delta_{ij}(1+\delta_{i2})$, so at the origin $g_{i\overline{j}}=\tilde{g}_{i\overline{j}} - \sigma_{i\overline{j}}$ satisfies
\begin{equation} \label{eq:8.11}
 g_{i\overline{j}}(0) = \delta_{ij}, \ \ \
g_{i\overline{j},k}(0) = \delta_{i3}\delta_{j2}\delta_{k1}, \ \ \
 g_{i\overline{j},kp}(0) =0,
\end{equation}
and
\begin{equation} \label{eq:8.12}
g_{i\overline{j},k\overline{\ell}}(0)  \,= \,0  \  \ \mbox{if} \ \{ i,k\} \neq \{ j,\ell\}, \ \ \ \ \ \ \  g_{i\overline{i},i\overline{i}}(0) \,  = \, -2,
\end{equation}
and
\begin{eqnarray}
\label{eq:8.13} && g_{i\overline{k},k\overline{i}}(0)  \ = \  \left\{ \begin{array}{ll} -1, \ \mbox{if} \ \{ i,k\} = \{ 1,2\} \ \mbox{or} \ \{ 2,3\} , \\ \ \,1, \ \ \mbox{if} \ \{ i,k\} = \{ 1,3\} .\end{array} \right. \\
\label{eq:8.14} && g_{i\overline{i},k\overline{k}}(0)  \ = \ \left\{ \begin{array}{lll} -1, \ \mbox{if} \ (ik) = (12) \ \mbox{or} \ (32) , \\ \ \,1, \ \ \mbox{if} \ (ik) = (31), \\ \ \,0, \ \ \mbox{if} \ (ik) = (13), (21) \ \mbox{or} \ (23) .\end{array} \right.
\end{eqnarray}
The curvature components of the Chern connection $\nabla^c$, defined by $\Theta_{ij} =\sum_{k,\ell} R^c_{k\overline{\ell}i\overline{j}} dz_k\wedge d\overline{z}_{\ell}$ where $\Theta = \overline{\partial}  \theta =\overline{\partial} (\partial g g^{-1})$,   is given by
$$ R^c_{k\overline{\ell}i\overline{j}} = - g_{i\overline{j},k\overline{\ell}} + \sum_{p,q} g_{i\overline{p},k} \overline{ g_{j\overline{q},\ell}  } g^{\overline{p}q} .$$
At the origin, $g_{i\overline{j}} (0)=\delta_{ij}$, and all $g_{i\overline{j},k} (0)=0$ except $g_{3\overline{2},1} (0)=1$, so the second term on the right hand side of the above equality is $\delta_{ij}\delta_{i3}\delta_{k\ell}\delta_{k1}$, and by (\ref{eq:8.12}), (\ref{eq:8.13}) and (\ref{eq:8.14}) we get at the origin that
\begin{eqnarray}
\label{eq:8.15} && R^c_{i\overline{j}k\overline{\ell}} = 0 , \ \mbox{if} \ \{i,k\} \neq \{ j,\ell\}, \\
\label{eq:8.16} && R^c_{i\overline{i}i\overline{i}} = 2, \\
\label{eq:8.17} && R^c_{1\overline{2}2\overline{1}}   \ = \ R^c_{3\overline{2}2\overline{3}} =1, \ \ R^c_{1\overline{3}3\overline{1}} =-1,\\
\label{eq:4.18} && R^c_{2\overline{2}1\overline{1}} = R^c_{3\overline{3}1\overline{1}} = R^c_{1\overline{1}3\overline{3}} = R^c_{2\overline{2}3\overline{3}} =0, \ \ R^c_{1\overline{1}2\overline{2}} = R^c_{3\overline{3}2\overline{2}} =1.
\end{eqnarray}
In other words, at the origin, the Chern curvature matrix is
\begin{equation}
\Theta = \left[ \begin{array}{ccc}   2dz_{1\overline{1}}+dz_{2\overline{2}} & dz_{2\overline{1}} &  -dz_{3\overline{1}}  \\ dz_{1\overline{2}}  & 2 dz_{2\overline{2}}  & dz_{3\overline{2}}  \\
- dz_{1\overline{3}}  & dz_{2\overline{3}}  & dz_{2\overline{2}}+ 2dz_{3\overline{3}}  \end{array} \right]
\end{equation}
Here we wrote $dz_{i\overline{j}}$ for $dz_i\wedge d\overline{z}_j$. In particular, $\sqrt{-1} \,\mathrm{tr} \, \Theta = 2\sqrt{-1}(dz_{1\overline{1}}+2dz_{2\overline{2}}+dz_{3\overline{3}}) = 2\tilde{\omega}$ as expected. The bisectional (Griffiths) curvature of $g$ is $R^c_{X\overline{X}Y\overline{Y}} = \,^tY \Theta (X, \overline{X}) \overline{Y}$, which is equal  to
\begin{eqnarray*}
&&  2\sum_{i=1}^3 |X_iY_i|^2 + |X_2Y_1|^2+|X_2Y_3|^2 + 2\,\mbox{Re} (X_1\overline{X}_2\overline{Y}_1Y_2)  + \ 2\,\mbox{Re} (X_2\overline{X}_3\overline{Y}_2Y_3) - 2\,\mbox{Re} (X_1\overline{X}_3\overline{Y}_1Y_3) \\
& = & |X_2Y_1|^2+|X_2Y_3|^2 + |X_1\overline{Y}_1 + X_2\overline{Y}_2|^2 + |X_1\overline{Y}_1 - X_3\overline{Y}_3|^2 + |X_2\overline{Y}_2 + X_3\overline{Y}_3|^2 \ \geq \ 0.
\end{eqnarray*}
This also indicates that $g$ is not-K\"ahler since $X$ is not a Hermitian symmetric space. Note that when $X=Y$, the holomorphic sectional curvature is given by
\begin{equation*}
 R^c_{X\overline{X}X\overline{X}} =  |X_1X_2|^2 + |X_2X_3|^2 + (|X_1|^2  + |X_2|^2)^2 + (|X_1|^2 - |X_3|^2)^2 + (|X_2|^2 + |X_3|^2)^2,
\end{equation*}
which is positive for any $X\neq 0$, so $(X,g)$ does have positive holomorphic sectional curvature. Also, the first, second and third Chern Ricci form of $g$ are respectively
$$ \mbox{Ric}(\omega) = 2\tilde{\omega}, \ \ \  \mbox{Ric}^{(2)}(\omega) = 4\omega - \tilde{\omega}, \ \ \ \mbox{Ric}^{(3)} (\omega) = 2\tilde{\omega },$$
which are all positive definite, with the first and third Ricci equal to each other.

Next we verify that $(X,g)$ is balanced and {\em BTP}. First let us recall the formula under a natural frame. Suppose $(z_1, \ldots , z_n)$ is a local holomorphic coordinate on a Hermitian manifold $(M^n,g)$, and write $\varepsilon_i=\frac{\partial}{\partial z_i}$. Under the frame $\varepsilon$, which we view as a column vector, the Levi-Civita connection $\nabla$, Chern connection $\nabla^c$, and Bismut connection $\nabla^b$ are given by
$$ \nabla^c \varepsilon = \theta \varepsilon, \ \ \ \nabla^b \varepsilon = \theta^b \varepsilon, \ \ \ \nabla \varepsilon = \theta^{(1)} \varepsilon + \overline{\theta^{(2)}} \,\overline{ \varepsilon}. $$
It is well-known that $\theta =\partial g g^{-1}$, where $g=(g_{i\overline{j}})$. Denote by $T$ the torsion tensor of $\nabla^c$, and write
$T(\varepsilon_i , \varepsilon_k) = 2\sum_j T^j_{ik} \varepsilon_j$, then we have
\begin{equation} \label{eq:8.20}
 T^j_{ik} = \frac{1}{2}\sum_{\ell} (g_{k\overline{\ell},i } - g_{i\overline{\ell},k }) g^{\overline{\ell}j} .
 \end{equation}
Since $\nabla$ is torsion free, it holds
$$ 2\langle \nabla_xy,z\rangle = x\langle y,z\rangle + y\langle x,z\rangle -z \langle x,z\rangle +\langle [x,y],z\rangle - \langle [y,z],x\rangle - \langle [x,z],y\rangle $$
for any vector fields $x$, $y$, $z$ on $M$, so under the natural frame we have
\begin{equation} \label{eq:8.21}
 \theta^{(1)}_{ij} = \frac{1}{2}\sum_{k,\ell} \big( g_{k\overline{\ell},i } + g_{i\overline{\ell},k }\big)g^{\overline{\ell}j}  dz_k +  \frac{1}{2}\sum_{k,\ell} \big( g_{i\overline{\ell},\overline{k} } - g_{i\overline{k},\overline{\ell} }\big) g^{\overline{\ell}j} d\overline{z}_k.
 \end{equation}
By the relation $\theta^b=2\theta^{(1)} -\theta$, we get
\begin{equation} \label{eq:8.22}
 \theta^b_{ij} = \sum_{k,\ell} g_{k\overline{\ell},i } g^{\overline{\ell}j}  dz_k +  2\sum_{r,k,\ell} g_{i\overline{r}}\overline{T^r_{k\ell} }  g^{\overline{\ell}j} d\overline{z}_k.
 \end{equation}
The {\em BTP} condition is given by
\begin{equation} \label{eq:8.23}
\nabla^bT=0 \ \Longleftrightarrow \  dT^j_{ik} =\sum_r \big( \theta^b_{ir} T^j_{rk} + \theta^b_{kr} T^j_{ir} - \theta^b_{rj} T^r_{ik} \big), \ \forall \ i,j,k.
 \end{equation}
If $g_{i\overline{j}}=\delta_{ij}$ at the origin $0$, then by (\ref{eq:8.22}) the {\em BTP} condition at $0$  is given by \begin{eqnarray}
\frac{\partial}{\partial z_{\ell}} T^j_{ik} & = & \sum_r \big(   g_{\ell \overline{r},i} T^j_{rk} +   g_{\ell \overline{r},k} T^j_{ir} -  g_{\ell \overline{j},r} T^r_{ik}   \big) \label{eq:8.24} \\
 \frac{\partial}{\partial \overline{z}_{\ell}} T^j_{ik} & = & 2 \sum_r \big(   T^j_{ir} \overline{ T^k_{\ell r}} -   T^j_{kr} \overline{ T^i_{\ell r}} +  T^r_{ik} \overline{ T^r_{j\ell }}   \big)  \label{eq:8.25}
\end{eqnarray}
Now let us check for our Wallach space $(X,g)$. At the origin, we have $g_{i\overline{j}}=\delta_{ij}$, and all $g_{i\overline{j},k} =0$ except $g_{3\overline{2},1} =1$, so by (\ref{eq:8.20}) we know that all components of $T$ vanishes except $T^2_{13}=\frac{1}{2}$. In particular, the torsion $1$-form $\eta=0$ as $\eta_k = \sum_i T^i_{ik}$.

For (\ref{eq:8.24}), the right hand side is zero because for each of these three terms, one of the two factors is zero when $r$ is $2$ or not $2$. Its left hand side at $0$ is equal to $\frac{1}{2}(g_{k\overline{j},i\ell} - g_{i\overline{j},k\ell})$, which is zero by the last equality in (\ref{eq:8.11}). For (\ref{eq:8.25}), twice of its left hand side  at $0$ is given by
$$  (g_{k\overline{j},i\overline{\ell}} - g_{i\overline{j},k\overline{\ell}}) - (g_{k\overline{2},i} - g_{i\overline{2},k} ) \overline{g_{j\overline{2},\ell} } .$$
When $\{i,k\}\neq \{j,\ell\}$, both sides of (\ref{eq:8.25}) are zero. The same is true when $i=j=k=\ell$, so we just need to check the $i\neq k$ and $\{ i,k\} =\{ j,\ell\}$ case. Assume that $i=j\neq k=\ell$. Then the twice of the left hand side of (\ref{eq:8.25}) at $0$ is equal to
$$  g_{k\overline{i},i\overline{k}} - g_{i\overline{i},k\overline{k}} + \delta_{i3}\delta_{k1} =  -\sigma_{k\overline{i},i\overline{k}} + \sigma_{i\overline{i},k\overline{k}} + \delta_{i3}\delta_{k1} = - \delta_{i2}(1+\delta_{k2}) + \delta_{i1}\delta_{k3} + \delta_{i3}\delta_{k1}.$$
Here we used the fact that $\tilde{g}_{k\overline{i},i\overline{k}} = \tilde{g}_{i\overline{i},k\overline{k}}$ and $\sigma_{k\overline{i},i\overline{k}}=0$. In the mean time, twice of the right hand side of (\ref{eq:8.25}) is
$$ 4 \sum_r ( T^i_{ir} \overline{T^k_{kr}} - |T^i_{kr}|^2 + |T^r_{ik}|^2) = -\delta_{i2}(\delta_{k1}+\delta_{k3}) + \delta_{i1}\delta_{k3} +\delta_{i3}\delta_{k3} .$$
Note that $\delta_{k1}+\delta_{k3} =1-\delta_{k2}$. The two sides are differed by $\delta_{i2}\delta_{k2}$, which is zero since $i\neq k$. So (\ref{eq:8.25}) holds in this case. The $i=\ell \neq k=j$ case can be verified similarly. So the Wallach threefold $(X,g)$ is indeed non-K\"ahler, balanced, and {\em BTP}.

In the remaining part of this section, we will verify that $(X,g)$ has constant Ricci curvature and non-negative sectional curvature for its Levi-Civita connection $\nabla$.

First let us recall some general formula from existing literature. Let $e$ be a local unitary frame on a Hermitian manifold $(M^n,g)$. We have
$$ \nabla^c_{\ell}T^j_{ik} - \nabla^b_{\ell}T^j_{ik} = 2\sum_r \big( T^j_{rk}T^r_{i\ell} + T^j_{ir}T^r_{k\ell} - T^r_{ik}T^j_{r\ell} \big). $$
By Proposition \ref{Tderivative}, we know that $\nabla^b_{1,0}T=0$ implies that the right hand side of the above equality is zero. Therefore we always have $\nabla^b_{1,0}T=0 \ \Longrightarrow \ \nabla^c_{1,0}T=0$. So if $g$ is {\em BTP}, then the $(1,0)$-part of Chern covariant differentiation of torsion vanishes: $T^j_{ik;\ell}=0$. The $(0,1)$-part of Chern covariant differentiation of torsion, on the other hand, is given by
$$ T^j_{ik;\overline{\ell}}= \frac{1}{2} \big( R^c_{k\overline{\ell}i\overline{j}} - R^c_{i\overline{\ell}k\overline{j}}\big) .$$
Denote by $R$ the curvature tensor of the Levi-Civita (Riemannian) connection of $g$. By \cite[Lemma 7]{YZ18Cur}, and using the above equality to replace the $(0,1)$-derivatives of $T$, we obtain
\begin{eqnarray}
R_{ijk\overline{\ell}} & = & T^{\ell}_{ij;k} + \sum_r \big( T^{\ell}_{ri}T^r_{jk} - T^{\ell}_{rj}T^r_{ik} \big)    \label{eq:8.26} \\
R_{k\overline{\ell}i\overline{j}} & = & \frac{1}{2}\big( R^c_{i\overline{\ell}k\overline{j}}  +  R^c_{k\overline{j}i\overline{\ell}}  \big) + \sum_r \big( T^r_{ik}\overline{T^r_{j\ell } } - T^j_{kr}\overline{T^i_{\ell r} } - T^{\ell}_{ir}\overline{T^k_{jr } } \big) \label{eq:8.27}
\end{eqnarray}
Now let us specialize to the Wallach threefold $(X,g)$ at the origin $0$. We have all $T^j_{ij}=0$ except $T^2_{13}=\frac{1}{2}$, so the right hand side of (\ref{eq:8.26}) is zero, hence $R_{ijk\overline{\ell}} =0$. By the information on $R^c_{i\overline{j}k\overline{\ell}}$, we get $R_{i\overline{j}k\overline{\ell}}=0$ if $\{ i,k\}\neq \{ j,\ell\}$, $R_{i\overline{i}i\overline{i}}=2$ for each $i$, and
\begin{equation} \label{eq:8.28}
R_{1\overline{1}2\overline{2}} = R_{3\overline{3}2\overline{2}} = - R_{1\overline{1}3\overline{3}}  = \frac{3}{4}, \ \ \ \ \  R_{1\overline{2}2\overline{1}} = R_{3\overline{2}2\overline{3}}   = \frac{1}{2}, \ \ \ \ \ R_{1\overline{3}3\overline{1}} = - \frac{1}{4}.
\end{equation}
 Note that for Riemannian curvature $R$ we always have  $R_{i\overline{j}k\overline{\ell}}=R_{k\overline{\ell}i\overline{j}}$. Now we compute the sectional curvature of $R$. Let $x$, $y$ be any two real tangent vector of $X$ at the origin $0$, with $x\wedge y\neq 0$. Write $x=X+\overline{X}$ and $y=Y+\overline{Y}$ for type $(1,0)$ tangent vectors $X$ and $Y$.  We have
 $$ x\wedge y = X\wedge Y + \overline{X} \wedge \overline{Y} + \big( X\wedge \overline{Y} - Y \wedge \overline{X} \big). $$
 By Gray's theorem, $R_{XYZW}=0$ for any type $(1,0)$ tangent vectors $X$, $Y$, $Z$, $W$. Also, we have shown that $R_{XYZ\overline{W}}=0$ for our manifold. By the first Bianchi identity, we have
 $$ - R_{XY\overline{X}\overline{Y}} = - R_{X\overline{X}Y\overline{Y}} + R_{X\overline{Y}Y\overline{Y}}. $$
 Thus we have
\begin{eqnarray*}
 R_{xyyx} & = &  - R (x\wedge y, x\wedge y) \ \, = \ \,  -2 R(X\wedge Y, \overline{X} \wedge \overline{Y} ) - R( X\wedge \overline{Y} - Y \wedge \overline{X} , X\wedge \overline{Y} - Y \wedge \overline{X}) \\
 & = & -2 R_{X\overline{X}Y\overline{Y}} + 2 R_{X\overline{Y}Y\overline{X}} - R( X\wedge \overline{Y} - Y \wedge \overline{X} , X\wedge \overline{Y} - Y \wedge \overline{X}) \\
 & = & -2 R_{X\overline{X}Y\overline{Y}} + 4 R_{X\overline{Y}Y\overline{X}} - 2\mbox{Re} \{ R_{X\overline{Y}X\overline{Y}} \}
\end{eqnarray*}
For $i\neq k$, let us write $R_{i\overline{i}k\overline{k}}=a_{ik}$ and $R_{i\overline{k}k\overline{i}}=b_{ik}$. We have
\begin{equation} \label{eq:8.29}
2 b_{ik} - a_{ik}= \frac{1}{4}, \ \ \  \ 2a_{ik}- b_{ik} = \left\{  \begin{array}{ll} -\frac{5}{4}, \ \mbox{if} \  \{i,k\} =\{ 1,3\}  \\ \ 1, \ \ \ \mbox{otherwise} \end{array}\right.,\ \ \ \      a_{ik}+ b_{ik} = \left\{  \begin{array}{ll} -1, \ \mbox{if} \  \{i,k\} =\{ 1,3\}  \\ \ \frac{5}{4}, \ \ \ \mbox{otherwise} \end{array} \right.
\end{equation}
Continue with the calculation of $R_{xyyx}$ above, we have
\begin{eqnarray*}
 R_{xyyx} & = &  \sum_{i,j,k,\ell} R_{i\overline{j}k\overline{\ell}} \,\{ -2X_i\overline{X}_j Y_k \overline{Y}_{\ell} + 4 X_i\overline{Y}_j Y_k \overline{X}_{\ell} -  2\mbox{Re} (X_i\overline{Y}_j X_k \overline{Y}_{\ell}) \}  \\
 & = & \sum_i 2 \{ 2|X_iY_i|^2 - 2\mbox{Re} (X_i^2\overline{Y}_i^2)\} + \sum_{i\neq k} a_{ik} \{-2|X_iY_k|^2 + 4 X_i\overline{Y}_iY_k\overline{X}_k  - 2\mbox{Re} (X_i\overline{Y}_iX_k\overline{Y}_k ) \} + \\
 & & + \sum_{i\neq k} b_{ik} \{ -2  X_i\overline{X}_kY_k\overline{Y}_i + 4 |X_iY_k|^2 - 2\mbox{Re} (X_i\overline{Y}_kX_k\overline{Y}_i ) \} \\
 & = & 4\sum_i \{ |X_iY_i|^2 - \mbox{Re} (X_i^2\overline{Y}_i^2)\} +  2\sum_{i<k} \{ (2b_{ik}-a_{ik})(|X_iY_k|^2 + |X_kY_i|^2) \} + \\
&&  + \, 2\sum_{i<k} \{ (2a_{ik}-b_{ik}) 2\mbox{Re} (X_i\overline{X}_k \overline{Y}_i Y_k ) \} -2\sum_{i<k} \{ (a_{ik}+b_{ik}) 2\mbox{Re} (X_iX_k\overline{Y}_i \overline{Y}_k )  \} \\
& = & 2\sum_{i<k} F_{ik},
\end{eqnarray*}
where
\begin{eqnarray*}
F_{ik} & = & |X_iY_i|^2 - \mbox{Re} (X_i^2\overline{Y}_i^2) + |X_kY_k|^2 - \mbox{Re} (X_k^2\overline{Y}_k^2) + \frac{1}{4}(|X_iY_k|^2 + |X_kY_i|^2) + \\
&& + \, (2a_{ik}-b_{ik}) 2\mbox{Re} (X_i\overline{X}_k \overline{Y}_i Y_k ) -  (a_{ik}+b_{ik}) 2\mbox{Re} (X_iX_k\overline{Y}_i \overline{Y}_k ).
\end{eqnarray*}
For $(ik)=(13)$, $(2a-b)=-\frac{5}{4}$ and $(a+b)=-1$, so we have
\begin{eqnarray*}
F_{ik} & = & |X_iY_i|^2 - \mbox{Re} (X_i^2\overline{Y}_i^2) + |X_kY_k|^2 - \mbox{Re} (X_k^2\overline{Y}_k^2) + \frac{1}{4}(|X_iY_k|^2 + |X_kY_i|^2) + \\
&& - \, \frac{5}{4} 2\mbox{Re} (X_i\overline{X}_k \overline{Y}_i Y_k ) + 2\mbox{Re} (X_iX_k\overline{Y}_i \overline{Y}_k )\\
& = & 2\{\mbox{Im}(X_i\overline{Y}_i)\}^2 + 2\{\mbox{Im}(X_k\overline{Y}_k)\}^2 + \frac{1}{4} |X_iY_k - X_kY_i |^2 -4 \mbox{Im}(X_i\overline{Y}_i)\mbox{Im}(X_k\overline{Y}_k) \\
& = & 2 \{ \mbox{Im}(X_i\overline{Y}_i) - \mbox{Im}(X_k\overline{Y}_k) \}^2 + \frac{1}{4} |X_iY_k - X_kY_i |^2 \ \geq \ 0.
\end{eqnarray*}
Similarly, for $(ik)=(12)$ or $(23)$, $2a-b=1$, $a+b=\frac{5}{4}$, so
\begin{eqnarray*}
F_{ik} & = & |X_iY_i|^2 - \mbox{Re} (X_i^2\overline{Y}_i^2) + |X_kY_k|^2 - \mbox{Re} (X_k^2\overline{Y}_k^2) + \frac{1}{4}(|X_iY_k|^2 + |X_kY_i|^2) + \\
&& + \,  2\mbox{Re} (X_i\overline{X}_k \overline{Y}_i Y_k ) - \frac{5}{4} 2\mbox{Re} (X_iX_k\overline{Y}_i \overline{Y}_k )\\
& = & 2\{\mbox{Im}(X_i\overline{Y}_i)\}^2 + 2\{\mbox{Im}(X_k\overline{Y}_k)\}^2 + \frac{1}{4} |X_i\overline{Y}_k - Y_i\overline{X}_k |^2 + 4 \mbox{Im}(X_i\overline{Y}_i)\mbox{Im}(X_k\overline{Y}_k) \\
& = & 2 \{ \mbox{Im}(X_i\overline{Y}_i) + \mbox{Im}(X_k\overline{Y}_k) \}^2 + \frac{1}{4} |X_i\overline{Y}_k - Y_i\overline{X}_k |^2 \ \geq \ 0.
\end{eqnarray*}
That is, $R_{xyyx}$ is equal to
\begin{equation*}
4(I_1+I_2)^2+4(I_2+I_3)^2 + 4(I_1-I_3)^2 +\frac{1}{2}|X_1\overline{Y}_2-Y_1\overline{X}_2|^2 +\frac{1}{2}|X_2\overline{Y}_3-Y_2\overline{X}_3|^2 +\frac{1}{2}|X_1Y_3-Y_1X_3|^2,
\end{equation*}
where $I_i=\mbox{Im}(X_i\overline{Y}_i)$. Therefore the Levi-Civita connection of $g$ has non-negative sectional curvature. We note that the sectional curvature is not strictly positive here: if we take $X_1=X_2=X_3\in {\mathbb R}\setminus \{ 0\}$, and $Y_1=\overline{Y}_2=Y_3=\rho \not\in {\mathbb R}$, then $I_1=I_3=-I_2$ and the above expression vanishes, so we get $R_{xyyx}=0$ yet $x\wedge y\neq 0$.

To see the Ricci curvature of $g$, let $Y=e_i$, then the above formula becomes
$$ R_{xyyx} = 4|X_i|^2 - 4\mbox{Re}(X_i^2) +  |X_j|^2 +  |X_k|^2,
$$
where $\{ i,j,k\} =\{ 1,2,3\}$. Similarly, if we let $Y=\sqrt{-1}e_i$, then we get
$$ R_{xyyx} = 4|X_i|^2 + 4\mbox{Re}(X_i^2) +  |X_j|^2 +  |X_k|^2.
$$
Add up the above two equalities for $i$ from $1$ to $3$, we get $12|X|^2=6|x|^2$. Let $\varepsilon_i = \frac{1}{\sqrt{2}}(e_i+\overline{e}_i)$ and  $\varepsilon_{i^{\ast}} = \frac{\sqrt{-1}}{\sqrt{2}}(e_i-\overline{e}_i)$. Then $\{ \varepsilon_i, \varepsilon_{i^{\ast}}\}$ form an orthonormal tangent frame, so the Ricci curvature of the Riemannian metric $g$ is
\begin{equation}
\mbox{Ric}(x) = \frac{1}{|x|^2} \sum_i \big( R_{x\varepsilon_i\varepsilon_ix} + R_{x\varepsilon_{i^{\ast}}\varepsilon_{i^{\ast}}x} \big)= \frac{1}{|x|^2} 6|x|^2 = 6.
\end{equation}
That is, $g$ is an Einstein metric on $X$ with Ricci curvature $6$.

\vspace{0.3cm}

\section{Balanced {\em BTP} threefolds of middle type}\label{mddtype3D}
\subsection{The structural theorem}\label{str}

In this subsection, let us consider the $\mathrm{rank}B=2$ case. Throughout this subsection we will assume that $(M^3,g)$ is a compact, balanced {\em BTP} manifold of middle type. Denote by $L$ the kernel of the $B$ tensor.

\begin{claim} \label{claim9.1}
$L$ is a holomorphic line bundle on $M^3$ satisfying $L^{\otimes 2}\cong {\mathcal O}_M$.
\end{claim}

\begin{proof}
Here and below ${\mathcal O}_M$ denotes the trivial line bundle of $M$. Let us scale the metric $g$ by a suitable constant multiple to make $a_1=a_2=\frac{1}{2}$ from now on. It follows from (\ref{eq:6.5}) and  (\ref{eq:6.6}) that $\theta^b_{13}=\theta_{23}^b=\theta^b_{33}=0$,  $\theta^b_{11}=\theta^b_{22}$ and $\theta^b_{12}+\theta^b_{21}=0$, which implies
\[\gamma = \frac{1}{2}\begin{bmatrix} 0 & -(\varphi_3+\bar{\varphi}_3) & \overline{\varphi}_2 \\
                                      \varphi_3 + \bar{\varphi}_3 & 0 & - \bar{\varphi}_1 \\
                                       -\varphi_2 & \varphi_1 & 0 \end{bmatrix}\!\!, \quad
\theta^b =  \begin{bmatrix} \alpha & \beta_0  & 0 \\ -\beta_0 & \alpha &  0  \\ 0 & 0 & 0 \end{bmatrix}\!\!, \quad
\theta =  \begin{bmatrix} \alpha & \beta  &  - \overline{\varphi}_2 \\ -\beta & \alpha &  \overline{\varphi}_1  \\ \varphi_2 & -\varphi_1 & 0 \end{bmatrix}\!\! ,  \quad
\tau =  \begin{bmatrix} \varphi_2\varphi_3 \\ \varphi_3\varphi_1  \\ 0  \end{bmatrix}\!\!,\]
where $\overline{\alpha}=-\alpha$, $\overline{\beta}_0=\beta_0$, and $\beta = \beta_0 +\varphi_3 + \overline{\varphi}_3$. This means that $\nabla^be_3=0$ for any special frame $e$. Since $\nabla^c_{\overline{i}}e_3=\sum_j \theta_{3j}(\bar{e}_i)e_j =0$, so $e_3$ is a local holomorphic vector field on $M^3$ which is  also a local section of $L$, thus $L$ is a holomorphic line subbundle of the holomorphic tangent bundle $T^{1,0}M$.

To see that $L^{\otimes 2}={\mathcal O}_M$, let us consider the structure equation $d\varphi =-\,^t\!\theta \wedge \varphi + \tau$. In our case it becomes
\begin{equation} \label{eq:9.1}
d\varphi = \begin{bmatrix} -\alpha \varphi_1 + \beta \varphi_2 \\ - \beta\varphi_1 - \alpha \varphi_2 \\ \varphi_2\overline{\varphi}_1 -\varphi_1\overline{\varphi}_2   \end{bmatrix}
\end{equation}
For convenience, $\varphi_i \wedge \overline{\varphi}_j$ is abbreviated as $\varphi_{i\overline{j}}$.
The curvature matrices of $\nabla^b$ and $\nabla^c$ under $e$ are
\begin{equation} \label{eq:9.2}
\Theta^b = \begin{bmatrix} d\alpha & d\beta_0  & 0 \\ -d\beta_0 & d\alpha &  0  \\ 0 & 0 & 0 \end{bmatrix}\!\!, \quad
\Theta = \begin{bmatrix} d\alpha - \varphi_{2\overline{2}} & d\beta + \varphi_{1\overline{2}} & 0 \\ -d\beta +\varphi_{2\overline{1}} & d\alpha -\varphi_{1\overline{1}} &  0  \\ 0 & 0 & \varphi_{1\overline{1}} + \varphi_{2\overline{2}} \end{bmatrix}\!\!.
\end{equation}

\begin{remark}
By the above formula for $\Theta^b$, we know that the holonomy group of any balanced BTP threefold of middle type is always contained in the subgroup of $U(2)\times 1 $ which commutes with $\begin{pmatrix} 0 & 1 \\ -1 & 0 \end{pmatrix} \times 1$, in $U(3)$.
\end{remark}

By taking the exterior differentiation of (\ref{eq:9.1}), we get the first Bianchi identity
\begin{equation} \label{eq:9.3}
d\alpha \wedge \varphi_1 - d\beta \wedge \varphi_2 =0, \ \ \ \ \ d\alpha \wedge \varphi_2 + d\beta \wedge \varphi_1 =0.
\end{equation}
It also holds that, from \eqref{eq:9.1},
\begin{equation} \label{eq:9.4}
d(\varphi_1\overline{\varphi}_1)  = -d(\varphi_2\overline{\varphi}_2)  = \beta (\varphi_2\overline{\varphi}_1 +\varphi_1\overline{\varphi}_2), \ \ \ \ \ \ d\beta -d\beta_0 = 2(\varphi_2\overline{\varphi}_1 -\varphi_1\overline{\varphi}_2).
\end{equation}
It follows that the Riemannian connection $\nabla$ under $e$ is
\[\begin{aligned}
\nabla e &= (\theta^b-\gamma)e + \overline{\theta_2} \overline{e} \\
         &= \begin{bmatrix} \alpha & \beta_0 + \frac{1}{2}(\varphi_3 + \overline{\varphi}_3) & -\frac{1}{2} \overline{\varphi}_2 \\
                            - \beta_0 - \frac{1}{2}(\varphi_3 + \overline{\varphi}_3) & \alpha & \frac{1}{2} \overline{\varphi}_1 \\
                            \frac{1}{2}\varphi_2 & -\frac{1}{2}\varphi_1 & 0 \end{bmatrix}\!\!e
            +\begin{bmatrix} 0 & 0 & -\frac{1}{2} \overline{\varphi}_2 \\
                             0 & 0 & \frac{1}{2} \overline{\varphi}_1 \\
                            \frac{1}{2}\overline{\varphi}_2 & -\frac{1}{2}\overline{\varphi}_1 & 0 \end{bmatrix}\!\! \overline{e},
\end{aligned}\]
since $(\theta_2)_{ij}= \sum_k \overline{T_{ij}^k} \varphi_k$. In particular, $\nabla e_3= \overline{\nabla e_3}$ is real. So we know that $e_3$ can only vary by a sign. In other words, $\varphi_3\otimes \varphi_3$ is globally defined on $M$, hence $L^{\otimes 2}$ is trivial. This proves Claim  \ref{claim9.1}.
\end{proof}

Note that since $T^3_{ik}=0$ for any $i,k$, it yields  $[e_3, \overline{e}_3] = \nabla^c_{e_3}\overline{e}_3 - \nabla^c_{\overline{e}_3}e_3 =0$, so the distribution $L$ generated by $e_3$ is actually a holomorphic foliation, which is also a flat holomorphic bundle using the restriction metric of $g$.

If $L={\mathcal O}_M$, let $\hat{M}=M$. If $L\neq {\mathcal O}_M$, then by the cyclic covering lemma, it gives us an unbranched double cover  $\pi : \hat{M}\rightarrow M$ so that $\pi^{\ast}L$ is trivial. Thus $\hat{M}$ is either $M$ itself, or a double cover of $M$. Lift the metric $g$ to $\hat{M}$, which we will still denote by $g$. Then $(\hat{M},g,J)$ is a balanced {\em BTP} threefold. The kernel of the $B$ tensor for $\hat{M}$ is now the trivial line bundle $L={\mathcal O}_{\hat{M}}$, or equivalently, we may choose special frames $e$ on $\hat{M}$ so that $e_3$ is globally defined. We claim that

\begin{claim} \label{claim9.2}
There exists a complex structure $I$ on $\hat{M}$ which is compatible with $g$, so that $I\circ J = J \circ I$, and  $(\hat{M}, g,I)$ is a Vaisman threefold.
\end{claim}

\begin{proof} Let $e$ be a special frame on $\hat{M}$, where $e_3$ is a global holomorphic vector field of unit length on $\hat{M}$. Write  $e_3=\frac{1}{\sqrt{2}}(\xi - \sqrt{-1}\xi')$, or equivalently,
\[\xi = \frac{1}{\sqrt{2}} (e_3 + \overline{e}_3), \quad \quad \xi' = \frac{\sqrt{-1}}{\sqrt{2}} (e_3 - \overline{e}_3), \]
Then $\xi$ and $\xi'$ are global real vector fields of unit length satisfying $J\xi = \xi'$. We have $\nabla \xi'=0$ and
\begin{equation}\label{5mfd} \begin{cases}
\nabla \xi = \frac{1}{\sqrt{2}}(\varphi_2 e_1 - \varphi_1 e_2 + \overline{\varphi}_2 \overline{e}_1 - \overline{\varphi}_1 \overline{e}_2) \\
\nabla e_1 = \alpha e_1 + (\beta_0 + \frac{1}{2}(\varphi_3 + \overline{\varphi}_3)) e_2 - \frac{1}{\sqrt{2}} \overline{\varphi}_2 \xi \\
\nabla e_2 = (-\beta_0-\frac{1}{2}(\varphi_3 + \overline{\varphi}_3)) e_1 + \alpha e_2 + \frac{1}{\sqrt{2}} \overline{\varphi}_1 \xi
\end{cases}\end{equation}
Note that $e_1$, $e_2$ are still local here. Write
$$ e_1=\frac{1}{\sqrt{2}} (\varepsilon_1 - \sqrt{-1} \varepsilon_2), \ \ \ \ \ \  e_2=\frac{1}{\sqrt{2}} (\varepsilon_3 - \sqrt{-1} \varepsilon_4). \ \ \ $$
Then $\{ \varepsilon_1, \varepsilon_2, \varepsilon_3, \varepsilon_4, \xi , \xi'\}$ becomes a local orthonormal frame of $(\hat{M}, g)$, and $J\varepsilon_1 =\varepsilon_2$, $J\varepsilon_3 =\varepsilon_4$, $J\xi =\xi'$. Let us also write
$$ \varphi_1 =\frac{1}{\sqrt{2}} ( \phi_1 + \sqrt{-1} \phi_2 ) , \ \ \ \ \ \ \varphi_2 =\frac{1}{\sqrt{2}} ( \phi_3 + \sqrt{-1} \phi_4 ), \ \ \ \ \varphi_3= \frac{1}{\sqrt{2}} ( \phi_5 + \sqrt{-1} \phi_6 ) . $$
Then $\{ \phi_1, \phi_2, \phi_3, \phi_4, \phi_5 , \phi_6\}$ becomes the dual coframe. For any real vector field $X$ on $\hat{M}$, by (\ref{5mfd}) we have
$$ \sqrt{2} \nabla_X\xi = \phi_3(X)\varepsilon_1 + \phi_4(X)\varepsilon_2 - \phi_1(X)\varepsilon_3 - \phi_2(X)\varepsilon_4.$$
Let us denote by $V$ the distribution in $\hat{M}$ spanned by $\varepsilon_i$, $1\leq i\leq 4$. Denote by $E'$ the distribution spanned by $\xi'$, and by $E=V\oplus {\mathbb R}\xi$ the orthogonal complement of $E'$. Then both $E'$ and $E$ are globally defined, parallel under $\nabla$, and $V$ is also globally defined since it is the orthogonal complement of $\xi$ in $E$. Define an orthogonal almost complex structure $I$ on $(\hat{M},g)$ by
$$ I(\varepsilon_3) = \varepsilon_1, \ \ \ \ \ \  I(\varepsilon_4) = \varepsilon_2, \ \ \ \ \ I(\xi)=\xi'. $$
Then $I$ is globally defined since for any $X\in V$ it is given by $I(X)=\sqrt{2}\nabla_X\xi$, and clearly, $I\circ J = J\circ I$. In order to show that $I$ is integrable and $(\hat{M},g,I)$ is Vaisman, let us carry out the computation in complex terms. Note that a standard orthonormal frame for $I$ becomes $\{ \varepsilon_3, \varepsilon_1, \varepsilon_4 , \varepsilon_2, \xi , \xi'\}$. In particular, the orientation induced by $I$ is opposite to that by $J$. Let us write $\beta'=\beta_0+\frac{1}{2}(\varphi_3+\overline{\varphi}_3)$. Since $\overline{\alpha}=-\alpha$ and $\overline{\beta'}=\beta'$, by (\ref{5mfd}) we obtain
\begin{equation*} \begin{cases}
 \nabla \varepsilon_1 = -\sqrt{-1}\alpha \varepsilon_2 + \beta' \varepsilon_3 - \frac{1}{\sqrt{2}} \phi_3 \xi \\
\nabla \varepsilon_2 = \sqrt{-1}\alpha \varepsilon_1 + \beta' \varepsilon_4 - \frac{1}{\sqrt{2}} \phi_4 \xi \\
\nabla \varepsilon_3 = -\beta' \varepsilon_1 - \sqrt{-1}\alpha  \varepsilon_4 + \frac{1}{\sqrt{2}} \phi_1 \xi \\
\nabla \varepsilon_4 = -\beta' \varepsilon_2 + \sqrt{-1}\alpha  \varepsilon_3 + \frac{1}{\sqrt{2}} \phi_2 \xi
\end{cases}\end{equation*}
Let us denote by
$$ u_1=\frac{1}{\sqrt{2}} ( \varepsilon_3 - \sqrt{-1} \varepsilon_1), \ \ \ \  u_2=\frac{1}{\sqrt{2}} ( \varepsilon_4 - \sqrt{-1} \varepsilon_2), $$
then $\{ u_1, u_2, e_3\}$ forms a local unitary frame of $(\hat{M},g,I)$. Its dual coframe is $\{ \psi_1, \psi_2, \varphi_3 \}$, where
$$ \psi_1 = \frac{1}{\sqrt{2}} (\phi_3+\sqrt{-1} \phi_1 ), \ \ \ \ \ \psi_2 = \frac{1}{\sqrt{2}} (\phi_4+ \sqrt{-1} \phi_2). $$
By a straight forward computation, we get
\begin{equation} \label{eq:nabla-u} \begin{cases}
 \nabla u_1 = -\sqrt{-1}\beta' u_1 -   \sqrt{-1} \alpha u_2+ \frac{\sqrt{-1}}{\sqrt{2}} \overline{\psi}_1 \xi \\
 \nabla u_2 = \sqrt{-1}\alpha  u_1 -   \sqrt{-1} \beta' u_2+ \frac{\sqrt{-1}}{\sqrt{2}} \overline{\psi}_2 \xi \\
\nabla e_3 = \frac{\sqrt{-1}}{2} \big ( \psi_1 u_1 + \psi_2 u_2 - \overline{\psi}_1 \overline{u}_1 -  \overline{\psi}_2 \overline{u}_2 \big)
\end{cases} \end{equation}
To show that $I$ is integrable, we need to verify that $[T^{1,0}_I, T^{1,0}_I] \subseteq T^{1,0}_I$, or equivalently, each $[u_i, u_j]$ is a combination of $u_k$ for $1\leq k\leq 3$, but without  $\overline{u}_k$ terms. Here we have written $u_3=e_3$. First, for $[u_1, u_2]$, by (\ref{eq:nabla-u}) we have
$$ [u_1, u_2] = \nabla_{u_1}u_2 - \nabla_{u_2}u_1= (\ast) \,u_1 + (\ast)\, u_2.$$
Also, by (\ref{eq:nabla-u}), $\nabla_{u_1}e_3= \frac{\sqrt{-1}}{2}u_1$, so
$$ [u_1, e_3] = \nabla_{u_1}e_3 - \nabla_{e_3}u_1 = \frac{\sqrt{-1}}{2}u_1 - (\ast )\, u_1 - (\ast )\, u_2 $$
is a combination of $u_1$ and $u_2$. Similarly, $[u_2, e_3]$ is a combination of $u_1$ and $u_2$. Therefore, the almost complex structure $I$ is integrable.

Next let us show that the Hermitian threefold $(\hat{M}, g,I)$ is Vaisman. Denote by $u$ the column vector of the local unitary frame, where $u_3=e_3$, and by $\psi$ the column vector of its dual coframe, where $\psi_3=\varphi_3$. By (\ref{eq:nabla-u}) the Levi-Civita connection $\nabla$ has connection matrices $\nabla u = \hat{\theta}_1 u + \overline{\hat{\theta}_2} \,\overline{u}$ where
$$ \hat{\theta}_1 = \sqrt{-1} \begin{bmatrix} -\beta' & -\alpha & \frac{1}{2} \overline{\psi}_1 \\  \alpha & -\beta' & \frac{1}{2} \overline{\psi}_2 \\  \frac{1}{2}  \psi_1 &  \frac{1}{2}  \psi_2 & 0 \end{bmatrix} ,
\ \ \ \ \ \ \hat{\theta}_2 = \sqrt{-1} \begin{bmatrix} 0 & 0 & -\frac{1}{2} \psi_1 \\  0 & 0 & -\frac{1}{2} \psi_2 \\  \frac{1}{2}  \psi_1 &  \frac{1}{2}  \psi_2 & 0 \end{bmatrix} .$$
Under the frame $u$, the only non-zero components of the Chern torsion of $(\hat{M}, I, g)$ are $\hat{T}^1_{13}=\hat{T}^2_{23}=\frac{\sqrt{-1}}{2}$, so the Gauduchon's torsion $1$-form is $\hat{\eta} =\sqrt{-1}\varphi_3$. By the structure equation, $\nabla \psi = - \,^t\!\hat{\theta}_1 \psi - \,^t\!\hat{\theta}_2\overline{\psi}$, we get
$$ \nabla \varphi_3 = \sum_j\{ -(\hat{\theta}_1)_{j3}\psi_j - (\hat{\theta}_2)_{j3}\overline{\psi}_j\} =\sqrt{-1} ( \psi_1\overline{\psi}_1 + \psi_2\overline{\psi}_2 ) .$$
Therefore,
$$ \nabla (\hat{\eta} + \overline{\hat{\eta}}) = \sqrt{-1}(\nabla\varphi_3-\nabla \overline{\varphi}_3)=0.$$
That is, the Lee form of $(\hat{M},g,I)$ is parallel under the Levi-Civita conneciton, which means that the Hermitian threefold is Vaisman.
\end{proof}

Next we show that the Bismut connection of $(\hat{M},g,J)$ and $(\hat{M},g,I)$ coincide:

\begin{claim} \label{claim9.3}
The Bismut connection of $(\hat{M},g,J)$ and $(\hat{M},g,I)$ coincide.
\end{claim}

\begin{proof}
Let us denote by $\nabla^b$, $\hat{\nabla}^b$ the Bismut connection of $(\hat{M},g,J)$ and $(\hat{M},g,I)$, respectively. We want to show that $\nabla^b=\hat{\nabla}^b$. Let $e$ be a special frame with $e_3$ global, and define $u_1$, $u_2$ as in the proof of Claim \ref{claim9.2}, then we have
\begin{equation} \label{eq:e-u} \begin{cases}
u_1=\frac{1}{2}(e_2 + \overline{e}_2 -\sqrt{-1}e_1 -\sqrt{-1}\overline{e}_1) \\
u_2=\frac{1}{2}(e_1 - \overline{e}_1 +\sqrt{-1}e_2 -\sqrt{-1}\overline{e}_2)
\end{cases}
\end{equation}
By the expression of the  matrix $\theta^b$ of $\nabla^b$ under $e$, we know that
$$ \nabla^be_3=0, \ \ \ \ \nabla^be_1 = \alpha e_1+\beta_0e_2, \ \ \ \ \nabla^be_2=-\beta_0e_1+\alpha e_2. $$
Using the fact $\overline{\alpha}=-\alpha$, $\overline{\beta}_0=\beta_0$, and (\ref{eq:e-u}), we get
\begin{eqnarray*}
 2\nabla^b u_1 & = & (-\beta_0e_1+\alpha e_2) + (-\beta_0\overline{e}_1-\alpha \overline{e}_2) -\sqrt{-1}(\alpha e_1+\beta_0e_2) -\sqrt{-1}( -\alpha \overline{e}_1+\beta_0\overline{e}_2) \\
 & = & \alpha(e_2-\overline{e}_2 -\sqrt{-1}e_1 +\sqrt{-1}\overline{e}_1 ) - \beta_0( e_1+\overline{e}_1 +\sqrt{-1}e_2 +\sqrt{-1}\overline{e}_2 ) \\
 & = & -\sqrt{-1}\alpha 2u_2 -\sqrt{-1}\beta_0 2u_1,
 \end{eqnarray*}
and $\nabla^bu_2$ can be computed similarly, so we get
\begin{equation} \label{eq:nablas-u}
\nabla^bu_1= -\sqrt{-1}\beta_0u_1 -\sqrt{-1}\alpha u_2, \ \ \ \ \nabla^bu_2= \sqrt{-1}\alpha u_1 -\sqrt{-1}\beta_0 u_2.
\end{equation}
On the other hand, under the unitary frame $u$ for $(\hat{M},g,I)$ with dual coframe $\psi$, the only non-zero components of the Chern torsion are $\hat{T}^1_{13}=\hat{T}^2_{23}=\frac{\sqrt{-1}}{2}$, hence the $\gamma$-tensor is given by
$$ \hat{\gamma} = \frac{\sqrt{-1}}{2} \begin{bmatrix} \varphi_3 + \overline{\varphi}_3 & 0 & -\overline{\psi}_1 \\  0 & \varphi_3 + \overline{\varphi}_3 & -\overline{\psi}_2 \\  -  \psi_1 &  - \psi_2 & 0 \end{bmatrix}  .$$
So by the expression for $\hat{\theta}_1$, we get the connection matrix $\hat{\theta}^b$ of $\hat{\nabla}^b $
$$ \hat{\theta}^b=\hat{\theta}_1 + \hat{\gamma} = \sqrt{-1} \begin{bmatrix} -\beta_0 & -\alpha & 0 \\  \alpha & -\beta_0 & 0 \\  0 &  0 & 0 \end{bmatrix}. $$
Here we used the fact that $\beta'=\beta_0+\frac{1}{2}(\varphi_3 + \overline{\varphi}_3)$. This means that
\begin{equation*}
\hat{\nabla}^b u_1 = -\sqrt{-1}\beta_0u_1 -\sqrt{-1}\alpha u_2, \ \ \ \ \hat{\nabla}^bu_2= \sqrt{-1}\alpha u_1 -\sqrt{-1}\beta_0 u_2, \ \ \ \ \hat{\nabla}^be_3=0.
\end{equation*}
Compare this with (\ref{eq:nablas-u}), we conclude that $\nabla^b=\hat{\nabla}^b$. This completes the proof of Claim \ref{claim9.3}.
\end{proof}

By the nice structure theorem of Ornea and Verbitsky \cite{OV} for compact Vaisman manifolds, we know that $\hat{M}$ is a smooth fibration over the circle $S^1$ with fiber being a compact Sasakian manifold $N$. On the universal covering level, since $\nabla \xi'=0$, we know that the universal cover $\tilde{M}$ of $M$ is a Riemannian product $\tilde{N}\times {\mathbb R}$, and $\tilde{N}$ is a Sasakian manifold, since $\xi$ is a unit vector field and $\sqrt{2}\nabla\xi$ gives the complex structure $I$ on the orthogonal complement $V={\xi}^{\perp}$ in $T\tilde{N}$. Since the $1$-form dual to $\xi$ is $\phi_5=\frac{1}{\sqrt{2}}(\varphi_3+\overline{\varphi}_3)$, we have
$$ d\phi_5 = \sqrt{2}\,d\varphi_3 = \sqrt{2}(\varphi_{2\bar{1}}- \varphi_{2\bar{1}}).$$
So $\phi_5 \wedge (d\phi_5)^2= 2\sqrt{2} (\varphi_3+\overline{\varphi}_3) \varphi_{1\bar{1}2\bar{2}} $ which is nowhere zero. The vector field $\xi$ is clearly a Killing field as $I$ is orthogonal with respect to $g$. By definition, this means that $\xi$ is the Reeb vector field and $(\tilde{N},g,\xi,\phi_5)$ is a Sasakian $5$-manifold.

Then we will prove the last statement in the middle type case of Theorem \ref{classification}.

\begin{claim}
Let $(M^3,g,J)$ be a compact non-K\"ahler threefold, which is balanced BTP of middle type. Then $(M^3,J)$ does not admit any pluriclosed metric.
\end{claim}

\begin{proof}
Assume on the contrary that there is a pluriclosed metric $\tilde{g}$ on $(M^3,J)$. Its K\"ahler form $\tilde{\omega}$ satisfies $\partial \overline{\partial}\tilde{\omega}=0$. Locally let $e$ be a special frame and $\varphi$ its dual coframe. As we have seen before,  $e_3$ and $\varphi_3$ are determined up to a sign, so in particular, $\varphi_3\overline{\varphi}_3$ is globally defined on $M$. Since $d\varphi_3=\varphi_{2\bar{1}}-\varphi_{1\bar{2}}=d\overline{\varphi}_3$, we know that
 $\ d\varphi_3 \, d\overline{\varphi}_3 = 2\varphi_{1\bar{1}2\bar{2}}\ $ is a global $(2,2)$-form on $M$. We have
\begin{eqnarray*}
 d(d\varphi_3\,\overline{\varphi}_3\,\tilde{\omega}) & = & d\varphi_3 \, d\overline{\varphi}_3 \,\tilde{\omega} -  d\varphi_3 \, \overline{\varphi}_3 \,\partial \tilde{\omega} \\
d(\varphi_3\,\overline{\varphi}_3\,\partial \tilde{\omega}) & = & d\varphi_3 \, \overline{\varphi}_3 \,\partial \tilde{\omega} -  \varphi_3 \, d\overline{\varphi}_3 \,\partial \tilde{\omega} + \varphi_3\,\overline{\varphi}_3\,\overline{\partial} \partial \tilde{\omega}
\end{eqnarray*}
Note that the middle term on the right hand side of the second equality is zero by type consideration. So if we add up these two equalities and integrate it over $M$, we would get
$$ 0 = \int_M d\varphi_3 \, d\overline{\varphi}_3 \,\tilde{\omega} = \int_M 2 \varphi_{1\bar{1}2\bar{2}} \,\tilde{\omega} = \int_M 2 \sqrt{-1} \,\tilde{g}_{3\bar{3}} \,\varphi_{1\bar{1}2\bar{2}3\bar{3}} = - \frac{1}{3} \int_M  \tilde{g}_{3\bar{3}}\,\omega^3 <0, $$
a contradiction. Here we have written  $\,\tilde{\omega} = \sqrt{-1} \sum_{i,j=1}^3 \tilde{g}_{ij} \varphi_i \wedge \overline{\varphi}_j\,$ under the frame $e$, so the matrix $(\tilde{g}_{i\bar{j}})$ is positive definite. The contradiction means that $M^3$ cannot admit any pluriclosed metric. This completes the proof of the claim.
\end{proof}

On the balanced {\em BTP} threefold $(M^3,g,J)$, we know from Theorem \ref{theorem1.1} the {\em BTP} condition indicates $R^b_{ijk\overline{\ell}}=0$ and $R^b_{i\overline{j}k\overline{\ell}}=R^b_{k\overline{\ell}i\overline{j}}$. Then it follows from (\ref{eq:9.2}) that the $(0,2)$ and $(2,0)$-parts of $d\alpha$ and $d\beta_0$ are zero, and there are real-valued local smooth functions $x$, $y$ and $z$ such that
\begin{eqnarray} \label{eq:9.5}
d\alpha & = & x (\varphi_{1\overline{1}} + \varphi_{2\overline{2}}) + \sqrt{-1}y (\varphi_{2\overline{1}} - \varphi_{1\overline{2}} ), \\
d\beta_0 & = &  -\sqrt{-1}y  (\varphi_{1\overline{1}} + \varphi_{2\overline{2}}) +z (\varphi_{2\overline{1}} - \varphi_{1\overline{2}} ). \label{eq:9.6}
\end{eqnarray}
It yields from (\ref{eq:9.3}) that $x=z+2$, hence
\begin{equation} \label{eq:9.7}
d\beta =  -\sqrt{-1}y  (\varphi_{1\overline{1}} + \varphi_{2\overline{2}}) +x (\varphi_{2\overline{1}} - \varphi_{1\overline{2}} ).
\end{equation}

Note that the Ricci form of the Chern curvature of $\omega_{J}$ is given by $\mbox{Ric}(\omega_{J})=2d\alpha$, which indicates $x=\frac{1}{4}\mbox{Scal}(\omega_{J})$ where $\mbox{Scal}(\omega_{J})$ is the Chern scalar curvature, hence $x$ is a global function. On the other hand, since $\varphi_3$ is determined up to a sign, so $\varphi_{3\bar{3}}$, hence $\sqrt{-1}( \varphi_{1\overline{1}} + \varphi_{2\overline{2}})=\omega_{J} - \sqrt{-1}\varphi_{3\bar{3}}$, are globally defined on $M^3$. The latter can be regarded as the K\"ahler form of $J\big|_{\xi^{\bot}}$. As a consequence, $y\,d\varphi_3 = -\sqrt{-1}(d\alpha -x(\varphi_{1\bar{1}}+ \varphi_{2\bar{2}}))$ is globally defined on $M$. This means that $y$ and $d\varphi_3= \varphi_{2\bar{1}}-\varphi_{1\bar{2}}$ are determined up to a sign. When $\varphi_3$ is not globally defined on $M$, by lifting to $\hat{M}$, we know that both $y$ and $ \varphi_{2\bar{1}}-\varphi_{1\bar{2}}$ are  globally defined. On $\hat{M}^3$, $\varphi_{2\overline{1}} - \varphi_{1\overline{2}}$ is the K\"ahler form of $I \big|_{\xi^{\bot}}$, while $y=\frac{1}{2}\mbox{tr}_{\omega_J}\,\sqrt{-1}d \beta$. Since $\varphi_{1\overline{1}} + \varphi_{2\overline{2}}$ and $\varphi_{2\overline{1}} - \varphi_{1\overline{2}}$ are $d$-closed, by taking the exterior differentiation of (\ref{eq:9.5}) or (\ref{eq:9.7}), we get
\begin{equation}  \label{eq:9.9}
x_3=y_3=0, \ \ \ \ \ x_1=-\sqrt{-1}y_2, \ \ \ \ \ x_2=\sqrt{-1}y_1.
\end{equation}
Here $x_i=e_i(x)$, $y_i=e_i(y)$. We suspect that $x$ and $y$ must be constants, but at this point we do not know how to prove it from the above system of equations, where the usual Bochner technique does not seem to apply, and more geometric information might be needed.

\vspace{0.3cm}

\subsection{Examples of Lie-Hermitian threefolds}\label{LH}
In this subsection, we will deal with the special case when the balanced {\em BTP} threefold $(M^3,g)$ of the middle type is assumed to be a {\em Lie-Hermitian} manifold, meaning that its universal cover is a (connected and simply-connected) Lie group $G$ equipped with a left-invariant complex structure and a compatible left-invariant metric.

\begin{proof}[\bf{Proof of Theorem \ref{LH_BBTP}}]
Let us  start from a global, unitary, left-invariant frame $e$ on $G$. An appropriate constant unitary change of the frame enables us to assume that $e$ is special. Under our assumption that $g$ is balanced {\em BTP} of the middle type, the formulae for connection matrices $\theta^b$ and $\theta$ are all valid.

Since $e$ is left-invariant, the $1$-form $\alpha$ and $\beta$ are linear combinations of $\varphi$ and $\overline{\varphi}$ with constant coefficients:
\begin{equation*}
\alpha = \sum_i \big( a_i\varphi_i - \overline{a}_i \overline{\varphi}_i \big) , \ \ \ \ \ \beta  = \sum_i \big( b_i\varphi_i + \overline{b}_i \overline{\varphi}_i \big).
\end{equation*}
The structure equation (\ref{eq:9.1}) implies
\begin{eqnarray*}
(d\alpha)^{2,0} & = & (a_1b_1+a_2b_2)\varphi_{12} + (a_1a_3+a_2b_3)\varphi_{13} + (a_2a_3-a_1b_3)\varphi_{23} \\
(d\beta)^{2,0} & = & (a_2b_1-a_1b_2 + b_1^2 + b_2^2)\varphi_{12} + (a_3b_1+b_2 b_3)\varphi_{13} + (a_3b_2-b_1 b_3)\varphi_{23} \\
(d\alpha)^{1,1} & = & 2(\mbox{Re}(a_2\overline{b}_1) -|a_1|^2)\varphi_{1\overline{1}} - 2(\mbox{Re}(a_1\overline{b}_2)+|a_2|^2)\varphi_{2\overline{2}} + P\varphi_{2\overline{1}}+ \overline{P}\varphi_{1\overline{2}} +\\
& & + (a_2\overline{b}_3 - a_1\overline{a}_3)\varphi_{1\overline{3}} - (a_2\overline{a}_3 + a_1\overline{b}_3)\varphi_{2\overline{3}} + (\overline{a}_2b_3 - \overline{a}_1a_3)\varphi_{3\overline{1}} - (\overline{a}_2a_3 + \overline{a}_1b_3)\varphi_{3\overline{2}} \\
(d\beta)^{1,1} & = & 2\sqrt{-1} \,\mbox{Im}(a_1\overline{b}_1- b_1\overline{b}_2) \varphi_{1\overline{1}} + 2\sqrt{-1} \,\mbox{Im}(a_2\overline{b}_2- b_1\overline{b}_2)\varphi_{2\overline{2}} + Q\varphi_{2\overline{1}}-\overline{ Q}\varphi_{1\overline{2}} +\\
& & + (b_2\overline{b}_3 - b_1\overline{a}_3)\varphi_{1\overline{3}} - (b_2\overline{a}_3 + b_1\overline{b}_3)\varphi_{2\overline{3}} - (\overline{b}_2b_3 - \overline{b}_1a_3)\varphi_{3\overline{1}} + (\overline{b}_2a_3 + \overline{b}_1b_3)\varphi_{3\overline{2}}
\end{eqnarray*}
where
\begin{eqnarray*}
P & = & a_3 - \overline{a}_3 - a_1\overline{b}_1- 2a_2 \overline{a}_1  +b_2\overline{a}_2  \\
Q & = & b_3 + \overline{b}_3 - |b_1|^2 - |b_2|^2 + a_2\overline{b}_1- b_2 \overline{a}_1.
\end{eqnarray*}
Then equalities (\ref{eq:9.5}) and (\ref{eq:9.7}) indicate
\begin{eqnarray}
\label{eq:9.10} && a_1b_1+a_2b_2\,=\,0, \ \ \ \ \    a_2b_1-a_1b_2 + b_1^2 + b_2^2 \,=\,0,  \\
\label{eq:9.11} &&  a_1 a_3+a_2 b_3 \,=\, -a_1 b_3+a_2 a_3 \,=\, 0 ,\\
\label{eq:9.12} &&  a_3b_1+b_2 b_3 \,=\, -b_1 b_3+b_2 a_3 \,=\, 0 ,\\
\label{eq:9.13} && x \,=\, 2\mbox{Re}(a_2\overline{b}_1) - 2 |a_1|^2 \,=\, -2\mbox{Re}(a_1\overline{b}_2)-2|a_2|^2 \,=\, Q, \\
\label{eq:9.14} && -y \,=\, \sqrt{-1}\,P \,=\, 2\mbox{Im}(a_1\overline{b}_1- b_1\overline{b}_2) \,=\, 2\mbox{Im}(a_2\overline{b}_2- b_1\overline{b}_2).
\end{eqnarray}

\vspace{0.15cm}

\noindent {\bf Case A: } $a_3^2+b_3^2 \neq 0$.

\vspace{0.15cm}

In this case, by (\ref{eq:9.11}) and (\ref{eq:9.12}), we get $a_1=a_2=0 $ and $b_1=b_2=0$. By (\ref{eq:9.13}) and (\ref{eq:9.14}) we get $x=y=a_3-\overline{a}_3=b_3+\overline{b}_3=0$. So in this case we have $\alpha =a_3(\varphi_3 - \overline{\varphi}_3)$, $\beta =b_3(\varphi_3 - \overline{\varphi}_3)$, where $a_3\in {\mathbb R}$ and $b_3 \in \sqrt{-1}{\mathbb R}$, hence it yields that
\begin{equation*}
\left\{ \begin{array}{lll} d\varphi_1 & = & (a_3\varphi_1 -b_3\varphi_2) \wedge (\varphi_3 - \overline{\varphi}_3 ), \\
d\varphi_2 & = & (b_3\varphi_1 +a_3\varphi_2) \wedge (\varphi_3 - \overline{\varphi}_3 ), \\
d\varphi_3 & = & \varphi_{2\overline{1}} - \varphi_{1\overline{2}}.
\end{array} \right.
\end{equation*}
By a unitary change of $(\varphi_1, \varphi_2, \varphi_3)$ into $( \frac{1}{\sqrt{2}} (\varphi_1 +\sqrt{-1}\varphi_2), \frac{1}{\sqrt{2}} (\sqrt{-1}\varphi_1 +\varphi_2),\varphi_3)$, the above is equivalent to the following
\begin{equation}
A_{a,b}: \ \ \
\left\{ \begin{array}{lll} d\varphi_1 & = & a\,\varphi_1  \wedge (\varphi_3 - \overline{\varphi}_3 ), \\
d\varphi_2 & = & b\,\varphi_2 \wedge (\varphi_3 - \overline{\varphi}_3 ), \\
d\varphi_3 & = & \sqrt{-1}\,(\varphi_{2\overline{2}} - \varphi_{1\overline{1}}).
\end{array} \right.
\end{equation}
Here $a, b\in {\mathbb R}$ are given by $a=a_3+\sqrt{-1}b_3$, $b=a_3-\sqrt{-1}b_3$. We will denote this Lie-Hermitian threefold by $A_{a,b}$, where $a,b$ are arbitrary real numbers. We have $x=y=0$ in this case, so the Ricci form $\mbox{Ric}(\omega_{J})=0$. Since $d(\varphi_{123})=-2\alpha \wedge \varphi_{123}$, we know that $A_{a,b}$ will have an invariant nowhere-zero holomorphic $3$-form (hence any quotient will have trivial canonical line bundle) if and only if $a_3=0$, or equivalently, if $a=-b$. When $a=b=0$, $A_{0,0}$ is clearly isometric to the nilmanifold $N^3$ given by (\ref{N_3}).

\vspace{0.15cm}

\noindent {\bf Case B: } $a^2_3+b_3^2 = 0$, where $(a_3,b_3) \neq 0$.

\vspace{0.15cm}
In this cases, it follows that $a_3 =\varepsilon i b_3$, where $i= \sqrt{-1}$, $\varepsilon=\pm 1$ and $b_3 \neq 0$. By equalities (\ref{eq:9.10}) through (\ref{eq:9.14}), we get $a_1 = \varepsilon i a_2$, $b_1 = \varepsilon i b_2$, $b_3 + \overline{b}_3 = 2 (|b_2|^2-|a_2|^2)$, and
\[x= 2\varepsilon \mbox{Im}(a_2\overline{b}_2) - 2|a_2|^2, \quad y= -2 \mbox{Im}(a_2 \overline{b}_2) +2 \varepsilon |b_2|^2. \]
Let $u=a_2-\varepsilon i b_2$, $v =a_2+\varepsilon i b_2$ and $w = 2\varepsilon i b_3$. By a unitary change of $(\varphi_1, \varphi_2, \varphi_3)$ into $( \frac{1}{\sqrt{2}} (\varphi_1 + \varepsilon i\varphi_2), \frac{1}{\sqrt{2}} (\varepsilon i \varphi_1 +\varphi_2),\varphi_3)$, we get the Lie Hermitian manifold
\begin{equation}
B_{u,v,w}^{\varepsilon}: \ \ \
\left\{ \begin{array}{lll} d\varphi_1 & = & \sqrt{2}v \,\varphi_{12} - \sqrt{2}\overline{v}\,\varphi_{1\bar{2}}
+ w\,\varphi_{13} - \overline{w}\,\varphi_{1\bar{3}}, \\
d\varphi_2 & = & - \sqrt{2}\overline{u}\,\varphi_{2\bar{2}}, \\
d\varphi_3 & = & \varepsilon \sqrt{-1}\,(\varphi_{2\bar{2}} - \varphi_{1\bar{1}}).
\end{array} \right.
\end{equation}
It is easy to see that $u,v,w$ satisfy $w-\overline{w}= \varepsilon i (|u-v|^2 - |u+v|^2)$, and
\[ x= \frac{1}{2}(|v|^2-|u|^2 - |u+v|^2), \quad y = \frac{\varepsilon}{2}(|u|^2 + |u-v|^2 - |v|^2). \]
It is clear that $B_{0,0,0}^{\varepsilon}$ is isomorphic to the nilmanifold $N^3$ given by \eqref{N_3}.

\vspace{0.2cm}

Next let us deal with the remaining case $a_3=b_3=0$. If $b_1=b_2=0$, then $Q=0$, hence by (\ref{eq:9.13}) we get $x=0$ and $a_1=a_2=0$. This gives us $\alpha=\beta=0$ and we end up with the nilmanifold $N^3$. Hence we may assume $(b_1,b_2)\neq (0,0)$, which splits into two cases below, depending on whether $b^2_1+b_2^2$ is zero or not.

\vspace{0.15cm}

\noindent {\bf Case C:} $(a_3,b_3)=(0,0)$ and $\ b_1^2+b_2^2 \neq 0$.

\vspace{0.15cm}

In this case, the equality (\ref{eq:9.10}) gives us
$$ a_1=b_2, \ \ \ a_2=-b_1.$$
Hence, if we let $u=-\overline{b}_2$ and $v=-\overline{b}_1$, which are complex numbers satisfying $|u|^2+|v|^2>0$, we can get the Lie-Hermitian manifold
\begin{equation}
C_{u,v}: \ \ \
\left\{ \begin{array}{lll} d\varphi_1 & = &\  \, u (\varphi_{1\overline{1}} + \varphi_{2\overline{2}} ) + v (\varphi_{2\overline{1}} - \varphi_{1\overline{2}}),  \\
d\varphi_2 & = & -v (\varphi_{1\overline{1}} + \varphi_{2\overline{2}}) + u (\varphi_{2\overline{1}} - \varphi_{1\overline{2}}), \\
d\varphi_3 & = & \ \varphi_{2\overline{1}} - \varphi_{1\overline{2}}.
\end{array} \right.
\end{equation}
Then it yields that $x=-2|u|^2-2|v|^2<0$ and $y=4\mbox{Im}(u\overline{v})$. In particular, the Chern scalar curvature $x$ is a negative constant. The complex structure is abelian. Clearly, $C_{0,0}$ is isometric to $N^3$.

\vspace{0.15cm}

\noindent {\bf Case D: } $(a_3,b_3)=(0,0)$ and $b_1^2+b_2^2=0$, where $(b_1,b_2)\neq (0,0)$.

\vspace{0.15cm}

This means that $a_3=b_3=0$ and $b_2=\varepsilon i b_1\neq 0$, where $i=\sqrt{-1}$ and $\varepsilon =\pm 1$. By equalities (\ref{eq:9.10}) through (\ref{eq:9.14}), we get $a_2=\varepsilon i a_1$, $a_1=\rho b_1$ for some $|\rho |=1$, and
$$ x = -2|b_1|^2(1+\varepsilon \mbox{Im}(\rho )), \quad y=\varepsilon x. $$
Let $u=\overline{b}_1$ and we get a Lie-Hermitian manifold
\begin{equation}
D_{u,\rho}^{\varepsilon }: \ \ \
\left\{ \begin{array}{lll} d\varphi_1 & = &\  \, \overline{u} (1+\rho \varepsilon i) \varphi_{12} + u \{  (-\overline{\rho } \varphi_{1\overline{1}} + \varepsilon i \varphi_{2\overline{2}} ) + ( - \varphi_{2\overline{1}} + \overline{\rho } \varepsilon i \varphi_{1\overline{2}}) \} , \\
d\varphi_2 & = &\  \, \overline{u} \varepsilon i(1+\rho \varepsilon i) \varphi_{12} + u \{  ( \varphi_{1\overline{1}} + \overline{\rho }\varepsilon i \varphi_{2\overline{2}} ) + ( -\overline{\rho } \varphi_{2\overline{1}} -  \varepsilon i \varphi_{1\overline{2}}) \}  ,\\
d\varphi_3 & = & \ \varphi_{2\overline{1}} - \varphi_{1\overline{2}},
\end{array} \right.
\end{equation}
where $u, \rho \in {\mathbb C}$, $|\rho |=1$, and $\varepsilon =\pm 1$. The complex structure of $D_{u,\rho }^{\varepsilon}$ is non-abelian except when $\rho =\varepsilon i$, in which case we have $D^{\varepsilon}_{u,\varepsilon i} = C_{u\varepsilon i, -u}$. The Chern scalar curvature $x$ of $D_{u,\rho }^{\varepsilon}$ is always non-positive. When $u=0$, we get the complex nilmanifold $N^3$ again.

From the discussion above, we know any Lie-Hermitian threefold which is balanced {\em BTP} of middle type, it must be a member of one of these four families. On the other hand, it is easy to check that each of the above four families of Lie-Hermitian manifolds are balanced and {\em BTP}. This completes the proof of the theorem.
\end{proof}

\vspace{0.3cm}

\subsection{Generalization to higher dimensions}

We have seen in the rank$B=2$ case of Theorem \ref{classification} that balanced {\em BTP} threefolds $(M^3,g, J)$ admits a covering $\hat{M}$ of degree at most $2$, such that on $(\hat{M},g)$ there exists another orthogonal complex structure $I$, so that $(\hat{M},g, I)$ is Vaisman. The universal cover $\tilde{M}$ splits as a Riemannian product $\tilde{N}^5\times {\mathbb R}$ where $\tilde{N}^5$ is a Sasakian $5$-manifold. The complex structures $I$ and $J$ are also closely related and satisfies the condition $IJ=JI$. Therefore, the classification of compact balanced {\em BTP} threefolds of middle type amounts to understanding this special types Vaisman threefolds, or equivalently, this special types of Sasakian $5$-manifolds with a `bi-Hermitian' structure, namely, it admits two commutative orthogonal complex structures simultaneously.

In this subsection, we want to fully understand this special `bi-Hermitian' structure. It turns out the phenomenon persists in higher dimensions as well, which will give us examples of balanced {\em BTP} manifolds in all dimensions $n\geq 3$. The $5$-dimensional manifold $\tilde{N}^5$ that we encountered before motivates us with the following definition.

\begin{definition}
The manifold $(N^{2m+1},g,\xi,J)$ is said to be an \textit{abelian Sasakian} manifold, if $(N^{2m+1},g,\xi)$ is a Sasakian manifold as defined in Definition \ref{SSK}, which admits an endomorphism $J$ of the tangent bundle $TN$, such that
\begin{enumerate}
\item The Killing vector field $\xi$ is in the kernel of $J$.
\item The endomorphism $J$ defines an integrable orthogonal complex structure on the distribution $H=\xi^{\bot}$ and the K\"ahler form $\omega_J$, defined as $\omega_J(X,Y) = g(JX,Y)$, is $d$-closed.
\item The space $T^{1,0}_{I}$ is $J$-invariant and the eigenspaces of $\sqrt{-1}$ and $-\sqrt{-1}$ of $J\big|_{T^{1,0}_{I}}$ have the same dimension, denoted by $n$,
where the space $T^{1,0}_{I}$ is induced from the Sasakian structure
 \[TN \otimes_{\mathbb{R}} \mathbb{C} = \langle \xi \rangle \oplus T^{1,0}_{I} \oplus T^{0,1}_{I},\quad I = \frac{1}{c}\nabla \xi.\]
\end{enumerate}
Clearly we must have $m=2n$ here, so the dimension of an abelian Sasakian manifold is necessarily $4n+1$, and the two complex structures $I$ and $J$ on the distribution $H=\xi^{\bot}$ satisfy \[I \circ J = J \circ I,\] which justifies the name of abelian Sasakian.
\end{definition}

\begin{proposition}\label{eq_AB-SSK}
Let $(N^{2m+1},g,\xi)$ be a Sasakian manifold and $J$ be an endomorphism of the tangent bundle $TN$, which is an integrable orthogonal complex structure on the distribution $\xi^{\bot}$ and satisfies $J(\xi)=0$. Then $J$ defines an abelian Sasakian structure on $(N^{2m+1},g,\xi)$ if and only if the K\"ahler form $\omega_{I}$ of $I = \frac{1}{c} \nabla \xi$ decomposes into two $d$-closed $I$-semi-positive $2$-forms $\omega_1$ and $\omega_2$, such that the space $T^{1,0}_{I}$ admits a $g$-orthogonal decomposition of two kernel distributions of $\omega_1$ and $\omega_2$ with the same rank $n$
\[T^{1,0}_{I} = \mbox{Ker}(\omega_1) \oplus \mbox{Ker}(\omega_2).\]
Here an $I$-semi-positive $2$-form $\omega$ on $N$ is defined to satisfy $\omega(X,Y)=\omega(IX,IY)$ for any $X,Y \in TN$ and $\omega(Z,\overline{Z}) \geq 0$ for any $Z \in T^{1,0}_{I}$, and the kernel distribution $\mbox{Ker}(\omega)$ of an $I$-semi-positive $2$-form $\omega$ is defined as $\{Z \in T^{1,0}_{I} \big| \omega(Z,\overline{Z}) =0\}$.
\end{proposition}

\begin{proof}
Let $(\xi,v,\overline{v})$ be a frame of the Sasakian manifold $(N^{2m+1}, g, \xi)$ such that $v$ is a unitary frame of the distribution
$\xi^{\bot}$, which yields the following structure equation
\[ \nabla \! \begin{bmatrix} \xi \\ v \\ \overline{v} \end{bmatrix}=
\begin{bmatrix} 0 & - ^{t}\!\nu & - ^{t}\!\overline{\nu} \\
\overline{\nu} & K & \overline{L}  \\
\nu  & L & \overline{K}  \\ \end{bmatrix}\!\!
\begin{bmatrix} \xi \\ v \\ \overline{v} \\ \end{bmatrix}\!\!,\]
where $K$ is a skew Hermitian matrix of $1$-forms and $L$ is a skew symmetric matrix of $1$-forms. The dual frame of $(\xi,v,\overline{v})$ is denoted by $(\phi,\zeta,\overline{\zeta})$. As
$I=\frac{1}{c} \nabla \xi$ induces an integrable orthogonal complex structure on the distribution $\xi^{\bot}$, it yields that
$\nu = -c \sqrt{-1} \zeta$. The dual version of the equation above is
\begin{equation}\label{SSK-dual}
\nabla \! \begin{bmatrix} \phi \\ \zeta \\ \overline{\zeta} \end{bmatrix}=
\begin{bmatrix} 0 & - ^{t}\!\overline{\nu} & - ^{t}\!\nu \\
\nu & \overline{K} & L  \\
\overline{\nu}  & \overline{L} & K  \\ \end{bmatrix}\!\!
\begin{bmatrix} \phi \\ \zeta \\ \overline{\zeta} \\ \end{bmatrix}\!\!.
\end{equation}
As $\nabla$ is torsion free, it follows that
\[ d \! \begin{bmatrix} \phi \\ \zeta \\ \overline{\zeta} \end{bmatrix}=
\begin{bmatrix} 0 & - ^{t}\!\overline{\nu} & - ^{t}\!\nu \\
\nu & \overline{K} & L \\
\overline{\nu}  & \overline{L} & K  \\ \end{bmatrix}\!\!
\begin{bmatrix} \phi \\ \zeta \\ \overline{\zeta} \\ \end{bmatrix}\!\!,\]
which implies that $d \phi = 2 c \sqrt{-1}\, ^{t}\!\zeta\,\overline{\zeta} $. Note that $I$ is integrable on $\xi^{\bot}$ and thus $d \zeta_i$ has no component of $\phi \wedge \overline{\zeta}_j$ and $\overline{\zeta}_k \wedge \overline{\zeta}_{\ell}$, which indicates $L$ is a matrix of $(1,0)$-forms of the complex structure $I$. Then it is clear that
\[ 0= d ( ^{t}\!\zeta \, \overline{\zeta}) = -(^{t}\!\overline{\zeta}\, ^{t}\!L\, \overline{\zeta} \,+\,  ^{t}\!\zeta\, \overline{L}\, \zeta),\]
which yields that $^{t}\!\zeta\, \overline{L}\, \zeta=0$ by the type of $L$. Since $L$ is skew symmetric, it follows that $L=0$.

If $J$ defines an abelian Sasakian structure on $(N, g, \xi)$, we may assume that the basis $v$ of $T^{1,0}_{I}$ splits into two column vectors, where the former one is still denoted by $v$ and the latter one by $v_{+n}$, both of which contains $n$ entries, such that
\[J v = \sqrt{-1} v, \quad J v_{+n} = -\sqrt{-1} v_{+n}.\]
This forces that $\zeta$ splits into two parts as $\begin{bmatrix} \zeta \\ \zeta^{+n} \end{bmatrix}$ and the matrix $K$ also splits into four submatrices of the same size, denoted by
\[K = \begin{bmatrix} A & C \\
                      - ^{t}\!\overline{C} & D \end{bmatrix}\!\!.\]
Then a unitary frame of $J$ on the distribution $\xi^{\bot}$ can be defined as
\[e=v,\quad e_{+n} = \overline{v}_{+n}.\]
The dual frame of $(\xi, e, e_{+n}, \overline{e}, \overline{e}_{+n})$ is denoted by $(\phi, \varphi, \varphi^{+n}, \overline{\varphi}, \overline{\varphi}^{+n})$, which satisfies
\[ \varphi = \zeta,\quad \varphi^{+n} = \overline{\zeta}^{+n}.\]
Then the equation \eqref{SSK-dual} can be reformulated, in terms of $(\phi, \varphi, \varphi^{+n}, \overline{\varphi}, \overline{\varphi}^{+n})$, as
\begin{equation}\label{AB-SSK-dual}
\nabla \! \begin{bmatrix} \phi \\ \varphi \\ \varphi^{+n} \\ \overline{\varphi} \\ \overline{\varphi}^{+n} \end{bmatrix}=
\begin{bmatrix} 0 & - ^{t}\!\overline{\mu} & - ^{t}\!\overline{\mu}^{+n} & - ^{t}\!\mu & - ^{t}\!\mu^{+n} \\
\mu & \overline{A} & 0 & 0 & \overline{C}\\
\mu^{+n} & 0 & D & - ^{t}\!\overline{C} & 0 \\
\overline{\mu} & 0 & C & A & 0 \\
\overline{\mu}^{+n}  & - ^{t}\!C & 0 & 0 &\overline{D} \\ \end{bmatrix}\!\!
\begin{bmatrix} \phi \\ \varphi \\ \varphi^{+n} \\ \overline{\varphi} \\ \overline{\varphi}^{+n} \end{bmatrix}\!\!,
\end{equation}
where $\mu = -c \sqrt{-1} \varphi, \mu^{+n}= c\sqrt{-1} \varphi^{+n}$. As $\nabla$ is torsion free, it yields that
\begin{equation}\label{AB-SSK-d}
d \! \begin{bmatrix} \phi \\ \varphi \\ \varphi^{+n} \\ \overline{\varphi} \\ \overline{\varphi}^{+n} \end{bmatrix}=
\begin{bmatrix} 0 & - ^{t}\!\overline{\mu} & - ^{t}\!\overline{\mu}^{+n} & - ^{t}\!\mu & - ^{t}\!\mu^{+n} \\
\mu & \overline{A} & 0 & 0 & \overline{C}\\
\mu^{+n} & 0 & D & - ^{t}\!\overline{C} & 0 \\
\overline{\mu} & 0 & C & A & 0 \\
\overline{\mu}^{+n}  & - ^{t}\!C & 0 & 0 &\overline{D} \\ \end{bmatrix}\!\!
\begin{bmatrix} \phi \\ \varphi \\ \varphi^{+n} \\ \overline{\varphi} \\ \overline{\varphi}^{+n} \end{bmatrix}\!\!,
\end{equation}
which implies that $d \phi = 2c \sqrt{-1} ( ^{t}\!\varphi \overline{\varphi} \,-\, ^{t}\!\varphi^{+n} \overline{\varphi}^{+n})$. As the K\"ahler form $\omega_{J}= \sqrt{-1}( ^{t}\!\varphi \overline{\varphi} \,+\, ^{t}\!\varphi^{+n} \overline{\varphi}^{+n})$ is necessarily $d$-closed by the definition of the abelian Sasakian structure, it follows that
\begin{equation}\label{d-closed}
d ( ^{t}\!\varphi \overline{\varphi})=d( ^{t}\!\varphi^{+n} \overline{\varphi}^{+n})=0.
\end{equation}
Note that $J$ is an orthogonal integrable complex structure on the distribution $\xi^{\bot}$, which indicates $d \varphi_i$ and $d \varphi_{i+n}$ has no component of $\phi \wedge \overline{\varphi}_j$, $\phi \wedge \overline{\varphi}_{j+n}$, $\overline{\varphi}_k \wedge \overline{\varphi}_{\ell}$, $\overline{\varphi}_k \wedge \overline{\varphi}_{\ell+n}$ and $\overline{\varphi}_{k+n} \wedge \overline{\varphi}_{\ell+n}$, for $1 \leq i,j,k,\ell \leq n$. Hence, $C$ is a matrix of $(0,1)$-forms of the complex structure $J$. Then
the equality \eqref{d-closed} implies that
\[ ^{t}\!\varphi\, C\, \varphi^{+n} \,=\, ^{t}\overline{\varphi}^{+n}\, ^{t}\!\overline{C}\, \overline{\varphi} =0. \]
Therefore $C=0$.

Note that the K\"ahler form $\omega_{I} $ of $I$ is $\sqrt{-1}( ^{t}\!\zeta \overline{\zeta} \,+\, ^{t}\!\zeta^{+n}\, \overline{\zeta}^{+n})$ and the K\"ahler form $\omega_{J} $ of $J$ is $\sqrt{-1}( ^{t}\!\varphi \overline{\varphi} \,+\, ^{t}\!\varphi^{+n} \overline{\varphi}^{+n}) = \sqrt{-1}( ^{t}\!\zeta \overline{\zeta} \,-\, ^{t}\!\zeta^{+n} \,\overline{\zeta}^{+n})$, from which we may decompose $\omega_{I}$ into
\[ \omega_{I} = \omega_1 + \omega_2,\quad \text{for}\quad \omega_1 = \sqrt{-1}\, ^{t}\!\zeta \overline{\zeta} ,\quad \omega_2 = \sqrt{-1}\, ^{t}\!\zeta^{+n} \,\overline{\zeta}^{+n}.\]
It is clear that $\omega_1, \omega_2$ are $d$-closed $I$-semi-positive $2$-forms on $N$ and $\mbox{Ker}(\omega_2),\mbox{Ker}(\omega_1)$ are distributions generated by $v, v_{+n}$, respectively, which are the eigenspaces of $J$ with respect to $\sqrt{-1},-\sqrt{-1}$. Then it follows that $T^{1,0}_{I}$ admits a $g$-orthogonal decomposition into $\mbox{Ker}(\omega_1) \oplus \mbox{Ker}(\omega_2)$.

Conversely, when the K\"ahler form $\omega_{I}$ of $I = \frac{1}{c} \nabla \xi$ decomposes into two $d$-closed $I$-semi-positive $2$-forms $\omega_1$ and $\omega_2$, such that the space $T^{1,0}_{I}$ admits a $g$-orthogonal decomposition with the same rank $n$
\[T^{1,0}_{I} = \mbox{Ker}(\omega_1) \oplus \mbox{Ker}(\omega_2).\]
We may define $J$ by designating $\mbox{Ker}(\omega_2)$ and $\mbox{Ker}(\omega_1)$ are the eigenspaces of $J$ with respect to the eigenvalues $\sqrt{-1}$ and $-\sqrt{-1}$ respectively. It is clear that $T^{1,0}_{I}$ is $J$-invariant and $J$ defines an abelian Sasakian structure.
\end{proof}

\begin{definition}
Let $(N_1^{4n_1+1}, g_1, \xi_1,J_1)$ and $(N_2^{4n_2+1}, g_2, \xi_2,J_2)$ be two abelian Sasakian manifolds. Motivated by twisted Sasakian product in Definition \ref{SSKproduct}, we may also define such twisted product for abelian Sasakian manfolds. On the product Riemannian manifold $M=N_1\times N_2$, of even dimension $2n=2(2n_1+2n_2+1)$, for $\kappa \in \{\kappa \in \mathbb{C} \big| \mathrm{Im} (\kappa) \neq 0\}$, we can define the natural almost complex structure $J_{\kappa}$ by
\begin{gather*}
J_{\kappa}\xi_1 = \mathrm{Re}(\kappa) \xi_1 + \mathrm{Im}(\kappa)\xi_2, \\
J_{\kappa}X_i = J_i X_i, \quad \forall \ X_i \in H_i , \ i=1,2.
\end{gather*}
Define the associated metric $g_{\kappa}$ as
\begin{gather*}
g_{\kappa}(\xi_1,\xi_1)= g_{\kappa}(J_{\kappa}\xi_1,J_{\kappa}\xi_1)=1, \quad g_{\kappa}(\xi_1,J_{\kappa}\xi_1)=0,\\
g_{\kappa}(\xi_i,X)=0,\quad \forall \ X \in H_1 \oplus H_2, \ i=1,2.\\
g_{\kappa}(X_i,Y_i)=g_i(X_i,Y_i), \quad g_{\kappa}(X_1,X_2)=0,\quad \forall \ X_i, Y_i \in H_i,\ i=1,2.
\end{gather*}
It is clear that $g_{\kappa}$ is a $J_{\kappa}$-Hermitian metric. As shown in the proof of the following proposition, $J_{\kappa}$ is integrable, and when $\kappa=\sqrt{-1}$, the Hermitian manifold $(N_1\times N_2, g_{\kappa}=g_1\times g_2, J_{\kappa})$ is the standard Hermitian structure on the product of two abelian Sasakian manifolds. We will call the Hermitian manifold $(N_1\times N_2, g_{\kappa} J_{\kappa})$ twisted abelian Sasakian product, denoted by $N_1 \times_{\kappa} N_2$.
\end{definition}

Abelian Sasakian manifolds arise naturally in the study of balanced {\em BTP} manifolds.

\begin{proposition}
Let $(N_1^{4n_1+1}, g_1, \xi_1,J_1)$ and $(N_2^{4n_2+1}, g_2, \xi_2,J_2)$ be two abelian Sasakian manifolds. Then twisted abelian Sasakian product $N_1\times_{\kappa} N_2$ is a non-K\"ahler manifold, which is balanced and BTP, and the rank of the $B$ tensor is $n-1$, for the complex dimension $n \geq 3$.
\end{proposition}

In the smallest possible dimension $n=2n_1 + 2n_2 +1=3$, the example above is always a Riemannian product, which is essentially $\tilde{N}^5 \times \mathbb{R}$ above, as $n_1$ or $n_2$ has to be zero. When the dimension increases, it may not be a Riemannian product for $\mathrm{Re}(\kappa) \neq 0$.

\begin{proof}
We consider the Hermitian structure $(g_{\kappa},J_{\kappa})$ on the product $N_1 \times N_2$. As notations applied in the proof of Proposition \ref{eq_AB-SSK} above, we may assume that $(\xi_1, e, e_{+n_1}, \overline{e}, \overline{e}_{+n_1})$ is a frame on $(N_1,g_1,\xi_1,J_1)$, where $e, e_{+n_1}$ is a unitary frame of $\xi_1^{\bot}$ endowed with the complex structure $J_1$, and the dual frame is denoted by $(\phi_1, \varphi, \varphi^{+n_1}, \overline{\varphi}, \overline{\varphi}^{+n_1})$, which satisfies the equation \eqref{AB-SSK-d}. And $(\xi_2, \hat{e}, \hat{e}_{+n_2}, \overline{\hat{e}}, \overline{\hat{e}}_{+n_2})$ is a frame on $(N_2,g_2,\xi_2,J_2)$, where $\hat{e}, \hat{e}_{+n_2}$ is a unitary frame of $\xi_2^{\bot}$ with the complex structure $J_2$, and the dual frame is denoted by $(\phi_2, \hat{\varphi}, \hat{\varphi}^{+n_2}, \overline{\hat{\varphi}}, \overline{\hat{\varphi}}^{+n_2})$.

For simplicity, we will write $\kappa=x + y\sqrt{-1}$, where $x,y \in \mathbb{R},\ y \neq 0$. Let
\[\begin{split}
e_0 &= \frac{1}{\sqrt{2}}(\xi_1 - \sqrt{-1}J_{\kappa} \xi_1)
= \frac{1}{\sqrt{2}}\Big( (1-x\sqrt{-1})\xi_1 - y\sqrt{-1} \xi_2\Big),\\
\varphi_0 &= \frac{1}{\sqrt{2}}\left(\phi_1 - \frac{x}{y}\phi_2 + \sqrt{-1}\frac{1}{y}\phi_2\right)
= \frac{1}{\sqrt{2}} \left( \phi_1 - \frac{x-\sqrt{-1}}{y} \phi_2 \right), \\
\end{split}\]
where $e_0$ is a $(1,0)$-vector and $\varphi_0$ is a $(1,0)$-form of of the almost complex structure $J_{\kappa}$. It is easy to verify
that $F=\,^{t}\!(e_0, e, e_{+n_1}, \hat{e}, \hat{e}_{+n_2})$ is a unitary frame of $J_{\kappa}$ on $N_1 \times N_2$ and $\Phi=\, ^{t}\!(\varphi_0, \varphi, \varphi^{+n_1}, \hat{\varphi}, \hat{\varphi}^{+n_2})$ is the dual frame, where $F$ and $\Phi$ are regarded as column vectors. The equation \eqref{AB-SSK-d} and the one on $(N_2, g_2, \xi_2,J_2)$ imply that
\[ \begin{split}
d \varphi_0 &= \frac{1}{\sqrt{2}}
\left( 2c_1 \sqrt{-1} \big(\, ^{t}\!\varphi \overline{\varphi} \,-\, ^{t}\! \varphi^{+n_1} \overline{\varphi}^{+n_1} \big)
- \frac{2c_2(1+x\sqrt{-1})}{y}
\big(\, ^{t}\!\hat{\varphi} \overline{\hat{\varphi}} \,-\, ^{t}\! \hat{\varphi}^{+n_2} \overline{\hat{\varphi}}^{+n_2} \big) \right),\\
d\! \begin{bmatrix}\varphi \\ \varphi^{+n_1} \end{bmatrix} &=
\frac{c_1}{\sqrt{2}}\big((x+\sqrt{-1})\varphi_0 - (x-\sqrt{-1})\overline{\varphi}_0\big)\! \begin{bmatrix} \varphi \\ -\varphi^{+n_1} \end{bmatrix}
+ \begin{bmatrix} \overline{A} & 0 \\ 0 & D \end{bmatrix} \!\! \begin{bmatrix} \varphi \\ \varphi^{+n_1} \end{bmatrix}, \\
d\! \begin{bmatrix} \hat{\varphi} \\ \hat{\varphi}^{+n_2} \end{bmatrix} &=
\frac{c_2}{\sqrt{2}}\big( y\varphi_0 - y \overline{\varphi}_0 \big)\! \begin{bmatrix} \hat{\varphi} \\ -\hat{\varphi}^{+n_2} \end{bmatrix}
+ \begin{bmatrix} \overline{\hat{A}} & 0 \\ 0 & \hat{D} \end{bmatrix} \!\! \begin{bmatrix} \hat{\varphi} \\ \hat{\varphi}^{+n_2} \end{bmatrix}, \\
\end{split}\]
where the matrices $\hat{A}$ and $\hat{D}$ come from the analogous equation on $(N_2, g_2, \xi_2,J_2)$ to the one \eqref{AB-SSK-d}. It implies that $J_{\kappa}$ is indeed a integrable complex structure, as $d \varphi_0$, $d \varphi$, $d \varphi^{+n_1}$, $d \hat{\varphi}$ and $d \hat{\varphi}^{+n_2}$ has no component of $(0,2)$-forms. The Riemannian connection $\nabla$ under the unitary frame $\Phi$ can be denoted by
\[ \nabla \Phi = \overline{\theta}_1 \Phi + \theta_2 \overline{\Phi},\]
where $\theta_1$ is skew Hermitian and $\theta_2$ is skew symmetric. As $\nabla$ is torsion free, it yields
\[d \Phi =  \overline{\theta}_1 \Phi + \theta_2 \overline{\Phi},\]
which indicates
\[\theta_1 = \begin{bmatrix}
0 & -\,^{t}\!\overline{\mu}_1 & -\,^{t}\!\overline{\mu}_2 & - \, ^{t}\!\overline{\mu}_3 & -\,^{t}\!\overline{\mu}_4 \\
\mu_1 & H_1 & 0 & 0 & 0 \\
\mu_2 & 0 & H_2 & 0 & 0 \\
\mu_3 & 0 & 0 & H_3 & 0 \\
\mu_4 & 0 & 0 & 0 & H_4
\end{bmatrix}\!\!,\quad
\theta_2 = \begin{bmatrix}
0 & -\,^{t}\!\mu_5 & -\,^{t}\!\mu_6 & - \, ^{t}\! \mu_7 & -\,^{t}\!\mu_8 \\
\mu_5 & 0 & 0 & 0 & 0 \\
\mu_6 & 0 & 0 & 0 & 0 \\
\mu_7 & 0 & 0 & 0 & 0 \\
\mu_8 & 0 & 0 & 0 & 0
\end{bmatrix}\!\!,\]
where
\begin{align*}
\mu_1 &= \frac{c_1\sqrt{-1}}{\sqrt{2}} \overline{\varphi},& \mu_2 &= -\frac{c_1\sqrt{-1}}{\sqrt{2}} \overline{\varphi}^{+n_1},&
\mu_3 &= -\frac{c_2(1+x\sqrt{-1})}{y\sqrt{2}} \overline{\hat{\varphi}},&
\mu_4 &= \frac{c_2(1+x\sqrt{-1})}{y\sqrt{2}} \overline{\hat{\varphi}}^{+n_2},\\
\mu_5 &= - \frac{c_1\sqrt{-1}}{\sqrt{2}} \varphi,& \mu_6 &= \frac{c_1\sqrt{-1}}{\sqrt{2}} \varphi^{+n_1},&
\mu_7 &= \frac{c_2(1+x\sqrt{-1})}{y\sqrt{2}} \hat{\varphi},&
\mu_8 &= -\frac{c_2(1+x\sqrt{-1})}{y\sqrt{2}} \hat{\varphi}^{+n_2},
\end{align*}
\begin{gather*}
H_1 = A - \frac{c_1 x}{\sqrt{2}}(\varphi_0 - \overline{\varphi}_0)E_{n_1},\quad
H_2 = \overline{D} + \frac{c_1 x}{\sqrt{2}}(\varphi_0 - \overline{\varphi}_0 )E_{n_1},\\
H_3 = \hat{A} - \frac{c_2}{\sqrt{2}} \Big((y-\frac{1-x\sqrt{-1}}{y})\varphi_0 - (y-\frac{1+x\sqrt{-1}}{y})\overline{\varphi}_0 \Big)E_{n_2}, \\
H_4 = \overline{\hat{D}} + \frac{c_2}{\sqrt{2}} \Big((y-\frac{1-x\sqrt{-1}}{y})\varphi_0 - (y-\frac{1+x\sqrt{-1}}{y})\overline{\varphi}_0 \Big)E_{n_2}.
\end{gather*}
Here $E_{n_1}$ and $E_{n_2}$ are the identity matrices of the sizes $n_1$ and $n_2$. Then the matrix $\theta_2$ determines the components of Chern torsion of the Hermitian metric $(g_{\kappa},J_{\kappa})$, which implies, for $1 \leq i,j\leq n_1$, $2n_1+1 \leq \kappa,\iota \leq 2n_1+n_2$, $0 \leq p,q \leq 2(n_1 +n_2)$,
\begin{align*}
T_{0\,i}^p &= -\frac{c_1 \sqrt{-1}}{\sqrt{2}} \delta_{ip}, &  T^{p}_{0\,i+n_1} &= \frac{c_1 \sqrt{-1}}{\sqrt{2}} \delta_{i+n_1\,p}, \\
T^p_{0\,\kappa}&= -\frac{c_2(1-x\sqrt{-1})}{y\sqrt{2}} \delta_{\kappa\, p},& T^{p}_{0\,\kappa+n_2} &= \frac{c_2(1-x\sqrt{-1})}{y\sqrt{2}} \delta_{\kappa+n_2\,p},
\end{align*}
with other components all vanishing, such as
\begin{align*}
T^p_{ij}&=0, & T^p_{i\,j+n_1}&=0, & T_{i\,\kappa}^p&=0, & T^p_{i\,\kappa+n_2}&=0, & T^p_{i+n_1\,j+n_1}&=0,\\
T^p_{i+n_1\,\kappa}&=0, & T^p_{i_1+n_1\,\kappa+n_2}&=0, & T^p_{\kappa\,\iota}&=0, & T^p_{\kappa\,\iota+n_2}&=0, & T^p_{\kappa+n_2\, \iota+n_2}&=0.
\end{align*}
From the definition of the Gauduchon's torsion $1$-form $\eta$, it yields that
\[\eta = \eta_0\varphi_0 + \sum_i \eta_i \varphi_i + \sum_i \eta_{i+n_1} \varphi_{i+n_1} +
\sum_{\kappa} \eta_{\kappa} \hat{\varphi}_{\kappa} + \sum_{\kappa} \eta_{\kappa+n_2} \hat{\varphi}_{\kappa+n_2}, \]
where
\begin{gather*}
\eta_0 = \sum_p T^p_{p0} = \sum_i (T_{i\,0}^i + T^{i+n_1}_{i+n_1\,0}) + \sum_{\kappa}(T^{\kappa}_{\kappa\,0} + T^{\kappa+n_2}_{\kappa+n_2\,0}) =0 \\
\eta_i=0,\quad \eta_{i+n_1}=0,\quad \eta_{\kappa}=0,\quad \eta_{\kappa+n_2}=0,
\end{gather*}
and thus $g_{\kappa}$ is balanced.

Then the matrix $\gamma$ be calculated by the components of Chern torsion as
\[ \gamma = \begin{bmatrix}
0 & \,^{t}\!\overline{\mu}_1 & \,^{t}\!\overline{\mu}_2 & \, ^{t}\!\overline{\mu}_3 & \,^{t}\!\overline{\mu}_4 \\
-\mu_1 & H_5 & 0 & 0 & 0 \\
-\mu_2 & 0 & H_6 & 0 & 0 \\
-\mu_3 & 0 & 0 & H_7 & 0 \\
-\mu_4 & 0 & 0 & 0 & H_8
\end{bmatrix}\!\!,\]
where $\mu_1, \mu_2, \mu_3, \mu_4$ are the entries in the matrix $\theta_1$ and
\begin{gather*}
H_5 = \frac{c_1 \sqrt{-1}}{\sqrt{2}}(\varphi_0 + \overline{\varphi}_0) E_{n_1} \quad
H_6 = -\frac{c_1 \sqrt{-1}}{\sqrt{2}}(\varphi_0 + \overline{\varphi}_0) E_{n_1} \\
H_7 = \frac{c_2}{\sqrt{2}}\big(\frac{1-x\sqrt{-1}}{y} \varphi_0 - \frac{1+x\sqrt{-1}}{y}\overline{\varphi}_0\big)E_{n_2} \\
H_8 = -\frac{c_2}{\sqrt{2}}\big(\frac{1-x\sqrt{-1}}{y} \varphi_0 - \frac{1+x\sqrt{-1}}{y}\overline{\varphi}_0\big)E_{n_2}
\end{gather*}
Note that the Bismut connection matrix $\theta^b = \theta_1 + \gamma$ and we have
\[ \theta^b = \begin{bmatrix}
0 & 0 & 0 & 0 & 0 \\
0 & H_1+H_5 & 0 & 0 & 0 \\
0 & 0 & H_2+H_6 & 0 & 0 \\
0 & 0 & 0 & H_3+H_7 & 0 \\
0 & 0 & 0 & 0 & H_4+H_8
\end{bmatrix}.\]
By the equalities $d \Phi = - \,^{t}\! \theta^b \Phi + \tau^b$ and $d \Phi = \overline{\theta}_1 \Phi + \theta_2 \overline{\Phi}$,
it yields that \[\tau^b = \,^{t}\!\gamma \Phi + \theta_2 \overline{\Phi} = \begin{bmatrix}
\nu_1 \\
\nu_2 \\
\nu_3 \\
\nu_4 \\
\nu_5
\end{bmatrix}\!\!,\]
where
\begin{gather*}
\nu_1 = \frac{2c_1 \sqrt{-1}}{\sqrt{2}}\big(\, ^{t}\!\varphi \overline{\varphi} \,-\, ^{t}\! \varphi^{+n_1} \overline{\varphi}^{+n_1} \big)
- \frac{2c_2(1+x\sqrt{-1})}{y\sqrt{2}}
\big(\, ^{t}\!\hat{\varphi} \overline{\hat{\varphi}} \,-\, ^{t}\! \hat{\varphi}^{+n_2} \overline{\hat{\varphi}}^{+n_2} \big) \\
\nu_2 =\frac{2c_1\sqrt{-1}}{\sqrt{2}}(\varphi_0 + \overline{\varphi}_0) \varphi,  \quad
\nu_3 = - \frac{2c_1\sqrt{-1}}{\sqrt{2}}(\varphi_0 + \overline{\varphi}_0) \varphi^{+n_1}, \\
\nu_4 = \big( \frac{2c_2(1-x\sqrt{-1})}{y\sqrt{2}}\varphi_0 - \frac{2c_2(1+x\sqrt{-1})}{y\sqrt{2}} \overline{\varphi}_0 \big) \hat{\varphi}, \quad
\nu_5 = -\big( \frac{2c_2(1-x\sqrt{-1})}{y\sqrt{2}}\varphi_0 - \frac{2c_2(1+x\sqrt{-1})}{y\sqrt{2}} \overline{\varphi}_0 \big) \hat{\varphi}^{+n_2}.
\end{gather*}
By the expression of $\theta^b$, it is easy to verify that
\[\nabla^b e_0 =0,\quad \nabla^b \varphi_0=0,\]
\begin{align*}
\nabla^b \,^{t}\!\varphi \overline{\varphi} &=0, & \nabla^b\, ^{t}\!\varphi^{+n_1} \overline{\varphi}^{+n_1}&=0,
& \nabla^b \,^{t}\!\hat{\varphi} \overline{\hat{\varphi}} &=0, & \nabla^b \, ^{t}\!\hat{\varphi}^{+n_2} \overline{\hat{\varphi}}^{+n_2}&=0, \\
\nabla^b \,^{t}\!\varphi \otimes e &=0, & \nabla^b\, ^{t}\!\varphi^{+n_1} \otimes e_{+n_1}&=0,
& \nabla^b \,^{t}\!\hat{\varphi} \otimes \hat{e} &=0, & \nabla^b \, ^{t}\!\hat{\varphi}^{+n_2} \otimes \hat{e}_{+n_2}&=0. \\
\end{align*}
Therefore,
\[\nabla^b T^b = \nabla^b (\, ^{t}\!\tau^b \otimes F \,+ \,^{t}\!\overline{\tau}^b \otimes \overline{F}) =0,\]
which establishes that $g_{\kappa}$ is {\em BTP}. It is easy to verify the tensor $B$ under the frame $F$ is
\[ \begin{bmatrix}
0 & 0 & 0 & 0 & 0 \\
0 & H_9 & 0 & 0 & 0 \\
0 & 0 & H_{10} & 0 & 0 \\
0 & 0 & 0 & H_{11} & 0 \\
0 & 0 & 0 & 0 & H_{12}
\end{bmatrix}\!\!,\]
where
\[H_9=H_{10}=2\left|\frac{c_1\sqrt{-1}}{\sqrt{2}}\right|^2 E_{n_1},
\quad H_{11} = H_{12}= 2\left|\frac{c_2(1-x\sqrt{-1})}{y\sqrt{2}}\right|^2 E_{n_2},\]
and thus $\mathrm{rank}B=2(n_1+n_2)=n-1$.
\end{proof}

\vspace{1cm}
%\textbf{Acknowledgement}: We would like to thank mathematicians Gabriel Khan, Lei Ni, Bo Yang,

\end{document}